\documentclass[reqno,11pt]{amsart}
\usepackage{cite}
\usepackage{amssymb}
\usepackage{amsthm}
\usepackage[mathscr]{euscript}
\usepackage[OT1]{fontenc}
\usepackage[latin1]{inputenc}
\usepackage{epsf}            
\usepackage{hyperref}
\usepackage[usenames]{color}
\usepackage[all]{xy}
\CompileMatrices
\usepackage{epic}
\usepackage{eepic}
 \usepackage{enumerate}
 \usepackage[active]{srcltx}
  \usepackage{graphicx}
  \usepackage{ae,aecompl}
\usepackage{comment}
\usepackage{appendix}

\numberwithin{equation}{section}

\DeclareMathAlphabet{\realcal}{U}{rsfs}{m}{n}
\def\Cl#1{\ensuremath{\realcal #1}}
\renewcommand{\to}{\longrightarrow}

\def\ci#1{\ensuremath{{\mathcal {#1}}}}
\def\z#1{\ensuremath{{{#1}}^{\ZZ}}}

\def\KK{\ensuremath{\mathbb K}}

\def\dia{\diamond}

\newcommand{\NN}{\mathbb N}
\def\NN{\ensuremath{\mathbb N}}

\def\ZZ{\ensuremath{\mathbb Z}}

\def\pv#1{\ensuremath{{\sf#1}}}

\def\D{\ensuremath{\ci D}}
\def\J{\ensuremath{\ci J}}
\def\R{\ensuremath{\ci R}}
\renewcommand{\L}{\ensuremath{\ci L}}
\renewcommand{\H}{\ensuremath{\ci H}}

\newtheorem{Thm}{Theorem}[section]

\newtheorem{Prop}[Thm]{Proposition}
\newtheorem{Lemma}[Thm]{Lemma}
\newtheorem{Cor}[Thm]{Corollary}
{\theoremstyle{remark}

}
{\theoremstyle{remark}
\newtheorem{Example}[Thm]{Example}}
{\theoremstyle{definition}
}
{\theoremstyle{definition}
\newtheorem{Def}[Thm]{Definition}}

{
  \theoremstyle{remark}
\newtheorem{Rmk}[Thm]{Remark}
}

\newcommand{\inv}{^{-1}}

\title{A categorical invariant of flow equivalence of shifts}
\thanks{The first author was supported by the Centro
  de Matem\'atica da Universidade de Coimbra (CMUC),
  funded by the European Regional Development Fund through
  the program COMPETE and by the Portuguese Government
  through the Funda\c {c}\~ao para a Ci\^encia e a
Tecnologia (FCT)
under the project
PEst-C/MAT/UI0324/2011. He was also supported by the FCT post-doctoral grant
  SFRH/BPD/46415/2008 and by the FCT
  project PTDC/MAT/65481/2006, within the framework of European programmes
 COMPETE and FEDER.  Some of this work was performed while the second author was at the School of Mathematics and Statistics, Carleton University, Ottawa, Canada under the auspices of an NSERC grant. This work was partially supported by a grant from the Simons Foundation (\#245268 to Benjamin Steinberg)  and the Binational Science
Foundation of Israel and the US (\#2012080 to Benjamin Steinberg)}
\author{Alfredo Costa}
\address{CMUC, Department of Mathematics, University of Coimbra,
  3001-454 Coimbra, Portugal.}
\email{amgc@mat.uc.pt}
\author{Benjamin Steinberg}
\address{Department of Mathematics\\
City College of New York\\
NAC 8/133\\
Convent Ave at 138th Street\\
New York, NY 10031}
\email{bsteinberg@ccny.cuny.edu}
\date{April 11, 2013; revised \today}

\begin{document}
\begin{abstract}
  We prove that the Karoubi envelope of a shift --- defined
   as the Karoubi envelope of the syntactic semigroup
   of the language of blocks of the shift ---
   is, up to natural equivalence of categories,
   an invariant of flow equivalence. More precisely, we show
   that the action of the Karoubi envelope on the Krieger cover of the
   shift is a flow invariant. An analogous result
   concerning the Fischer cover of a synchronizing shift is also obtained.
    From these main results, several flow equivalence invariants --- some new and some old ---
    are obtained.
    We also show that
    the Karoubi envelope is, in a natural sense,
    the best possible syntactic invariant
    of flow equivalence of sofic shifts.
   
   Another application concerns the classification
   of Markov-Dyck and Markov-Motzkin shifts:
   it is shown that, under mild conditions, two graphs define flow equivalent
   shifts if and only if they are isomorphic.

    Shifts with property ($\mathscr A$) and their associated
   semigroups,  introduced by Wolfgang
   Krieger, are interpreted in terms of the Karoubi envelope,
   yielding a proof
   of the  flow invariance of the associated semigroups
   in the cases usually considered (a result recently
   announced by Krieger),
   and also a proof that
   property ($\mathscr A$) is decidable for sofic shifts.
\end{abstract}

  \makeatletter
  \@namedef{subjclassname@2010}{%
    \textup{2010} Mathematics Subject Classification}
  \makeatother

\maketitle

\tableofcontents

\section{Introduction}

Two discrete-time dynamical systems are
\emph{flow equivalent} if their
suspension flows (or mapping tori) are equivalent.
For symbolic dynamical systems,
Parry and Sullivan
characterized flow equivalence
as the equivalence relation between shifts generated by
conjugacy and a non-symmetric relation which at present
is called \emph{symbol expansion}~\cite{Parry&Sullivan:1975};
see also~\cite[Section 13.7]{MarcusandLind} and
\cite{Beal&Berstel&Eilers&Perrin:2010arxiv,Bates&Eilers&Pask:2011}.
Special attention has been given to the classification of shifts of finite type
up to flow equivalence; in that context,
complete and decidable algebraic invariants were obtained
in the irreducible case~\cite{Franks:1984}. Complete algebraic invariants
were also obtained in the reducible case~\cite{Huang:2001,Boyle:2002}.
However, as pointed out in~\cite{Bates&Eilers&Pask:2011},
only small progress have been made in
the strictly sofic case, where
one finds few useful flow equivalence
invariants, even for irreducible shifts.

The role of the syntactic semigroup of the language of
finite blocks of a shift $\Cl X$,
which we call the \emph{syntactic semigroup of $\Cl X$},
has been considered in the
literature~\cite{Beauquier:1985,Jonoska:1996a,Jonoska:1998,Beal&Fiorenzi&Perrin:2005b,Beal&Fiorenzi&Perrin:2005a,Costa:2006,Costa:2006b}
essentially in the context of (strictly) sofic shifts.
 In~\cite{Costa&Steinberg:2011}, one finds a
 characterization of the abstract semigroups which are
 the syntactic semigroup of an irreducible sofic shift.

For a semigroup $S$, let $\KK(S)$ be
the Karoubi envelope (also known as Cauchy completion or idempotent
splitting) of $S$. It is a certain small category that plays a crucial
role in finite semigroup
theory thanks in part to the Delay Theorem of Tilson (see~\cite{Tilson:1987},
where $\KK(S)$ is denoted by $S_E$.)

For a shift $\Cl X$, let $S(\Cl X)$
be its syntactic semigroup.
In this paper we prove that the Karoubi envelope $\KK(S (\Cl X))$ (more briefly denoted  $\KK(\Cl X)$) is, up to equivalence of categories, a flow
equivalence invariant of $\Cl X$. This
says, in a sense to be made more precise, that $\Cl X$ determines $S(\Cl
X)$ up to Morita equivalence. We also show that among sofic shifts
this is the best possible syntactic invariant of flow equivalence, as
every flow equivalence invariant of sofic shifts
which is also invariant under isomorphism of syntactic semigroups is shown
to factor through the equivalence relation identifying shifts with equivalent
Karoubi envelopes.

The category $\KK(\Cl X)$ is of little use when dealing with
shifts of finite type. Indeed,
a shift of finite
type is conjugate with an edge shift, and it is
easy to see that if
$\Cl X$ is an irreducible edge shift then $S(\Cl X)$
is isomorphic to a Brandt semigroup $B_n$ for some
$n$~\cite[Remark 2.23]{Costa:2007}.  Unfortunately, all finite Brandt semigroups have equivalent Karoubi envelopes.
However, for other classes, including strictly sofic shifts,
we do obtain interesting results.

We also investigate
the actions of the category $\KK(\Cl X)$ on
the Krieger cover of $\Cl X$ and, if $\Cl X$
is synchronizing, on its Fischer cover. We show that these actions are
invariant under flow equivalence.
As seen in Section~\ref{sec:prop-comun-graph},
this enables a new proof of the invariance under flow
equivalence of the proper communication graph of a sofic
shift (a result from~\cite{Bates&Eilers&Pask:2011}).

In Section~\ref{sec:label-preord-set}, we provide
examples of pairs of almost finite type
shifts $\Cl X$ and $\Cl Y$
such that $\KK(\Cl X)$ and $\KK(\Cl Y)$ are not equivalent,
whereas other
flow equivalence invariants
fail to separate them.
In the first of these examples we use a labeled poset, derived
from Green's relations on $S(\Cl X)$, which was shown
in~\cite{Costa:2006} to be
a conjugacy invariant of sofic shifts.  This is a more refined version of an invariant
previously considered in~\cite{Beal&Fiorenzi&Perrin:2005a}.

In Section~\ref{sec:classes-sofic-shifts} we see further examples
of how the Karoubi envelope can be used to show that
some classes of shifts
are closed under flow equivalence.
In particular we deduce again that the class of
 almost finite type shifts is stable under flow equivalence,
 a result from~\cite{Fujiwara&Osikawa:1987}.

The Karoubi envelope is applied in
Section~\ref{sec:markov-dyck-shifts} to classify up to flow
equivalence two classes of non-sofic shifts,
the Markov-Dyck and the Markov-Motzkin shifts of Krieger and
Matsumoto~\cite{Krieger&Matsumoto:2011},
thus reproving and generalizing the
classification of Dyck shifts, first obtained
in~\cite{Matsumoto:2011}.

The Karoubi envelope is also used in Section~\ref{sec:shifts-with-kriegers}
to give a new perspective
on the class of shifts with property
($\mathscr A$), introduced by Wolfgang Krieger
in~\cite{Krieger:2000}, and also studied in~\cite{Krieger:2012,Hamachi&Krieger:2013,Hamachi&Krieger:2013b}.
The semigroup which Krieger associated
to each shift with property ($\mathscr A$)
is shown to be encapsulated in the Karoubi envelope.
From that result one deduces the invariance,
recently announced by Krieger, of such semigroup under flow equivalence,
assuming the density of a special subset of the shift,
an assumption which is usually made when studying property ($\mathscr A$)
shifts.
The interpretation in terms of the Karoubi envelope is also used to
show that property  ($\mathscr A$) is decidable for sofic shifts and to construct the first examples of sofic shifts without propert ($\mathscr A$).

 The poset of subsynchronizing subshifts of a sofic shift
 was shown in~\cite{Jonoska:1998} to be a conjugacy invariant of sofic
 shifts, with the aim of studying the
 structure of reducible sofic shifts.
 In Section~\ref{sec:subsynchr-subsh-sofi} this invariance is
 recovered and extended to flow equivalence, using the Karoubi envelope.

The basic outline of the paper is as follows.
We begin with
two sections of  preliminaries.
The first contains some aspects of
 category theory and semigroup theory that we shall require in order to extract flow equivalence invariants from
the category $\KK(\Cl X)$.
The second preliminary section is about symbolic dynamics.
Section~\ref{sec:stat-main-results}
states our main results.
In Sections~\ref{sec:prop-comun-graph}
to~\ref{sec:subsynchr-subsh-sofi}, consequences of our main results are explored, as well as relations
 with previous work, from the viewpoint of the classification of
 shifts up to flow equivalence.
 Although the principal motivation in this paper
is the study of flow equivalence, we also deduce
in Section~\ref{sec:eventual-conjugacy} that
the Karoubi envelope is an invariant of
eventual conjugacy in the case of sofic shifts.
The proofs of the main results are left
to Sections~\ref{sec:proof-theor-reft:m}
and~\ref{sec:master}.

 There is also a short appendix completing
 an argument in the proof
 of Theorem~\ref{t:flow-inv-sub-sync}, using
 technical tools from Section~\ref{sec:proof-theor-reft:m}.

\section{The Karoubi envelope of a semigroup}

\subsection{Categorical preliminaries}

The reader is referred
to~\cite{MacLane:1998} for basic notions from
category theory.  A \emph{category} $C$ consists of a class $C_0$ of \emph{objects}, a class $C_1$ of \emph{arrows} or \emph{morphisms} and mappings $d,r\colon C_1\to C_0$ selecting the domain $d(f)$ and range $r(f)$ of each arrow $f$.  In addition, there is an associative product on pairs of composable arrows $(f,g)\mapsto fg$, where composable means that $d(f)=r(g)$, and for each object $c\in C_0$, there is an identity arrow $1_c$ so that $1_cf=f$ and $g1_c=g$ when these compositions make sense.
A category is said to be \emph{small} if its objects and
arrows form a set. The set of all arrows $f\colon c\to d$ is denoted $C(c,d)$ and is called a~\emph{hom-set}.

A \emph{functor} $F\colon C\to D$ consists of a pair of mappings
$F\colon C_0\to D_0$ and $F\colon C_1\to D_1$ preserving all the above
structure, e.g., $F(1_c)=1_{F(c)}$ and $F(fg)=F(f)F(g)$, etc.  A
\emph{natural transformation} $\eta\colon F\Rightarrow G$ of functors
$F,G\colon C\to D$ is a family $\{\eta_c\}_{c\in C_0}$ of arrows of
$D$ such that $\eta_c\colon F(c)\to G(c)$ and, for each arrow $f\colon
c\to c'$ of $C$, the diagram
\[\xymatrix{F(c)\ar[r]^{F(f)}\ar[d]_{\eta_c} &
  F(c')\ar[d]^{\eta_{c'}}\\ G(c)\ar[r]_{G(f)} & G(c')}\]
commutes. The
class of natural transformations $F\Rightarrow G$ forms a category in
the obvious way and two functors are \emph{isomorphic}, written
$F\cong G$ if they are isomorphic in this category.  Contravariant functors are defined similarly except that $F(fg)=F(g)F(f)$.

 Two categories $C$ and $D$ are
\emph{equivalent} if there are functors $F\colon C\to D$ and $G\colon
D\to C$ such that $FG\cong 1_D$ and $GF\cong 1_C$.
Such a functor $G$ is said to be a \emph{quasi-inverse} of $F$.
A functor $F\colon
C\to D$ between small categories is an equivalence (i.e., it
has a quasi-inverse) if and only if it is fully faithful and essentially
surjective.  \emph{Fully faithful} means bijective on hom-sets,
whereas \emph{essentially surjective} means that every object of $D$
is isomorphic to an object of $F(C)$.
The former is in accordance with the usual terminology
for functors which are injective on hom-sets (the \emph{faithful} functors)
and for those surjective on hom-sets (the \emph{full} functors.)

\subsection{The Karoubi envelope and Morita equivalence of semigroups}
An important notion in this paper is the \emph{Karoubi envelope}
$\KK(S)$ (also known as \emph{Cauchy completion} or \emph{idempotent splitting}) of a
semigroup $S$.  It is a small category whose object set is the set
$E(S)$ of idempotents of $S$.
Morphisms in $\KK(S)$ from $f$ to $e$ are represented by arrows
$e\longleftarrow f$, with source on the
right.  All other categories will be treated as usual with arrows
drawn from left to right. This is to keep our notation for arrows of $\KK(S)$ consistent with~\cite{Tilson:1987}.
A  morphism $e\longleftarrow f$ is a triple
$(e,s,f)$ where $s\in eSf$. Note that
$s\in eSf$ if and only if $s=esf$, because $e$ and $f$ are idempotents.
Composition of morphisms is given
by
$(e,s,f)(f,t,g)=(e,st,g)$.
The identity at $e$ is $(e,e,e)$.  Note that the Karoubi envelope is
functorial.

It is easy to show that idempotents $e,f$ are isomorphic in $\KK(S)$ if and
only if there exist $x,x'\in S$ such that $xx'x=x$, $x'xx'=x'$,
$x'x=e$ and $xx'=f$.  In semigroup terms (which will be explained in more detail
in Subsection~\ref{sec:greens-relat-relat}), this says that $e,f$ are $\D$-equivalent~\cite{Rhodes&Steinberg:2009}, whereas in analytic terms this corresponds to von Neumann-Murray equivalence.

If $e$ is an idempotent of the semigroup $S$, then $eSe$ is
a monoid with identity~$e$, which is called the \emph{local monoid of
  $S$ at $e$}.
The \emph{local monoid of a category $C$ at an object $c$}
is the endomorphism monoid of $c$ in $C$.
The local monoids of $S$ correspond to the local monoids
of $\KK(S)$, more precisely, $eSe$ and $\KK(S)(e,e)$ are isomorphic for
every $e\in E(S)$.

An element $s$ of a semigroup $S$ has \emph{local units $e$ and $f$}, where $e$
and $f$ are idempotents of $S$, if
$s=esf$.
The set ${LU}(S)=E(S)SE(S)$ of elements of $S$ with
local units is a subsemigroup of $S$.
If ${LU}(S)=S$, then
we say that $S$ has \emph{local units}.
In general, ${LU}(S)$ is the largest subsemigroup of $S$ which has
local units (it may be empty.)  Clearly $\KK(S)=\KK({LU}(S))$ and so the
Karoubi envelope does not distinguish between these two semigroups.
Talwar defined in~\cite{Talwar3} a notion of Morita equivalence of
semigroups with local units in terms of equivalence of certain
categories of actions.  It was shown
in~\cite{Lawson:2011,FunkLawsonSteinberg} that semigroups
$S$ and $T$ with local units are Morita equivalent if and only if
$\KK(S)$ and $\KK(T)$ are equivalent categories.  Thus we shall say that
semigroups $S$ and $T$ are \emph{Morita equivalent up to local units}
if $\KK(S)$ and $\KK(T)$ are equivalent categories, or in other words if
$LU(S)$ is Morita equivalent to $LU(T)$. In this paper,
we will show that flow equivalent shifts have syntactic semigroups
that are Morita equivalent up to
local units.

If $C$ is a category, then we shall say an assignment $S\mapsto F(S)$ of an object $F(S)$ of $C$ to each semigroup $S$ is a \emph{Karoubi invariant} if $F(S)\cong F(T)$ whenever $\KK(S)$ and $\KK(T)$ are equivalent.

\subsection{Categories with zero}

A \emph{semigroup with zero} is a semigroup $S$ with an element $0$
such that $0s=0=s0$ for all $s\in S$.  The element $0$ is unique.

A \emph{pointed set} $(X,x)$ is a set $X$ together with a distinguished element $x\in X$ called the \emph{base point}.  A morphism of pointed sets $f\colon (X,x)\to (Y,y)$ is a function $f\colon X\to Y$ with $f(x)=y$.  We will customarily denote the base point by $0$ if confusion cannot arise. The category of pointed sets will be denoted $\mathbf{Set}_0$.

A \emph{category with zero} is a category $C$ enriched over the
monoidal category of pointed sets.  What this means concretely is
that, for all objects $c,d$ of $C$, there is a zero morphism
$0_{c,d}\in C(c,d)$
such that
for all $f\colon c'\to c$ and $g\colon d\to d'$ one has
$0_{c,d}f = 0_{c',d}$ and $g0_{c,d} = 0_{c,d'}$.
It is easy to check that the zero morphisms
are uniquely determined and so,
from now on, we will drop the subscripts on zero morphisms
when convenient.

The category of pointed sets is an example of a category with $0$, where the zero
map sends each element to the base point.
The most important example for us is the case where $S$ is a semigroup with zero and $C$ is the Karoubi envelope $\KK(S)$.  Then $0_{e,f}=(f,0,e)$ is the zero morphism of $\KK(S)(e,f)$. An object $c$ of a category with zero is said to be \emph{trivial} if
$1_c=0_c$.  Notice then that the only morphisms into and out of a
trivial object are zero morphisms.

Note that
if $F$ is a full functor between categories with zero, then $F(0)=0$.
This enables us to register the following remark, for later reference.

\begin{Rmk}\label{r:fully-faithful-are-strict}
  If $F$ is a fully faithful functor between categories with zero
  then $F(x)=0$ if and only if $x=0$.
\end{Rmk}

In general, a morphism of categories
with zero should be defined as a
functor preserving zero morphisms.

\subsection{Actions of $\KK(S)$}\label{sec:actions}

Recall that a (right) 
\emph{action} of a semigroup $S$ on a set $Q$
is a function $\mu\colon Q\times S\to S$,
with notation $\mu(q,s)=q\cdot s=qs$,
such that $q\cdot (st)=(q\cdot s)\cdot t$.
The action defines a function $\varphi$ from $S$ to the monoid
$Q^Q$ of transformations of~$Q$, given
by $\varphi(s)(q)=q\cdot s$. An action of a semigroup $S$ with zero on a \emph{pointed set} $(Q,0)$ is an action such that $0s=0$ and $q0=0$ for all $s\in S$ and $q\in Q$. Actions on pointed sets are essentially the same thing as actions by partial functions. Morphisms between actions on pointed sets are defined in the obvious way.

A (right)
\emph{action} of a small category $C$ with zero on a pointed set $Q$
is a contravariant
functor $\mathbb A\colon C\to \mathbf{Set}_0$, preserving $0$,
such that $\mathbb A(c)$ is a pointed subset of $Q$, for every object
$c$ of $C$. If $s\colon c\to d$ is a morphism of $C$,
we use the notation
$q\cdot s$ for $\mathbb A(s)(q)$, where $q\in \mathbb A(d)$,
and the notation $\mathbb A(d)\cdot s$
for the image of the function $\mathbb A(s)\colon \mathbb A(d)\to
\mathbb A(c)$. The notations $q\cdot s$ and $\mathbb A(d)\cdot s$
may be simplified to $qs$ and $\mathbb A(d)s$.

In concrete terms this means that $\mathbb A(c)$ is a subset of $Q$ containing $0$ for each object of $c\in C$; if $s\colon c\to d$ is an arrow then we have a map $\mathbb A(d)\to \mathbb A(c)$ given by $q\mapsto qs$ such that $0s=0$ and $q0=0$ for all $s,q$ and where the obvious associativity and identity actions hold.

We can define a category $\mathbf{Act}_0$ whose objects consist of
pairs $(\mathbb A,C)$ where $C$ is a category with zero and $\mathbb
A$ is an action of $C$ on a pointed set.  A morphism $F\colon (\mathbb
A,C)\to (\mathbb B,D)$ is a pair $(\eta,F)$ where $F\colon C\to D$ is
a zero-preserving functor and $\eta\colon \mathbb A\to \mathbb B\circ
F$ is a natural transformation.
Two actions $(\mathbb A,C)$
and $(\mathbb A',D)$ of categories with zero on pointed sets are
\emph{equivalent}, written $\mathbb A\sim\mathbb A'$, if they are isomorphic in the category $\mathbf{Act}_0$.

\begin{Def}\label{def:action-of-K(S)}
  Consider an action of a semigroup
  $S$ with zero on a pointed set $Q$.
  Let $\mathbb A_Q$ be the action of $\KK(S)$ on $Q$ such that
  $\mathbb A_Q(e)=Qe$ for
  every object $e$ of $\KK(S)$,
  and such that $q\cdot (e,s,f)=q\cdot s$ for $q\in Qe$.
\end{Def}

It is easy to check the construction of
Definition~\ref{def:action-of-K(S)} is functorial in
an appropriate category, yielding
that isomorphic actions of semigroups
produce isomorphic actions of the respective Karoubi envelopes

Let $C$ be a category.  An assignment of an object $F(Q,S)$ of $C$ to each action of a semigroup $S$ with zero on a pointed set $Q$ is said to be a \emph{Karoubi invariant} of the action if $F(Q,S)$ is isomorphic to $F(Q',T)$ whenever the actions $(\KK(S),\mathbb A_Q)$ and $(\KK (T),\mathbb A_{Q'})$ are equivalent.

\subsection{Green's relations}\label{sec:greens-relat-relat}

Throughout this paper we use basic notions from semigroup theory
that can be found
in standard
texts~\cite{Clifford&Preston:1961,Lallement:1979,Rhodes&Steinberg:2009}.
Green's (equivalence) relations, which we next recall,  are among them.
The relation $\J$ is defined
on a semigroup $S$ by putting $s\mathrel{\J} t$ if $s$ and $t$ generate the same
two-sided principal ideal, that is, if $S^1sS^1=S^1tS^1$, where
$S^1$ denotes the monoid obtained from $S$ adjoining an identity.
Similarly, one
defines the $\R$- and $\L$-relations,
by replacing two-sided ideals with right (respectively,
left) ideals.
The intersection of the equivalence relations $\R$ and $\L$ is
denoted by $\H$.
The maximal subgroups of $S$ are the $\H$-classes containing idempotents.
Finally, Green's relation $\D$ is defined by $s\mathrel{\D} t$ if and only if there exists $u\in S$ with $s\mathrel{\R} u \mathrel{\L} t$. This is known to be equivalent to the existence of $v\in S$ with $s\mathrel{\L} v\mathrel{\R} t$. This is the smallest equivalence relation containing $\L$ and $\R$. Details can be found in~\cite{Clifford&Preston:1961}.
For finite semigroups, one has $\D=\J$~\cite[Appendix A]{Rhodes&Steinberg:2009}.

An element $s$ of a semigroup $S$ is \emph{regular} if
$s=sxs$ for some $x\in S$.
A $\D$-class contains regular elements if and only if
it contains an idempotent, if and only if
all its elements are regular.
A $\D$-class with regular elements is called \emph{regular}.

Let $H$ be an $\H$-class of $S$.
The set $T=\{x\in S^1\colon xH\subseteq H\}$ is a submonoid of $S^1$,
called the \emph{left stabilizer} of $H$.
The quotient of $T$ by its left action on $H$
is a group known as the \emph{Sch\"utzenberger group of $H$}.
Exchanging right and left, one obtains an isomorphic group.
If $H$ is a group (which occurs if and only if it contains an idempotent), then it is isomorphic to its Sch\"utzenberger group.
Two $\H$-classes contained in the same $\D$-class have isomorphic
Sch\"utzenberger groups, hence the
expression~\emph{Sch\"utzenberger group of a $\D$-class} is
meaningful.  See~\cite{Clifford&Preston:1961} for details.

The following lemma concerning the Karoubi envelope $\KK(S)$ of a semigroup is well known and easy to prove~\cite{Lawson:2011,Rhodes&Steinberg:2009}.

\begin{Lemma}\label{l:isomorphism-criterion}
  Two objects $e$ and $f$ of $\KK(S)$ are isomorphic
  if and only if
  $e$ and $f$ are $\D$-equivalent in $S$.  Moreover, the automorphism group of $e$ is isomorphic to the Sch\"utzenberger group of the $\D$-class of $e$.
\end{Lemma}

The following lemma will be useful several times.

\begin{Lemma}\label{l:a-lemma-on-conjugate-words}
Suppose that $u,v$ are two elements of a semigroup $S$ such that $u=zt$ and $v=tz$ with $z,t\in S$.  If $u$ is an idempotent, then so is $v^2$.  Moreover, $u\mathrel{\D} v$.
\end{Lemma}
\begin{proof}
Trivially, $v^4=tu^3v=tuz=v^2$ and so $v^2$ is idempotent.  Next, let $s=uz$ and observe that $st=uzt=u^2=u$.  Also, $ts=tztz=v^2$ and $zv^2=ztztz=u^2z=uz=s$.  Thus we have $u\mathrel{\R} s\mathrel{\L} v^2$.
\end{proof}

\section{Symbolic dynamics}\label{sec:basic-defin-symb}

\subsection{Shifts}

A good reference for the notions that we shall use here from symbolic
dynamics is~\cite{MarcusandLind}.
To make the paper reasonably
self-contained and to introduce notation, we recall some basic
definitions.

Let $A$ be a finite alphabet, and consider the set
$\z A$ of all bi-infinite sequences over $A$.
The \emph{shift} on $\z A$ is the homeomorphism
$\sigma_A\colon \z A\to \z A$ (or just $\sigma$) defined by
 $\sigma_A((x_i)_{i\in \ZZ})=(x_{i+1})_{i\in \ZZ}$.
We endow $\z A$ with the product topology with respect to the
discrete topology on $A$. In particular, $\z A$ is a compact totally
disconnected space (we include the Hausdorff property in the definition of
compact.)

We assume henceforth that all alphabets are finite.
A \emph{symbolic dynamical system}
is a non-empty closed subset $\Cl X$ of $\z A$, for some
alphabet $A$, such that $\sigma(\Cl X)=\Cl X$.
Symbolic dynamical systems are also known as
\emph{shift spaces}, \emph{subshifts}, or, more plainly, just
\emph{shifts}. We shall generally prefer the latter option except
when emphasizing that one shift is a subshift of another.

We can consider the category
of shifts, whose objects are the shifts
and where a morphism
between two shifts $\Cl X\subseteq \z A$ and $\Cl Y\subseteq \z B$
is a continuous function $\Phi\colon\Cl X\to\Cl Y$ such that
$\Phi\circ\sigma_A=\sigma_B\circ \Phi$.
In this category, an isomorphism is called a \emph{conjugacy}.
Isomorphic shifts are said to be \emph{conjugate}.

  A \emph{block} of $(x_i)_{i\in\ZZ}\in \z A$ is a word
  of the form $x_ix_{i+1}\cdots x_{i+n}$
  (briefly denoted by $x_{[i,i+n]}$), where $i\in\ZZ$ and
  $n\geq 0$. If $\Cl X$ is a subset of $A^{\ZZ}$ then
  $L(\Cl X)$ denotes the set of blocks of elements of $\Cl X$.
  If $\Cl X$ is a subshift of $A^{\ZZ}$
  and $x\in A^{\ZZ}$, then $x\in \Cl X$ if and only if
  $L(\{x\})\subseteq L(\Cl X)$~\cite[Corollary
  1.3.5]{MarcusandLind}.

  If $x$ is a periodic point with period $n\geq 1$
  (that is, if $\sigma^n(x)=x$), then
  we may represent $x$ by $u^\infty$, where
  $u=x_{[0,n-1]}$.

\subsection{The syntactic semigroup of a shift}

We shall make use of several well-known fundamental ideas and
facts about the interplay between finite automata, formal languages
and semigroups which, with some variations, can be found in several books, such
as~\cite{Eilenberg:1976,PinBook}.

  The free semigroup and the free monoid over an
  alphabet $A$ are denoted by $A^+$ and $A^\ast$, respectively.
Recall that a subset of $A^+$ is in this context called a
   \emph{language}.
If $A$ is an alphabet, let $A^+_0 = A^+\cup \{0\}$ be the free semigroup on $A$ with zero.  The multiplication of $A^+$ is extended to make $0$ a zero element.
Let $L\subseteq A^+$ be a language and let $u\in
A^+_0$.
The
\emph{context of $u$ in $L$} is the set $[u]_L=\{(x,y)\in A^\ast\colon xuy\in L\}$.
Of course, $[0]_L=\emptyset$.
The relation $\equiv_L$ on $A^+_0$
such that $u\equiv_L v$ if and only if $[u]_L=[v]_L$
is a semigroup congruence,
and the quotient semigroup $A^+_0/{\equiv_L}$
is the \emph{syntactic semigroup (with zero) of $L$}, denoted $S(L)$.
The quotient homomorphism $A^+_0\to S(L)$ is the \emph{syntactic
  homomorphism}, denoted $\delta_L$.
Note that the class of $0$ is the zero of $S(L)$.
Since $u\equiv_L v$ if and only if
$[u]_L$ and $[v]_L$ are equal,
we may identify the $\equiv_L$-equivalence class of $u$
with~$[u]_L$.

An onto homomorphism $\varphi\colon A^+_0\to S$
\emph{recognizes} the language $L$ of $A^+$ if $L=\varphi^{-1}(P)$ for
some subset $P$ of $S$, in which case we also say that $S$
recognizes~$L$,
a property equivalent
to  $S$ being isomorphic to a quotient
of $A^+_0$ by a congruence that saturates $L$
(a semigroup is \emph{saturated} by a congruence when it is a union of some of its congruence classes.)
It turns out that the syntactic congruence $\equiv_L$
is the greatest congruence (for the inclusion) that saturates~$L$.
This means that if a semigroup $S$ recognizes $L$ as a subset of $A^+_0$,
then $S(L)$ is a homomorphic image of $S$.

For a subshift \Cl X of \z A, we use
the notations $\delta_{\Cl X}$,
$S(\Cl X)$, ${LU}(\Cl X)$ instead of $\delta_{L(\Cl X)}$, $S(L(\Cl X))$
and ${LU}(S(\Cl X))$ respectively.
Also, instead of $[u]_{L(\Cl X)}$, we may use
$[u]_{\Cl X}$ or simply $[u]$ if no confusion arises.
We say that $S(\Cl X)$ is
the \emph{syntactic semigroup of $\Cl X$}.
If $\Cl X\subseteq \z A$ and $A\subseteq B$, then the
syntactic semigroups of $\Cl X$ viewed as a subset of $\z A$ and of
$\z B$ coincide.

\begin{Rmk}\label{rmk:idempotents-implies-periodic-points}
  One has $u\in A^+_0\setminus L(\Cl X)$
if and only if $[u]=\emptyset$.
In particular, if $[u]$ is a non-zero idempotent, then
the periodic point $u^\infty$ belongs to~$\Cl X$.
\end{Rmk}

    Let $S$ be a semigroup with zero
    such that $SsS\neq \{0\}$ for every $s\in S\setminus \{0\}$.  We say such a semigroup is \emph{prolongable}.
  Suppose $S$ is generated by a finite set~$A$,
  and let $\varphi\colon A^+\to S$
  be an onto homomorphism $S$.
  Then $L=\varphi^{-1}(S\setminus \{0\})$
  is clearly a factorial prolongable language,
  and so there is a unique subshift $\Cl X_\varphi$ of $\z A$
  such that $L=L(\Cl X_\varphi)$
  (cf.~\cite[Proposition 1.3.4]{MarcusandLind}.)
  We say that \emph{$\Cl X_\varphi$ is the shift induced by $\varphi$},
  or, more vaguely, \emph{induced by $S$}.

  Let us call a congruence \emph{trivial} if it is the
equality relation.  A semigroup $S$ with zero is \emph{$0$-disjunctive} if
  the greatest congruence saturating  $\{0\}$, (equivalently,
  saturating $S\setminus \{0\}$) is the trivial one.
  The next simple lemma concerns
  a well known property about syntactic semigroups
  (cf.~\cite[Proposition 5.3]{Lallement:1979}.)
 
  \begin{Lemma}\label{l:0-disjunctive-are-syntactic-semigroups}
    The syntactic semigroup $S(\Cl X)$ of a shift $\Cl X$ is $0$-disjunctive.
  Conversely, if $S$ is a prolongable semigroup with zero which is
    $0$-disjunctive, then~$S$ is the syntactic semigroup of every
    shift induced by $S$.
  \end{Lemma}

\subsection{Labeled graphs and sofic shifts}

\subsubsection{Labeled graphs}

In this paper, \emph{graph} will always mean a multi-edge directed
graph. By a~\emph{labeled graph (over an alphabet $A$)} we mean
a pair $(G,\lambda)$
consisting of a graph $G$ and a map $\lambda\colon E(G)\to A$,
where $E(G)$ is the set of edges of $G$,
such that if $e$ and $f$ are distinct edges
with the same origin and the same terminus, then
$\lambda(e)\neq\lambda(f)$.
The letter $\lambda(e)$ is the \emph{label} of $e$.
The labeled graph $(G,\lambda)$ is \emph{right-resolving} if
distinct edges with the same origin have distinct labels. It is
\emph{complete} if, for every vertex $q$ and every letter
$a\in A$, there is an edge starting in $q$ with label $a$.
Note that a complete, right-resolving labeled graph over  $A$ may be viewed as
an action of $A^+$ over the set of vertices: the existence of
an edge from $q$ to $r$ labeled $a$ corresponds to the equality $q\cdot a=r$.

\subsubsection{Minimal automaton}

An \emph{automaton} is a labeled graph (whose vertices
in this context are often called \emph{states})
together with a set $I$ of \emph{initial states} and a set $F$
of \emph{final states}.
The automaton is \emph{deterministic} if it is
right-resolving and has a
single initial state.
The set of words labeling paths from $I$
to $F$ is the language \emph{recognized} by the automaton.
We view a labeled graph as an automaton in which all states are
initial and final.

Consider a language $L\subseteq A^+$.
Let $u$ be an element of $A^\ast_0=A^\ast\cup\{0\}$, the free monoid
with zero. The \emph{right context of $u$ in $L$} is the
set
  $R_{L}(u)=\{w\in A^\ast\mid uw\in L\}$.
Viewing $L$ as a language over the alphabet $A\cup \{0\}$,
we can consider its \emph{minimal complete deterministic
  automaton}, which is the terminal object in the category of
complete deterministic automatons, over $A\cup \{0\}$,
recognizing~$L$.
This automaton, which we denote by $\mathfrak M(L)$,
can be realized as follows: the states are the
right contexts of $L$,
the initial state is $R_L(1)$,
the final states are the
right contexts $R_L(u)$ such that $u\in L$,
and the action of an element $v$ of $A^+_0$
is given by $R_{L}(u)\cdot v=R_{L}(uv)$.
A \emph{sink} in a labeled graph
is a vertex $z$ such that all edges starting in $z$ are loops;
the vertex $R_L(0)=\emptyset$ is the unique sink of $\mathfrak M(L)$.
The language $L$ is \emph{recognizable} if it can be recognized by a finite
automaton; $L$ is recognizable if and only if $\mathfrak M(L)$ is
finite, if and only if $S(L)$ is finite~\cite{Eilenberg:1976,PinBook}.

Note that $u,v\in A^+_0$ have the
same action on the states of $\mathfrak M(L)$ if and only if
$[u]_L=[v]_L$. In particular, we may consider the action of
$S(L)$ on the set of states defined by
$R_{L}(u)\cdot [v]_L=R_{L}(uv)$.

 For a subshift \Cl X of \z A, we use
 the notations $\mathfrak M(\Cl X)$
 and $R_{\Cl X}(u)$
 instead of $\mathfrak M(L(\Cl X))$
 and $R_{L(\Cl X)}(u)$.

\subsubsection{Sofic shifts}

A graph $G$ is \emph{essential} if
the in-degree and the out-degree
of each vertex is at least one.
If $L(\Cl X)$, for a shift $\Cl X$, is recognized by the essential labeled graph $(G,\lambda)$, then $\Cl X$ is the \emph{shift presented by $(G,\lambda)$}.

The shifts that can be presented by a finite labeled graph are
called \emph{sofic}~\cite[Chapter 3]{MarcusandLind}.
The sofic shifts are the shifts $\Cl X$
such that $L(\Cl X)$ is recognizable.
That is, $\Cl X$ is sofic if and only if $S(\Cl X)$ is finite.
The most studied class of sofic shifts is that of
\emph{finite type shifts}~\cite[Chapter 2]{MarcusandLind}: a
subshift $\Cl X$ of $\z A$
is of finite type when $L(\Cl X)=A^+\setminus A^\ast W A^\ast$, for
some finite subset $W$ of $A^+$. 
An \emph{edge shift} is a shift presented
by a finite essential labeled graph
$(G,\lambda)$
such that the mapping $\lambda$ is one-to-one.
 One of the characterizations of the shifts of finite type is that
 they are the shifts conjugate to edge
 shifts~\cite[Theorem 2.3.2]{MarcusandLind}.

   A subshift $\Cl X$ of \z A is \emph{irreducible} if, for all
$u,v\in L(\Cl X)$, there is $w\in A^\ast$ such that
$uwv\in L(\Cl X)$.
A sofic shift is irreducible if and only if it can be presented by a
strongly connected labeled graph~\cite[Proposition 3.3.11]{MarcusandLind}.

\subsection{Synchronizing shifts}\label{sec:synchronizing-shifts}

A word $u$ of $L(\Cl X)$
is \emph{synchronizing}\footnote{In~\cite{MarcusandLind},
  a synchronizing word is called an \emph{intrinsically
    synchronizing} word. There is some diversity of terminology in
  the literature (cf.~\cite[Remark 2.6]{Bates&Eilers&Pask:2011}.)}
if $vu,uw\in L(\Cl X)$ implies $vuw\in L(\Cl X)$.
An irreducible shift $\Cl X$ is \emph{synchronizing} if
 $L(\Cl X)$ contains a synchronizing word.
Every irreducible sofic shift is synchronizing
(cf.~\cite[Lemma 2]{Fischer:1975} and~\cite[Proposition
3.1]{Blanchard&Hansel:1986}.)
Also, a shift is of finite type if and only if there is some $n$ such
that every word of $L(\Cl X)$ of length at least $n$ is synchronizing,
in which case the shift is said to be an
\emph{$n$-step shift}~\cite[Theorem 2.1.8]{MarcusandLind}.

The following lemma about synchronizing words can be useful.

\begin{Lemma}\label{l:synchwords}
Let $\Cl X\subseteq A^{\mathbb Z}$ be a shift for which $L(\Cl X)$ has a
synchronizing word (e.g., a synchronizing shift.)
\begin{enumerate}
\item The union of $A^+\setminus L(\Cl X)$
  with the set of synchronizing words of $L(\Cl X)$ is an ideal of
  $A^+$.
\item If $u$ is synchronizing, then $[u]$ is idempotent if and only if $u^2\in L(\Cl X)$.
\item If $v$ is synchronizing, then $uv\in L(\Cl X)$ implies
  $R_{\Cl X}(uv) = R_{\Cl X}(v)$.
\end{enumerate}
\end{Lemma}
\begin{proof}
Suppose that $u$ is synchronizing and let $r\in A^+$.  Then if
$vru,ruw\in L(\Cl X)$, one has that $(vr)uw\in L(\Cl X)$ because $u$
is synchronizing.  Thus $ru$ is synchronizing.  Similarly, $ur$ is
synchronizing. This proves the first item.

Suppose that $u$ maps to an idempotent of $S(\Cl X)$.  Then trivially $u^2\in L(\Cl X)$.  For the converse, suppose that $u^2\in L(\Cl X)$.  If $v,w\in A^*$ with $vuw\in L(\Cl X)$, then because $vu,uu\in L(\Cl X)$, we have $vuu\in L(\Cl X)$.  But then $vuu,uw\in L(\Cl X)$ implies that $vu^2w\in L(\Cl X)$.  Conversely, if $vu^2w\in L(\Cl X)$, then $vu,uw\in L(\Cl X)$ and so $vuw\in L(\Cl X)$ because $u$ is synchronizing.

The final statement follows because $vw\in L(\Cl X)$ if and only if $uvw\in L(\Cl X)$ by the definition of a synchronizing word.
\end{proof}

The main results of this paper concern the Karoubi envelope
of $S(\Cl X)$, which has at least one object, namely $0$.
The next proposition
gives a sufficient condition for
the existence of other objects; it is not necessary,
as witnessed by the classes of shifts analyzed in
Section~\ref{sec:markov-dyck-shifts}.

\begin{Prop}\label{p:sync-stabilizing-idemp}
  Let $\Cl X$ be a synchronizing shift.
  If $u$ is
  a synchronizing word of $L(\Cl X)$, then
  $R_{\Cl X}(u)=R_{\Cl X}(1)\cdot e$ for some
  idempotent $e\in S(\Cl X)\setminus \{0\}$.
\end{Prop}

\begin{proof}
  As $\Cl X$ is irreducible, there is $v$ such
  that $uvu\in L(\Cl X)$. Then $vu,uvu\in L(\Cl X)$
  implies that $vuvu\in L(\Cl X)$.  By Lemma~\ref{l:synchwords}, $vu$
  is synchronizing and $e=[vu]$ is a non-zero idempotent.  Lemma~\ref{l:synchwords} also yields $R_{\Cl X}(u) = R_{\Cl X}(vu) = R_{\Cl X}(1)\cdot e$.
\end{proof}

\subsection{Krieger and Fischer covers}\label{sec:krieg-fisch-covers}

Denote the set of negative integers by $\mathbb Z^-$,
and the set of non-negative integers by $\mathbb N$.
Given an element $x=(x_i)_{i\in\mathbb Z^-}$ of
$A^{\mathbb Z^-}$ and an element $y=(y_i)_{i\in\mathbb N}$
of $A^{\mathbb N}$, denote by $x.y$ the element
$z=(z_i)_{i\in\mathbb Z}$ of $A^{\mathbb Z}$
for which $z_i=x_i$ if $i<0$, and $z_i=y_i$ if $i\ge 0$.

If $x\in A^{\mathbb Z^-}$ and $u=a_1\cdots a_n\in A^+$,
with
$a_i\in A$ when $1\le i\le n$,
then $xu$ denotes the element of $A^{\mathbb Z^-}$
given by the left-infinite sequence
$\cdots x_{-3}x_{-2}x_{-1}a_1\cdots a_n$.
Similarly, if $x\in A^{\mathbb N}$ then
$ux\in A^{\mathbb N}$ is given by the
right-infinite sequence $a_1\cdots a_nx_{0}x_{1}x_{2}\cdots$.

Let $C_{\Cl X}(x)=\{y\in A^{\mathbb N}:x.y\in \Cl X\}$,
where $\Cl X$ is a subshift of $\z A$
and $x\in A^{\mathbb Z^-}$.
If $u\in A^+$, then $C_{\Cl X}(x)=C_{\Cl X}(z)$
implies $C_{\Cl X}(xu)=C_{\Cl X}(zu)$.
This enables the following definition.

\begin{Def}[Krieger cover]
  Let $Q(\Cl X)=\{C_{\Cl X}(x)\mid x\in A^{\mathbb Z^-}\}\cup\{\emptyset\}$.
  Denote by $\mathfrak K^0(\Cl X)$
  the right-resolving complete labeled graph over $A\cup\{0\}$,
  with vertex set $Q(\Cl X)$,
  defined by the action of $A^+_0$ on
  $Q(\Cl X)$ given by
  \begin{equation*}
  C_{\Cl X}(x)\cdot u=C_{\Cl X}(xu)\quad\text{if $u\in A^+$},\quad C_{\Cl X}(x)\cdot 0=\emptyset,
  \end{equation*}
  having $\emptyset$ as unique sink.
The labeled graph $\mathfrak K(\Cl X)$
over $A$
obtained from
$\mathfrak K^0(\Cl X)$ by elimination of the vertex $\emptyset$ is
the \emph{right Krieger cover of $\Cl X$}
(cf.~\cite[Definition 0.11]{Fiebig&Fiebig:1991}.)
\end{Def}

Krieger introduced in~\cite{Krieger:1984}
this cover for sofic shifts only.
If $\Cl X$ is sofic, then the right Krieger cover of $\Cl X$
embeds in the automaton $\mathfrak M(\Cl X)$ and it is
computable~\cite[Section 4]{Beal&Berstel&Eilers&Perrin:2010arxiv}.
There are examples of synchronizing
shifts whose Krieger graph is uncountable (cf., the example in the proof
of~\cite[Corollary 1.3]{Fiebig&Fiebig:1991}.)
Hence, in the non-sofic case, the Krieger cover
may not be a labeled subgraph of $\mathfrak M(\Cl X)$,
which is always at most countable.

We next relate $S(\Cl X)$ with the Krieger cover.

\begin{Lemma}\label{l:definition-of-action-on-Krieger-cover}
  Consider a subshift $\Cl X$ of $\z A$, and let $u,v\in A^+$.
  Then $[u]\subseteq  [v]$ if and only if $C_{\Cl X}(xu)\subseteq C_{\Cl X}(xv)$ for all
  $x\in  A^{\mathbb Z^-}$.
\end{Lemma}
\begin{proof}
Assume that $[u]\subseteq [v]$ and suppose that $y\in C_{\Cl X}(xu)$.
Then we have $x_{[-n,-1]}uy_{[0,n]}\in L(\Cl X)$ for every $n\geq 1$.
As $[u]\subseteq  [v]$,
we obtain
$x_{[-n,-1]}vy_{[0,n]}\in L(\Cl X)$ for every $n\geq 1$,
which means that $y\in C_{\Cl X}(xv)$.

Next suppose that $[u]\nsubseteq [v]$ and suppose $wuz\in L(\Cl X)$
and $wvz\notin L(\mathcal X)$.  Then,
as $L(\Cl X)$ is prolongable, we can find left infinite $x$ and right infinite $y$ with $xwu.zy\in \mathcal X$. Then $zy\in C_{\Cl X}(xwu)$ and $zy\notin C_{\Cl X}(xwv)$.
\end{proof}

Hence, by Lemma~\ref{l:definition-of-action-on-Krieger-cover},
the action of $S(\Cl X)$
on the set of vertices of $\mathfrak K^0(\Cl X)$
given
by $C_{\Cl X}(x)\cdot [u]=C_{\Cl X}(xu)$
is well defined and faithful.

In a graph $G$, a subgraph $H$
is \emph{terminal} if every edge of $G$ starting in a vertex of $H$
belongs to~$H$.
A \emph{strongly connected component} of $G$ is a maximal strongly connected
subgraph of~$G$.
 As a reference for the next definition,
 we give  Definition 0.12 in~\cite{Fiebig&Fiebig:1991},
 and the lines following it in~\cite{Fiebig&Fiebig:1991},
 for justification.

\begin{Def}[Fischer cover]
  If $\Cl X$ is a synchronizing shift,
  then  $\mathfrak K(\Cl X)$
  has a sole strongly connected terminal component. This
  component presents $\Cl X$. It is called the
\emph{right Fischer cover of $\Cl X$}, and we denote it by
$\mathfrak F(\Cl X)$.

We denote by $\mathfrak F^0(\Cl X)$ the terminal
complete labeled subgraph of $\mathfrak K^0(\Cl X)$
obtained from $\mathfrak F(\Cl X)$
by adjoining the sink state $\{\emptyset\}$,
and by $Q_{\mathfrak F}(\Cl X)$ the vertex set of
$\mathfrak F^0(\Cl X)$.
\end{Def}

  See~\cite[Theorem 2.16]{Fiebig&Fiebig:1991}
  and~\cite[Corollary 3.3.19]{MarcusandLind}
  for characterizations of the Fischer cover.

The next result was
shown in~\cite{Beauquier:1985} for irreducible sofic
shifts.
The generalization for synchronizing shifts offers no additional
difficulty.

\begin{Prop}\label{p:fischer-cover}
  Let $\Cl X$ be a synchronizing shift.
  Then the labeled graph
obtained from
 $\mathfrak M(\Cl X)$ by eliminating the sink vertex $\emptyset$
 has a unique terminal strongly connected component,
  which is isomorphic with $\mathfrak F(\Cl X)$.
  Its vertices are the right-contexts of synchronizing words.
\end{Prop}

We have defined the \emph{right} Krieger and Fischer covers.
The \emph{left} Krieger and Fischer covers are defined analogously,
changing directions when needed.

If $\mathcal A=(G,\lambda)$ is a right-resolving labeled graph with
state set $Q$ and alphabet $A$, then recall we write $qa=q'$ if there
is an edge from $q$ to $q'$ labeled by $a$.  The \emph{transition
  semigroup} $S(\mathcal A)$ is the semigroup of partial mappings of
$Q$ generated by the maps $a\mapsto qa$.  We write $Q^0$ for $Q\cup
\{0\}$ with $0$ an adjoined sink state and define $qs=0$ if $qs$ was
undefined and $0s=0$ for all $s\in S(\mathcal A)$.
For every shift $\Cl X$,
and because the action of $S(\Cl X)$ on its Krieger
cover $\mathfrak {K}(\Cl X)$ is faithful
(Lemma~\ref{l:definition-of-action-on-Krieger-cover}),
the transition semigroup of $\mathfrak {K}(\Cl X)$
is isomorphic to $S(\Cl X)$. One can also easily check that if $\Cl X$ is
synchronizing then the action of
$S(\Cl X)$ on the Fischer cover of $\Cl X$ is faithful: the proof
given in~\cite[Proposition 4.8]{Beal&Berstel&Eilers&Perrin:2010arxiv}
for irreducible sofic shifts holds for every synchronizing shift.
Therefore, if $\Cl X$ is synchronizing then $S(\Cl X)$ is isomorphic
to the transition semigroup of the Fischer cover $\mathfrak {F}(\Cl X)$.

\subsection{Conjugacy and Nasu's theorem}

 A right-resolving labeled graph $\mathcal A=(G,\lambda)$ over an alphabet $A$ is \emph{bipartite} if there are partitions $Q = Q_1\uplus Q_2$ of the states and
$A = A_1\uplus  A_2$ of the alphabet such that all edges labeled by $A_1$ go from $Q_1$ to $Q_2$ and
all edges labeled by $A_2$ go from $Q_2$ to $Q_1$.

Let $\mathcal A_1=(G_1,\lambda_1)$ be the right-resolving labeled graph over $A_1A_2$ with state set $Q_1$ obtained by turning each path of length $2$ from $Q_1$ to itself into an edge (labeled by the product of the labels of the two edges) and define $\mathcal A_2=(G_2,\lambda_2)$ with alphabet $A_2A_1$ and state set $Q_2$, analogously. We call $\mathcal A_1,\mathcal A_2$ the \emph{components} of $\mathcal A$.

Write $\mathcal A\sim \mathcal B$ if there is a bipartite right-resolving labeled graph with components $\mathcal A,\mathcal B$ and let $\simeq$ be the equivalence relation on right-resolving labeled graphs generated by $\sim$.

The following fundamental result is in~\cite{Nasu:1986}, where it is stated for sofic shifts but is well known to apply to any shift.

\begin{Thm}[Nasu~1986]\label{t:nasu}
Let $\Cl X_1,\Cl X_2$ be shifts with Krieger covers $\mathfrak K(\Cl
X_1)$ and $\mathfrak K(\Cl X_2)$, respectively.  Then
$\Cl X_1$ and $\Cl X_2$ are conjugate
if and only if $\mathfrak K(\Cl X_1)\simeq \mathfrak K(\Cl X_2)$.
\end{Thm}

\subsection{Flow equivalence}\label{sec:flow-equivalence}
Fix an alphabet $A$ and a letter $\alpha$ of $A$.
Let $\dia$ be a letter not in $A$.
Denote by $B$ the alphabet $A\cup\{\dia\}$.
The \emph{symbol expansion of $A^\ast$ associated to $\alpha$}
is the unique monoid homomorphism $\ci E\colon A^\ast\to B^\ast$
such that $\ci E(\alpha)=\alpha\dia$ and $\ci E(a)=a$ for all
$a\in A\setminus\{\alpha\}$.
Note that $\ci E$ is injective.

The \emph{symbol expansion of a subshift \Cl X of \z A relatively to $\alpha$}
is the least subshift $\Cl X'$ of \z B such that $L(\Cl X')$
contains $\ci E(L(\Cl X))$.
A \emph{symbol expansion of $\Cl X$} is a symbol expansion of $\Cl X$
relatively to some
letter.

The mapping $\ci E$
admits the following natural extension of its domain and range.
If $x\in A^{\mathbb Z^-}$
and $y\in A^{\mathbb N}$,
then $\ci E(x)$ and $\ci E(y)$
are respectively the elements of $B^{\mathbb Z^-}$ and $B^{\mathbb N}$
given by
$\ci E(x)=\ldots \ci E(x_{-3})\,\ci E(x_{-2})\,\ci E(x_{-1})$
and
$\ci E(y)=\ci E(y_{0})\,\ci E(y_{1})\,\ci E(y_{2})\ldots$.
Moreover,  $\ci E(x.y)$ denotes $\ci E(x).\ci E(y)\in B^{\mathbb Z}$.
Note that $\Cl X'$ is the least subshift of
$B^{\mathbb Z}$ containing $\ci E(\Cl X)$.

\emph{Flow equivalence} is the least equivalence relation
between shifts containing the conjugacy and symbol expansion
relations.
The classes of finite type shifts, of
sofic shifts, and of irreducible shifts are all easily seen to be closed under flow equivalence.
See~\cite[Section 13.6]{MarcusandLind} for motivation for studying
flow equivalence. Here, we just remark
that the original definition of flow equivalence (that
two shifts are flow equivalent if their suspension flows are
topologically equivalent)
was proved in~\cite{Parry&Sullivan:1975} to be equivalent to the one
we use, explicitly for finite type shifts,
but as pointed out in~\cite[Lemma 2.1]{Matsumoto:2001},
implicitly for all shifts. See also~\cite[page 87]{Parry&Tuncel:1982}.

\section{Statement of the main results}\label{sec:stat-main-results}

In this section we state the main results of this paper, deferring
proofs to the final sections.

\begin{Def}[Karoubi envelope of a shift]
Let $\Cl X$ be a shift. We define the \emph{Karoubi envelope $\KK(\Cl X)$ of $\Cl X$} to be the Karoubi envelope $\KK(S(\Cl X))$ of its syntactic semigroup.
\end{Def}

Note that the category $\KK(\Cl X)$
is a category with zero, since $S(\Cl X)$ is a semigroup with zero.
Our principal result is that the natural equivalence class of $\mathbb
K(\Cl X)$ is a flow equivalence invariant of $\Cl X$.

\begin{Thm}\label{splittingisinvflow}
If $\Cl X$ and $\Cl Y$ are flow equivalent
shifts, then the categories $\KK(\Cl X)$ and $\KK(\Cl Y)$ are equivalent, i.e., $S(\Cl X)$ and $S(\Cl Y)$ are Morita equivalent up to local units. Hence every Karoubi invariant of $S(\Cl X)$ is a flow equivalence invariant of $\Cl X$.
\end{Thm}

We remark that if $\Cl X$ and $\Cl Y$ are sofic shifts given by
presentations, then one can effectively determine whether $\KK(\Cl X)$
is equivalent to $\KK(\Cl Y)$.  This is because these categories are
finite and effectively constructible and so one can
in principle check all functors between them and see if there is one which is an equivalence.

For a shift $\Cl X$, denote by $\mathbb A_{\Cl X}$
the action $\mathbb A_{Q(\Cl X)}$ arising from the action of
$S(\Cl X)$ on~$Q(\Cl X)$. The reader is referred back to Definition~\ref{def:action-of-K(S)} for the notation.

   \begin{Thm}\label{t:main-D-cate-AYF}
   If $\Cl X$ and $\Cl Y$ are flow equivalent shifts, then the actions
   $(\mathbb A_{\Cl X}, \KK(\Cl X))$ and $(\mathbb A_{\Cl Y},\KK(\Cl Y))$ are equivalent. Hence any Karoubi invariant of the action of $S(\Cl X)$ on $Q(\Cl X)$ is an invariant of flow equivalence.
   \end{Thm}

The proofs of Theorems~\ref{splittingisinvflow} and~\ref{t:main-D-cate-AYF} are carried out in
Section~\ref{sec:proof-theor-reft:m}.

\begin{Rmk}\label{r:dead-end}
Recall from Subsection~\ref{sec:krieg-fisch-covers}
the definition of right Fischer cover when $\Cl X$ is a synchronizing shift. In particular,  $Q_{\mathfrak F}(\Cl X)$ is the unique minimal $S(\Cl X)$-invariant subset of $Q(\Cl X)$ strictly containing the sink state.

 Therefore, the action of $S(\Cl X)$
 on $Q(\Cl X)$ restricts in a natural way
 to an action of $S(\Cl X)$ on $Q_{\mathfrak F}(\Cl X)$,
 denoted $\mathbb A_{\Cl X}^{\mathfrak F}$.
\end{Rmk}

  \begin{Thm}\label{t:main-D-cate-AYF-Fischer-version}
   If $\Cl X$ and $\Cl Y$ are flow equivalent synchronizing shifts,
   then the actions $(\mathbb A_{\Cl X}^{\mathfrak F}, \KK(\Cl X))$ and $(\mathbb A_{\Cl Y}^{\mathfrak F},\KK(\Cl Y))$ are equivalent. Hence each Karoubi invariant of the action of $S(\Cl X)$ on $Q_{\mathfrak F}(\Cl X)$ is an invariant of flow equivalence.
  \end{Thm}

  The proof of
  Theorem~\ref{t:main-D-cate-AYF-Fischer-version}
  is deferred to Section~\ref{sec:proof-theor-reft:m-1}.

If $\Cl X$ does not contain periodic points, for instance,  if $\Cl X$ is minimal non-periodic (cf.~\cite[Section
 13.7]{MarcusandLind}),
 then $0$ is the unique
 idempotent of $S(\Cl X)$
 (cf.~Remark~\ref{rmk:idempotents-implies-periodic-points}),
 and so in such cases
 our main results have
 no applications.
 On the other hand, they do have meaningful consequences
for sofic, synchronizing, and other
classes of shifts. In the next few sections we examine some of these
consequences.

Before doing that, we show that the Karoubi envelope of a sofic shift
is a universal syntactic invariant of flow equivalence of sofic shifts.
But first, we need
a suitable formalization of what this means. 

An equivalence relation $\vartheta$ on the class of sofic shifts is an
\emph{invariant of flow equivalence} if whenever $\Cl X$ and $\Cl Y$
are flow equivalent sofic shifts, one has $\Cl X\mathrel{\vartheta}
\Cl Y$.  For instance, our main theorem says that having naturally
equivalent Karoubi envelopes is an invariant of flow equivalence.

Intuitively, a syntactic invariant of a sofic shift $\Cl X$ should be something which depends only on the syntactic semigroup $\Cl X$, or equivalently, on the action of $S(\Cl X)$ on the states of the Krieger cover.  Formally, we say that an equivalence relation $\vartheta$ on sofic shifts is a \emph{syntactic invariant} if $S(\Cl X)\cong S(\Cl Y)$ implies that $\Cl X\mathrel{\vartheta} \Cl Y$.  Our final main result, proved in Section~\ref{sec:master}, states that the Karoubi envelope is the universal syntactic invariant of flow equivalence of sofic shifts in the sense that any other invariant factors through it.

\begin{Thm}\label{t:master-syntactic-invariant}
Suppose that $\vartheta$ is a syntactic invariant of flow equivalence
of sofic shifts and that $\Cl X$ and $\Cl Y$ are sofic
shifts such that $\KK(\Cl X)$ is equivalent to $\KK(\Cl Y)$. Then
$\Cl X\mathrel{\vartheta}\Cl Y$.
\end{Thm}

We remark that the sofic hypothesis is essential
in Theorem~\ref{t:master-syntactic-invariant}.
The expected generalization to all shifts does not hold.
For example, if $\Cl X$ is the union of a periodic shift
$\Cl Y$ with a minimal non-periodic shift,
then $\KK(\Cl X)$ and $\KK(\Cl Y)$ are isomorphic,
but $S(\Cl X)$ is infinite and $S(\Cl Y)$ is finite.
And of course, finiteness of the syntactic semigroup
is a flow invariant.

\section{The proper comunication graph}\label{sec:prop-comun-graph}

For a graph $G$, let $PC(G)$ be the set of non-trivial
strongly connected components of $G$. Here we consider a strongly connected graph to be trivial if it consists of one vertex and no edges (a single vertex with some loop edges is deemed non-trivial.)  Consider in $PC(G)$ the partial order
given by $C_1\leq C_2$ if and only if there is in $G$
a path from an element of $C_1$ to an element of $C_2$.
Following the terminology of~\cite{Bates&Eilers&Pask:2011}, the
\emph{proper communication graph of $G$} is the acyclic directed graph with vertex set $PC(G)$ and edge set given by the irreflexive relation $<$.  It is proved
in~\cite{Bates&Eilers&Pask:2011} that the proper communication graph
of the right (left)
Krieger cover of a sofic shift is a flow equivalence invariant.
We shall see in this section that this result can naturally be seen
as a consequence of Theorem~\ref{t:main-D-cate-AYF}, and in the
process of doing this, we generalize it
to arbitrary shifts.

Let $S$ be a semigroup with zero acting on a pointed set $Q$.  Let $I=Q\cdot E(S)$ and define a preorder $\preceq$ on $I$ by $q\preceq q'$ if $q'LU(S)\subseteq qLU(S)$, that is, $q'=qs$ for some $s\in LU(S)$.  As usual, define $q\sim q'$ if $q\preceq q'$ and $q'\preceq q$.  Then $P(Q)=I/{\sim}$ is a poset, which can be identified with the poset of cyclic $LU(S)$-subsets $\{q\cdot LU(S)\mid q\in I\}$) ordered by reverse inclusion.  In~\cite{newpaper}, we show that $P(Q)$ is a Karoubi invariant of the action of $S$ on $Q$.  More precisely, we have the following result that will appear in~\cite{newpaper}.

\begin{Thm}\label{t:actionposet}
Suppose $S,T$ are semigroups with zero acting on pointed sets $Q,Q'$, respectively.  Suppose that $F\colon \KK(S)\to \KK(T)$ is an equivalence and that $\eta\colon \mathbb A_Q\Rightarrow A_{Q'}\circ F$ is an isomorphism.  Then there is a well defined isomorphism $f\colon P(Q)\to P(Q')$ of posets given by $f(q) = \eta_e(q)LU(T)$ for $q\in Qe$.
\end{Thm}

Define for a shift $\Cl X$, the poset $P(\Cl X) = P(Q(\Cl X))$.  Note that $q\preceq q'$ if and only if there is a path from $q$ to $q'$ in $\mathfrak K^0(\Cl X)$.

Theorems~\ref{t:main-D-cate-AYF} and~\ref{t:actionposet} immediately
imply the following.
\begin{Cor}\label{c:ISX-is-flow-invariant}
  The poset $P(\Cl X)$ is a flow equivalence invariant.
\end{Cor}

\begin{Rmk}\label{rmk:what-psi-does-to-fischer-cover}
    If $\Cl X$ is a synchronizing shift then,
    by Proposition~\ref{p:fischer-cover},
    the elements of
    $Q_{\mathfrak F}(\Cl X)$ are the sink vertex $\emptyset$
    and the vertices of the form $R_{\Cl X}(u)$,
    with $u$ a synchronizing word.
    Therefore,
    Proposition~\ref{p:sync-stabilizing-idemp} implies that $qLU(\Cl X)$  belongs to $P(\Cl X)$ for
  every $q\in Q_{\mathfrak F}(\Cl X)$.  Moreover, by Remark~\ref{r:dead-end}, one has $qLU(\Cl X)=q'LU(\Cl X)$ for all $q,q'\in Q_{\mathfrak F}(\Cl X)\setminus \{\emptyset\}$ and this common element is the minimum element of $P(\Cl X)\setminus \{\emptyset\}$.
  Therefore, it follows from Theorem~\ref{t:actionposet}
  that $q\in Q_{\mathfrak F}(\Cl X)$
  if and only if $\eta_e(q)\in Q_{\mathfrak F}(\Cl Y)$,
  for every $e$ with $q\in Qe$.
\end{Rmk}

The following simple fact will be useful to specialize to sofic shifts.

\begin{Lemma}\label{l:trivial-remark-about-loops}
  Suppose that $\Cl X$ is a sofic shift.  Then $q$ belongs to
  a non-trivial strongly connected component
  of $\mathfrak K^0(\Cl X)$ if and only if
  there is an idempotent $e\in S(\Cl X)$ such that $q=q\cdot e$.
\end{Lemma}
\begin{proof}
If $qe=q$ with $e\in E(S(\Cl X))$, then any word representing $e$ labels a loop at $q$ and so $q$ belongs to a non-trivial strongly connected component.  Conversely, if $q$ belongs to a non-trivial strongly connected component, then there is a word $u$ labeling a non-empty loop rooted at $q$.
  Then $q=q\cdot [u^n]$ for all $n\geq 1$.
  On the other hand, since $S(\Cl X)$ is finite, there is $m\geq 1$ such that
  $[u^m]$ is idempotent (cf.~\cite[Proposition 1.6]{PinBook}.)
  This completes the proof.
\end{proof}

It is clear from Lemma~\ref{l:trivial-remark-about-loops} that $P(\Cl X)\cong PC(\mathfrak K^0(\Cl X))$ for a sofic shift $\Cl X$.

\begin{Cor}\label{c:proper-c-graph}
  The proper communication graph of the right (left) Krieger cover of a
  sofic shift is a flow equivalence invariant.
\end{Cor}

\section{The labeled preordered set of the $\D$-classes of
  $S(\Cl X)$}\label{sec:label-preord-set}

In this section we use Theorem~\ref{t:main-D-cate-AYF}
and a result from~\cite{newpaper}
to obtain a flow equivalence invariant
(Theorem~\ref{t:isomorphism-of-preordered-sets-with-ranks-first-version})
which, as we shall observe, improves some related results
from~\cite{Beal&Fiorenzi&Perrin:2005a,Costa:2006}.
We close the section with some examples.

\subsection{Abstract semigroup setting}

Given a semigroup $S$, let $\mathfrak {D}(S)$ be the set
of \D-classes of $S$.
Endow $\mathfrak {D}(S)$ with the preorder
$\preceq$ such that, if $D_1$ and $D_2$ belong to $\mathfrak {D}(S)$,
then $D_1\preceq D_2$ if and only if
there are $d_1\in D_1$ and $d_2\in D_2$ such that $d_2$
is a factor of $d_1$.
Note that if $\ci D=\ci J$ (for example, if $S$ is finite), then
$\preceq$ is a partial order.

If we assign to each element $x$ of a preordered set $P$
a label $\lambda_P(x)$ from some set, we obtain a new structure,
called \emph{labeled preordered set}.
A morphism in the category of
labeled preordered sets
is a morphism $\varphi\colon P\to Q$ of preordered sets
such $\lambda_Q\circ\varphi=\lambda_P$.

Suppose that the semigroup $S$ acts on a set $Q$.
Then each element of $S$ can be viewed as a transformation
on $Q$.
The \emph{rank} of a transformation is the cardinality of its
image. Elements of $S$ which are $\D$-equivalent have
the same rank, as transformations of $Q$ (cf.~\cite[Proposition 4.1]{Lallement:1979}.)
We define a labeled preordered set
$\mathfrak {D}_\ell\bigl(S)$ as follows. The underlying preordered set is the set of $\D$-classes of $LU(S)$.  The label of a $\D$-class $D$ is $(\varepsilon, H,r)$ where $\varepsilon\in\{0,1\}$, with $\varepsilon=1$ if and only if $D$ is regular, $H$ is the Sch\"utzenberger group of $D$ and $r$ is the rank
of an element of $D$, viewed as a mapping on $Q$ via the action of $S$.
Denote by $\mathfrak {D}_Q\bigl(S)$ the resulting labeled poset.  The following result will be proved in~\cite{newpaper}.

\begin{Thm}\label{t:preordered-sets-with-ranks}
  If $S$ is a semigroup with zero acting on a pointed set $Q$, then $\mathfrak {D}_Q\bigl(S)$ is a Karoubi invariant of $(Q,S)$.
\end{Thm}

\subsection{Application to shifts}

Consider a shift $\Cl X$.
 Recall that $S(\Cl X)$ acts faithfully as a transformation semigroup
 on the set $Q(\Cl X)$,
 and that if $\Cl X$ is synchronizing then the restriction
 of this action to $Q_{\mathfrak F}(\Cl X)$ is also
 faithful.

 The sink state $\emptyset$ is not a vertex of the Krieger cover of
$\Cl X$. Therefore, following analogous conventions found
in~\cite{Beal&Fiorenzi&Perrin:2005a,Costa:2006},
we let $\mathfrak {KD}(\Cl X)$  be the labeled preordered set obtained from $\mathfrak D_{Q(\Cl X)}(LU(\Cl X))$,
by replacing  the label $(\varepsilon,H,r)$
of each $\D$-class $D$ by the label $(\varepsilon,H,r-1)$.

Similarly, if $\Cl X$ is synchronizing,
then we replace
in $\mathfrak D_{Q_{\mathfrak F}(\Cl X)}(LU(\Cl X))$ the label $(\varepsilon,H,r)$
of each $\D$-class $D$ by $(\varepsilon,H,r-1)$, and
denote  the resulting labeled preordered
set by $\mathfrak {FD}(\Cl X)$. An immediate consequence of
  Theorems~\ref{t:main-D-cate-AYF} and~\ref{t:preordered-sets-with-ranks} is then the following.

\begin{Thm}\label{t:isomorphism-of-preordered-sets-with-ranks-first-version}
  For every shift $\Cl X$,
  the labeled preordered  set
   $\mathfrak {KD}(\Cl X)$
   is a flow equivalence invariant.
   If $\Cl X$ is synchronizing then
   $\mathfrak {FD}(\Cl X)$
   is also a flow equivalence invariant.
\end{Thm}

Theorem~\ref{t:isomorphism-of-preordered-sets-with-ranks-first-version} improves on a weaker conjugacy invariant introduced
in the doctoral thesis of the first author~\cite{Costa:2007}.
The invariance
under conjugacy of
$\mathfrak {KD}(\Cl X)$
when $\Cl X$ is sofic
(and of
$\mathfrak {FD}(\Cl X)$ if $\Cl X$ is moreover irreducible)
was first proved in~\cite{Costa:2006}. Forgetting the non-regular $\D$-classes,
we extract
  the weaker conjugacy invariants of sofic shifts, first obtained in~\cite{Beal&Fiorenzi&Perrin:2005a}.

\subsection{Examples}

\begin{Example}\label{eg:complicated-example}
Let \Cl X and \Cl Y be the
irreducible sofic
shifts on
 the $14$-letter alphabet $\{a_1,\ldots,a_{14}\}$
  whose right Fischer covers
  are respectively represented in Figure~\ref{fig:difficult-example}.
  The dashed edges are those whose label appears only in one edge.
    \begin{figure}[h]
    \includegraphics[scale=0.85,angle=0]{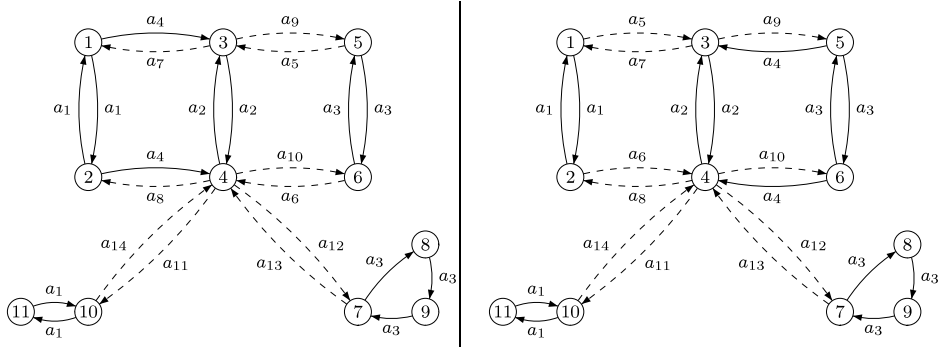}
  \caption{The irreducible sofic shifts \Cl X and \Cl Y
    in Example~\ref{eg:complicated-example}.}
    \label{fig:difficult-example}
  \end{figure}
  The labeled ordered sets
  $\mathfrak {KD}(\Cl X)$
  and
  $\mathfrak {KD}(\Cl Y)$
  are represented in Figure~\ref{fig:lpijsi-x-y}
  by their Hasse diagrams,
  where within each node of the diagram we find information
  identifying the $\D$-class and its label.
  For example, in the first diagram, the notation
  $a_2|(1,C_2,3)$ means the node represents
  the $\D$-class of $[a_2]_\Cl X$ and
  that its label is the triple $(1,C_2,3)$, where,
  as usual, $C_n$ denotes the cyclic group of order $n$.
  The computations were carried out using \pv{GAP}~\cite{GAP4:2006,Delgado&Linton&Morais:automata,Delgado&Morais:sgpviz}.
  The labeled ordered sets
  $\mathfrak {KD}(\Cl X)$
  and
  $\mathfrak {KD}(\Cl Y)$
  are not isomorphic, hence \Cl X and \Cl Y are not flow
  equivalent.

\begin{figure}[h]
  \centering
    \includegraphics[scale=0.85,angle=0]{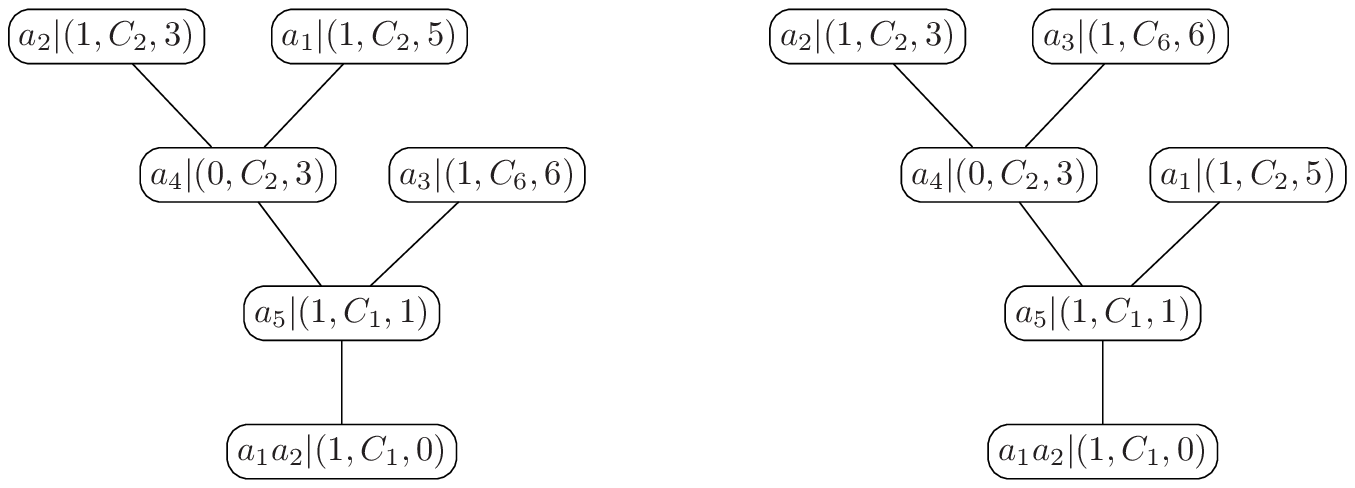}
  \caption{
  $\mathfrak {KD}(\Cl X)$
  and $\mathfrak {KD}(\Cl Y)$.}
  \label{fig:lpijsi-x-y}
\end{figure}
\end{Example}

\begin{Rmk}
  In Example~\ref{eg:complicated-example},
  the Fischer covers of $\Cl X$
  and $\Cl Y$ have the same  underlying labeled
  graph, as seen in Figure~\ref{fig:difficult-example}.
   The
  left Fischer covers, and the right and left Krieger covers
  of $\Cl X$ and $\Cl Y$,
  also have respectively the same  underlying labeled
  graph. The interest in this remark is
  that Rune~Johansen announced in his doctoral thesis~\cite{Johansen:2011b}
  a result of
  Boyle, Carlsen and Eilers~\cite{Boyle&Carlsen&Eilers} stating that if two sofic shifts are flow
  equivalent then the edge shifts defined by the underlying
  unlabeled
  graphs of their right (left) Krieger covers are also flow equivalent, in a
  canonical way~(cf.~\cite[Theorem 2.11 and Proposition
  2.12]{Johansen:2011b}), and
  in the case of irreducible sofic shifts, the same happens
  with the edge shifts defined by the right (left) Fischer covers.
\end{Rmk}

  \begin{Example}\label{eg:two-monoids}
    Let $\Cl X$ and $\Cl Y$ be the irreducible sofic shifts
    whose Fischer covers are respectively
    presented in Figure~\ref{fig:monoidal-shifts}.
            \begin{figure}[h]
          \includegraphics[scale=0.85,angle=0]{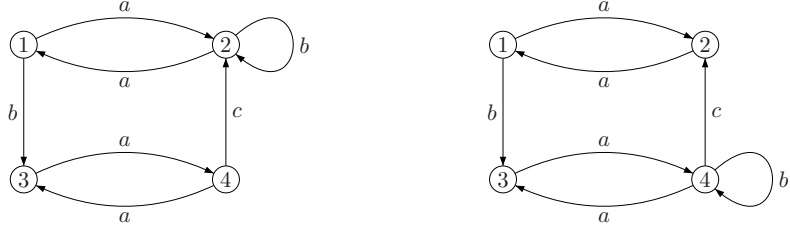}
  \caption{Two irreducible shifts $\Cl X$ and  $\Cl Y$.}
    \label{fig:monoidal-shifts}
  \end{figure}
    Note that $a^2$ fixes all vertices, whence $S(\Cl X)$ and $S(\Cl Y)$ are monoids.
    The $\D$-class of the identity is at the top of the Hasse diagram
    of the $\D$-classes, and so,
    by~Theorem~\ref{t:isomorphism-of-preordered-sets-with-ranks-first-version},
    flow equivalent shifts whose syntactic semigroup is a
    monoid have the same number of vertices in the Krieger cover,
    and also in the Fischer cover if they are synchronizing.
 Using
 \pv{GAP}~\cite{GAP4:2006,Delgado&Linton&Morais:automata},
 one can check that the right Krieger covers of $\Cl X$ and $\Cl Y$ have the same
 proper communication graph, but distinct number of
    vertices ($7$ and $6$, respectively.) Hence,
  $\mathfrak {KD}\bigl(S(\Cl X)\bigr)\ncong \mathfrak
  {KD}\bigl(S(\Cl Y)\bigr)$, and so
  $\Cl X$ and $\Cl Y$ are not flow equivalent by
  Theorem~\ref{t:isomorphism-of-preordered-sets-with-ranks-first-version}.
  \end{Example}

 Clearly, two equivalent categories have the same local monoids, up to
 isomorphism.
 Therefore, from Theorem~\ref{splittingisinvflow},
 we immediately extract the following criterion for sofic shifts
 whose syntactic semigroup is a monoid.

 \begin{Prop}\label{p:flow-criterion-for-monoidal-case}
   Let $\Cl X$ and $\Cl Y$ be flow equivalent shifts.
   If $S(\Cl X)$ is a monoid then $S(\Cl X)$ embeds into $S(\Cl Y)$.
   In particular, if $S(\Cl X)$ and $S(\Cl Y)$ are both finite monoids, then
   $S(\Cl X)$ and $S(\Cl Y)$ are isomorphic.\qed
 \end{Prop}

 \begin{Example}
   We return to the shifts from Example~\ref{eg:two-monoids}.
   Simple computations with \pv{GAP} show that the monoids $S(\Cl X)$ and $S(\Cl Y)$
   are not isomorphic and so
   Proposition~\ref{p:flow-criterion-for-monoidal-case}
  again shows that $\Cl X$ and $\Cl Y$ are not flow equivalent.
 \end{Example}

 Next we give an example concerning non-sofic synchronizing shifts.

 \begin{Example}
   In Figure~\ref{fig:synchronizing-example}
   one sees the
   Fischer covers of two sofic shifts, respectively
   $\Cl X$ and $\Cl Y$, both having
   $bc$ as a synchronizing word (for checking that these are Fischer covers,
   one may use~\cite[Theorem 2.16]{Fiebig&Fiebig:1991}.)
                  \begin{figure}[h]
          \includegraphics[scale=0.85,angle=0]{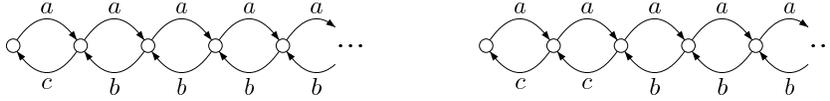}
  \caption{Fischer covers of two synchronizing shifts.}
    \label{fig:synchronizing-example}
  \end{figure}
   If $u\in L(\Cl X)$, then the
   rank of the action of $[u]_\Cl X$ in $\mathfrak F(\Cl X)$
   is one or infinite,
   depending on whether $c$ is a factor of $u$ or not.
   On the other hand,
   the idempotent $[av]_\Cl Y$ has rank two. Therefore
   $\Cl X$ and $\Cl Y$ are not flow equivalent, by
   Theorem~\ref{t:isomorphism-of-preordered-sets-with-ranks-first-version}.
 \end{Example}

\section{Classes of sofic shifts defined by classes of
  semigroups}\label{sec:classes-sofic-shifts}

Given a class $\pv V$ of semigroups with zero closed under isomorphism, we denote by
$\ci S(\pv V)$ the class of
shifts~$\Cl X$ such that $S(\Cl X)$ belongs to
$\pv V$.
The \emph{localization of\/ $\pv V$}, denoted by~$\L\pv V$, is the
class of semigroups $S$ whose local monoids $eSe$ (with
$e\in E(S)$) belong to~$\pv V$.
We mention that it follows easily from known
results that an irreducible shift $\Cl X$ is
of finite type if and only if
$S(\Cl X)\in \L\pv {Sl}$ (cf.~\cite[Proposition 4.2]{Costa:2006b}) where $\pv {Sl}$ is the class of finite idempotent and commutative monoids.

Since the local monoids of $S$ are precisely the endomorphism monoids of the category $\KK (S)$, the following is immediate.

\begin{Thm}\label{t:localV-invariant-under-flow-equivalence}
  If $\pv V$ is a class of monoids with zero closed under isomorphism,
  then the class $\ci S(\L\pv V)$ is closed under
  flow equivalence.
\end{Thm}

More generally,
one can use Theorem~\ref{splittingisinvflow}
to show that the classes proved in~\cite{Costa:2006b}
to be closed under conjugacy are actually closed under flow equivalence.

Theorem~\ref{t:localV-invariant-under-flow-equivalence}
provides a method to show that some
natural classes of sofic shifts
are closed under flow equivalence.
As an example, let us consider
the class of \emph{almost finite type} shifts.
For background and motivation
see~\cite[Section 13.1]{MarcusandLind}.
The shifts in Examples~\ref{eg:complicated-example}
and~\ref{eg:two-monoids} are almost finite type shifts.
Denote by $\ci S_I(\pv V)$
the intersection of $\ci S(\pv V)$
with the class of irreducible shifts.
It turns out that
the class of almost finite type shifts is
$S_I(\L\, \pv {ECom})$,
where $\pv {ECom}$ is the class of finite monoids with commuting
idempotents~\cite{Costa:2006b}. Hence,
from Theorem~\ref{t:localV-invariant-under-flow-equivalence} we
deduce the next result, first proved in~\cite{Fujiwara&Osikawa:1987}.

\begin{Thm}
The class of almost finite type shifts is closed under flow
equivalence.\qed
\end{Thm}

In~\cite{Beal:1993,Beal&Fiorenzi&Perrin:2005a} a class of sofic shifts called \emph{aperiodic shifts} is studied.
In~\cite{Beal&Fiorenzi&Perrin:2005a}
it is proved that the class of irreducible aperiodic
shifts is~$S_I(\pv A)$,
where $\pv A$ is the class of finite \emph{aperiodic semigroups} (a semigroup
is \emph{aperiodic} if the group of units of each of its local monoids is trivial),
and that this class is closed under conjugacy.
Since $\pv A=\L\pv A$
we obtain the following improvement.

\begin{Thm}
The class of irreducible aperiodic shifts is closed under flow equivalence.\qed
\end{Thm}

\section{Markov-Dyck and Markov-Motzkin shifts}\label{sec:markov-dyck-shifts}

The examples we have given so far were mainly of sofic or synchronizing
shifts.
In this section
we apply Theorem~\ref{splittingisinvflow} to classify with respect to flow
equivalence two classes of
shifts,  introduced and studied by W.~Krieger and his collaborators~\cite{Krieger&Matsumoto:2011,Kriger:2006,Krieger:2000,Krieger:1974,Hamachi&Inoue&Krieger:2009}, that in general are non-sofic and
non-synchronizing.
These shifts, called Markov-Dyck and Markov-Motzkin shifts,
are the main subject of the article~\cite{Krieger&Matsumoto:2011}.
Previously, they had appeared as special cases of more general
constructions in~\cite{Kriger:2006,Krieger:2000,Krieger:1974,Hamachi&Inoue&Krieger:2009}.  These shifts are defined in terms of graph inverse semigroups.

An inverse semigroup $S$ is a semigroup such that for all $s\in S$,
there is a unique element $s^*$ such that both $ss^*s=s$ and
$s^*ss^*=s^*$ hold.  Note that $ss^*,s^*s$ are then both idempotents.
Idempotents of an inverse semigroup commute and hence form a
subsemigroup which is a semilattice.
In an inverse semigroup, one has that $s\mathrel{\L} t$ (respectively, $s\mathrel{\R} t$) if and only if $ss^*=tt^*$ (respectively, $s^*s=t^*t$.) Consequently, one has $ss^*\mathrel{\D} s^*s$.
See~\cite{Petrich:1984,Paterson} for more on inverse semigroups.

Graph inverse semigroups, in the special case of graphs with no
multiple-edges, were first considered by Ash and Hall~\cite{AshHall}.
They were considered in greater generality
in~\cite{Paterson,graphinverse}.
The recent paper~\cite{Jones&Lawson:2014} is a useful reference.

Let $G$ be a (directed) graph and let $G^\ast$ be the free category generated
by $G$.  The objects of $G^\ast$ are the vertices of $G$ and the morphisms are (directed) paths, including an empty path at each vertex.  Since
we are adopting the Category Theory convention for the composition of
morphisms, in this context a path of $G$ should be understood
as follows:
a non-empty finite sequence of edges $x_1\cdots
x_n$ is a path of $G$ if the domain of $x_i$ is the range of
$x_{i+1}$, for $1\leq i<n$.
Of course, the empty path at a vertex $q$, denoted $1_q$, is the identity of
$G^\ast$ at $q$.
The domain and range of a path $u$ will be denoted respectively
by $\alpha u$ and $\omega u$.

Associated to $G$ is an inverse semigroup $P_G$, called the \emph{graph inverse semigroup
of $G$}. The underlying set of $P_G$ is the set of all pairs $(x,y)$
of morphisms $x,y\in G^\ast$
such that $x$ and $y$ have a common domain,
together with an extra element $0$,
which is the zero element of $P_G$.
The pair $(x,y)$ is usually denoted $xy^{-1}$.
We think on $y^{-1}$ as a formal inverse of the path~$y$,
obtained by reversing the directions of the edges in $y$.
We make the identifications
$y=(y,1_{\alpha y})$ and $y^{-1}=(1_{\alpha y},y)$.
Moreover, we have $1_q^{-1}=1_q$, for every vertex $q$ of $G$.
The multiplication between elements of $P_G\setminus \{0\}$ is given by the following rules:
\begin{equation*}
  xy^{-1}\cdot uv^{-1}=
  \begin{cases}
    xzv^{-1}&\text{if $u=yz$ for some path $z$},\\
    x(vz)^{-1}&\text{if $y=uz$ for some path $z$},\\
    0&\text{otherwise}.
  \end{cases}
\end{equation*}

The semigroup $P_G$ is an inverse semigroup
in which $(xy^{-1})^* = yx^{-1}$ and $0^*=0$.
Its non-zero idempotents are
the pairs of the form $xx^{-1}$, with $x\in G^\ast$. Note that $1_q$
is idempotent, for every vertex $q$ of $G$, and
$x^{-1}x=1_{\alpha x}$ for every $x\in G^\ast$.

Let $S$ be an inverse semigroup with zero and let $\Sigma$ be a set of
(semigroup) generators for $S$.  Let $\pi\colon \Sigma^*\to S$ be the
canonical projection.  Then $L=\pi\inv(S\setminus \{0\})$ is a
factorial and prolongable language and hence we have an associated
shift (cf.~\cite[Proposition 1.3.4]{MarcusandLind}.)

From now on, we assume that the graph $G$ is finite.
Let $\Sigma_G$ be the set of elements of $P_G$ of the form $x$ or $x^{-1}$, with $x$ an edge
of $G$. In general, $P_G$ is generated
by the union $\Lambda_G$ of $\Sigma_G$ and
the set of empty paths $1_q$.  Let
$\rho\colon (\Lambda_G)^+_0\to P_G$ be the canonical projection.  The \emph{Markov-Motzkin shift} $M_G$ is the shift with language $\rho\inv (P_G\setminus \{0\})$.

If the out-degree of each vertex is at least $1$, then $P_G$ is
generated by $\Sigma_G$ because if $q$ is a vertex and $x$ is an
out-going edge from $q$, then $1_q=x^{-1}x$.  We may therefore
consider the unique onto homomorphism $\pi\colon (\Sigma_G)_0^+\to P_G$
extending the identity map.
The \emph{Markov-Dyck shift} $D_G$
is the shift with language $L(D_G)=\pi^{-1}(P_G\setminus\{0\})$.

Note that the syntactic semigroups of $D_G$ and $M_G$
are homomorphic images of~$P_G$, as $P_G$ recognizes $L(D_G)$ and $L(M_G)$.
Let us now characterize when $P_G$ is the syntactic semigroup of $M_G$ and
$D_G$.
An inverse semigroup~$S$ is called
\emph{fundamental}, or an \emph{antigroup}, if the only congruence
contained in the $\H$-relation is the trivial one.

 \begin{Lemma}\label{l:inversesyntactic}
 Let $S$ be an inverse semigroup with zero.
 Then $S$ is $0$-disjunctive
 if and only if $S$ is fundamental and $E(S)$ is
 $0$-disjunctive.
 \end{Lemma}

 Lemma~\ref{l:inversesyntactic} is Lemma IV.3.10 from \cite{Petrich:1984}.
 It is instrumental in the following proof of
the desired characterization of $M_G$ and $P_G$.

\begin{Lemma}\label{l:PG-is-syntactic}
  Let $G$ be a finite graph.  Then $P_G$ is the syntactic semigroup of
  $M_G$ if and only if $G$ has no vertex of in-degree exactly one.
  If, in addition,
  each vertex has out-degree at least one, then $P_G$ is the syntactic semigroup of $D_G$ if and only if $G$ has no vertex of in-degree exactly one.
\end{Lemma}
\begin{proof}
  In view of Lemma~\ref{l:0-disjunctive-are-syntactic-semigroups},
  it suffices to check when $P_G$ is $0$-disjunctive.
By Lemma~\ref{l:inversesyntactic}, this occurs if and
only if $P_G$ is fundamental and $E(P_G)$ is $0$-disjunctive.
It turns out that the $\H$-relation in $P_G$ is trivial~\cite[Theorem
2.1]{Jones&Lawson:2014}, and so $P_G$ is fundamental. On the other
hand, it is shown in~\cite[Lemma~2.20]{Jones&Lawson:2014} that $E(P_G)$
is $0$-disjunctive\footnote{The definition of $0$-disjunctive
  semilattice given in \cite[Lemma~2.20]{Jones&Lawson:2014}
  appears in a different form, which is however equivalent to the one
  we are using, see \cite[Exercise IV.3.13(ii)]{Petrich:1984}.} if and only if $G$ has no vertex of in-degree
exactly one.
\end{proof}

It is easy to see that $P_G$ is finite if and only if $G$ is
acyclic. Therefore, the above lemma implies that the shifts $M_G$ and
$D_G$ are not sofic in general.

In view of our mains results and of Lemma~\ref{l:PG-is-syntactic},
we are naturally interested in investigating the Karoubi envelope
of $P_G$. It is for that purpose, that we introduce the next notions.
Recall that a morphism $f\colon d\to c$
of a category $C$ is said to be a \emph{split monomorphism}
if it has a left inverse, that is, if there is a morphism
$g\colon c\to d$ such that $gf=1_d$.
The composition of two split monomorphisms is a split monomorphism,
and so we can consider the subcategory $L_C$ of $C$ formed by the
split monomorphisms of $C$. Note that
an equivalence $C\to D$
restricts to an equivalence $L_C\to L_D$.

    For a semigroup~$S$,
    a morphism
    $(e,s,f)$ of $\KK(S)$ is a split monomorphism if and only if
    $s\mathrel{\mathcal L}f$~\cite[Proposition 3.1]{Margolis&Steinberg:2012}.
    In particular, if $S$ has a zero, then the only split monomorphisms $(e,s,f)$ with $s=0$ are of the form $(e,0,0)$.  In other words, if we consider the
    full subcategory $\mathbb L(S)$ of $L_{\KK(S)}$
    whose objects are the non-zero idempotents, then there are no morphisms of the form $(e,0,f)$ in this subcategory. Note that an equivalence
    $L_{\KK(S)}\to L_{\KK(T)}$
    restricts to an equivalence
    $\mathbb L(S)\to \mathbb L(T)$ because $0$ is the unique initial object of $L_{\KK(S)}$ and similarly for $L_{\KK(T)}$.

The argument for the following lemma is essentially contained
in~\cite{strongmorita}.
    \begin{Lemma}\label{l:equivalence-to-free-category}
      For every graph $G$, the category $\mathbb L(P_G)$ is equivalent
      to the free category $G^\ast$.
    \end{Lemma}

    \begin{proof}
      An object of $\mathbb L(P_G)$
      is an idempotent of the form
      $uu^{-1}$, with $u\in G^\ast$. Since
      $uu^{-1}\mathrel{\mathcal D} 1_{\alpha u}$,
      the category $\mathbb L(P_G)$ is equivalent
      to the full subcategory $\mathbb L(P_G)'$
      whose objects are the idempotents of the form
      $1_q$.
      If $u,v\in G^\ast$ are such that
      $(1_q,uv^{-1},1_r)$ is a split monomorphism of $\KK(P_G)$,
      then $1_r=vu\inv uv\inv=vv\inv$ and so $v=1_r$.
      Therefore,
      $\mathbb L(P_G)'$ is isomorphic to~$G^\ast$.
    \end{proof}

It is well known that two free categories $G^*$ and $H^*$ on graphs $G,H$ are equivalent if and only if $G$ and $H$ are isomorphic.  Since we don't know a precise reference, we sketch a proof.

\begin{Lemma}\label{l:freecatequiv}
Suppose that $G,H$ are graphs with $G^*$ equivalent to $H^*$.
Then $G$ and~$H$ are isomorphic.
\end{Lemma}
\begin{proof}
In a free category, there are no isomorphisms other than the
identities. Hence, two functors with codomain a free category are isomorphic if and only if they are equal.  It follows that any equivalence $F\colon G^*\to H^*$ is actually an isomorphism.  A morphism $u$ of a free category is an edge if and only if it cannot be factored $u=vw$ with $v,w$ non-empty paths.  Thus $F$ must restrict to a graph isomorphism $G\to H$.
\end{proof}

As a corollary of the preceding two lemmas, we deduce that Morita equivalent graph inverse semigroups have isomorphic underlying graphs.

\begin{Cor}\label{c:Moritaequivforgraphsemigroups}
Let $G,H$ be graphs.  Then $P_G$ is Morita equivalent to $P_H$ if and only if $G$ and $H$ are isomorphic.
\end{Cor}

The case where $G$ is a finite one-vertex graph with at least $2$
loops has received special attention in the literature and motivates the
general case. If $N\ge 2$ is the number of loops of the one-vertex
graph $G$, then
the corresponding Markov-Dyck shift is
denoted $D_N$, and is called the \emph{Dyck shift of rank $N$}.
In~\cite{Matsumoto:2011} it is proved, by means of the computation of
certain flow invariant abelian groups, that if $D_N$ and $D_M$
are flow equivalent then $N=M$.
This is proved here again, as a special case
of the next result.

\begin{Thm}\label{t:Markov-Dyck-shift}
  Let $G$ and $H$ be finite graphs such that the in-degree of none of their vertices is exactly one. Then the Markov-Motzkin shifts $M_G$ and $M_H$ are flow equivalent if and only if $G$ and $H$ are isomorphic.  If in addition, the out-degree of each vertex of $G$ and $H$ is at least one, then the Markov-Dyck shifts
  $D_G$ and $D_H$ are flow equivalent if and only if $G$ and $H$
  are isomorphic.
\end{Thm}

\begin{proof}
  Sufficiency is obvious. Conversely,  suppose that $M_G$ and $M_H$ are
  flow equivalent.  By Lemma~\ref{l:PG-is-syntactic}, $P_G$ and $P_H$ are the syntactic semigroups of $M_G$ and $M_H$, respectively.  They are inverse semigroups and hence have local units.  Theorem~\ref{splittingisinvflow} implies that $P_G$ and $P_H$ are Morita equivalent. Corollary~\ref{c:Moritaequivforgraphsemigroups} then yields $G\cong H$. The proof for Markov-Dyck shifts is identical.
\end{proof}

Theorem~\ref{t:Markov-Dyck-shift} generalizes results of
Krieger~\cite{Krieger:2000,Kriger:2006}, who showed that if $G$
and
$H$ are
strongly connected and each vertex of $G$ and $H$ has in-degree at
least $2$, then $D_G$ is conjugate to $D_H$ if and only if $G$ and $H$
are isomorphic. Krieger's proof relies
on the invariance of a semigroup that is studied in the next section.

\section{Shifts with Krieger's Property ($\mathscr A$)}
\label{sec:shifts-with-kriegers}

  Wolfgang Krieger introduced in~\cite{Krieger:2000} (see also~\cite{Krieger:2012,Hamachi&Krieger:2013,Hamachi&Krieger:2013b}) the class of
  shifts with  property ($\mathscr A$) whose definition
  will be recalled in this section.
  To each shift with property ($\mathscr A$), Krieger associated a semigroup which he
  showed to be a conjugacy invariant.
  Recently,
  Krieger announced that property ($\mathscr A$) is a flow
   equivalence invariant as is the corresponding semigroup.
  In this section, we give an alternative characterization
  of the Krieger semigroup in terms of the Karoubi envelope,
  from which we deduce a new proof of its flow invariance,
  under a condition usually assumed when studying property~($\mathscr A$).

\subsection{Non-zero divisors}

  Let $C$ be a small category. Two morphisms $f$ and $g$ of $C$ are
  \emph{isomorphic}, denoted $f\cong g$, when there are
  isomorphisms $\varphi$ and $\psi$ of $C$ such that
  $f=\varphi g\psi$.
   The relation $\cong$ is an equivalence relation
   on the set of morphisms of $C$
   satisfying the property that
   if $f\cong g$ and  $f'\cong g'$, with
   the compositions $ff'$ and $gg'$ defined, then $ff'\cong gg'$.

   Suppose moreover that $C$ is a category with zero.
   We may then consider the equivalence relation
   $\cong_0$ refining $\cong$ by identifying all zero morphisms into a
   single partition class, leaving intact the $\cong$-classes
   of non-zero morphisms.  Denote by $C/{\cong_0}$ the set of equivalence classes
   defined by $\cong_0$.
   We use the notation $\langle f\rangle$ for the
   $\cong_0$-equivalence class of $f$.
   In the absence of ambiguity,
   the $\cong_0$-class of the zero morphisms is denoted~$0$.

   A category in which each
of its hom-sets has at most one element is essentially the same thing as a preorder~\cite{MacLane:1998} and so we call them preorders.

  \begin{Def}\label{def:semigroup-CT}
    Consider a small category $C$ with zero.
    Let $T$ be a subcategory of $C$
    which is a preorder and contains the non-zero isomorphisms of $C$.
  We define a binary  operation
  $\circ_T$
  on $C/{\cong_0}$
  as follows (note that
  the operation is well defined
  because $T$ contains the non-zero isomorphisms):
  \begin{enumerate}
  \item  For every $x$ in
  $C/{\cong_0}$ we have $0\circ_T x=x\circ_T 0=0$.
  \item Suppose $f_1$ and $f_2$ are non-zero morphisms of $C$.
      If there is a (necessarily unique) morphism $h\colon r(f_2)\to
      d(f_1)$ in $T$, then $\langle f_1\rangle \circ_T \langle f_2\rangle= \langle f_1hf_2\rangle$, otherwise $\langle f_1\rangle \circ_T \langle f_2\rangle=0$.
  \end{enumerate}
  The operation $\circ_T$ endows $C/{\cong_0}$ with
  a structure of semigroup with zero.
  We denote by $C_{\circ_T}$ this semigroup.
  \end{Def}

For a small category $C$ with zero, a non-zero morphism $f$ of $C$ is
  called a \emph{non-zero divisor} if $gf\neq 0$ and $fh\neq 0$ whenever
  $g,h$ are non-zero morphisms with $d(g)=r(f)$
  and $d(f)=r(h)$. The non-zero divisors form a
  subcategory $C_{\pv{nzd}}$ of $C$
  containing all non-zero isomorphisms of $C$.
  If $T=C_{\pv{nzd}}$ is a preorder, then we denote the operation
    $\circ_T$ by $\bullet$, and the semigroup $C_{\circ_T}$
    by~$C_{\bullet}$.
The proof of the following proposition is a routine exercise.

\begin{Prop}\label{p:invariance-of-the-krieger-semigroup}
 Let $C$ and $D$ be equivalent small categories with zero.
 Then $C_{\pv {nzd}}$ is
 equivalent to $D_{\pv {nzd}}$.
 In particular, $C_{\pv {nzd}}$
 is a preorder if and only if
 $D_{\pv {nzd}}$ is a preorder.
 If $C_{\pv {nzd}}$ and $D_{\pv {nzd}}$
 are preorders, then the semigroups $C_{\bullet}$
 and $D_{\bullet}$ are isomorphic.
\end{Prop}

Let $S$ be a semigroup with zero. Let us say that
       a non-zero morphism $(e,s,f)$ of $\KK(S)$
       is a \emph{strong non-zero divisor} if
       we have $rs\neq 0$ and $st\neq 0$
       whenever $r$ and $t$ are elements of $S$
       such that $re\neq 0$
       and $ft\neq 0$.
      The strong non-zero divisors
  of $\KK(S)$ form a subcategory $T=\KK(S)_{\pv {snzd}}$.
   If $T$ is a preorder, then we denote $\circ_T$ by $\circ$,
   and $\KK(S)_{\circ_T}$ by~$\KK(S)_{\circ}$.
   We shall see in this section that
the
semigroup associated by Krieger to a shift $\Cl X$ with property ($\mathscr A$) is $\KK(\Cl X)_\circ$
(Theorem~\ref{t:characterization-of-the-Krieger-associated-semigroup}.)

   Note that a strong-non zero divisor morphism is a non-zero divisor.
   In the next proposition, whose easy proof we omit, a sufficient
   condition for the converse is given.

  \begin{Prop}\label{p:sufficient-conditions-for-catrobust-equal-robust}
    Suppose $S$ is a semigroup with zero
   $S$ satisfying the following condition:
   for every $s\in S\setminus\{0\}$, there are
    $r,t\in S$ and idempotents $e,f\in S$ such that $erstf\neq 0$.
    Then we have $\KK(S)_{\pv{snzd}}=\KK(S)_{\pv{nzd}}$.
  \end{Prop}

  \begin{Rmk}\label{rmk:sufficient-conditions-for-catrobust-equal-robust}
       Suppose that $\Cl X$ is a
       sofic subshift of $A^{\mathbb Z}$.
       All sufficiently long words of $A^+$
       have a factor that maps to an idempotent in
       $S(\Cl X)$ (cf.~\cite[Theorem 1.11]{PinBook}),
       from which it follows that $S(\Cl X)$
       satisfies the condition in
       Proposition~\ref{p:sufficient-conditions-for-catrobust-equal-robust},
       thus $\KK(\Cl X)_{\pv{snzd}}=\KK(\Cl Y)_{\pv{nzd}}$.
  \end{Rmk}

\subsection{The set \protect{$\mathcal A(\Cl X)$}}

In this section, $\Cl X$ is always a subshift of $\z A$. For each $u\in A^+$, let
\begin{align*}
  \omega^+(u)
  &=\{v\in L(\Cl X)\mid wu\in L(\Cl X)\Rightarrow wuv\in L(\Cl X)\},\\
  \omega^-(u)
  &=\{v\in L(\Cl X)\mid uw\in L(\Cl X)\Rightarrow vuw\in L(\Cl X)\}.
\end{align*}
Consider also the following subsets of $\Cl X$:
\begin{align*}
   \mathcal A_n^+(\Cl X)
   &=\{x\in\Cl X\mid \forall i\in\mathbb Z,\,
   x_i\in\omega^+(x_{[i-n,i-1]})\},\\
   \mathcal A_n^-(\Cl X)
   &=\{x\in\Cl X\mid \forall i\in\mathbb Z,\,
   x_i\in\omega^-(x_{{[}i+1,i+n{]}})\}.
\end{align*}

\begin{Rmk}\label{r:glueing-ahead}
  Suppose that $uv\in L(\mathcal A_n^+(\Cl X))$
  and that $u$ has length at least~$n$. Then $wu\in L(\Cl X)$ implies
  $wuv\in L(\Cl X)$. A dual remark holds for
  $\mathcal A_n^-(\Cl X)$.
\end{Rmk}

If non-empty, the intersection $\mathcal A_n(\Cl X)=\mathcal A_n^+(\Cl
X)\cap\mathcal A_n^-(\Cl X)$ is an $n$-step finite type subshift of $\z A$.
Note that $n\leq m$ implies $\mathcal A_n(\Cl X)\subseteq \mathcal A_m(\Cl X)$.
Finally, denote by $\mathcal A(\Cl X)$ the union
$\bigcup_{n\geq 1}\mathcal A_n(\Cl X)$.

\begin{Rmk}\label{r:for-finite-type-shifts-AisX}
If $\Cl X$ is an $n$-step finite shift, then
$\Cl X=\mathcal A_n(\Cl X)=\mathcal A(\Cl X)$.
\end{Rmk}

From hereon, we denote by $\Cl E(\Cl X)$ the set of words
of $A^+$ such that $[u]$ is a non-zero idempotent.
The length of an element $w\in A^*$ is denoted~$|w|$. W.~Krieger, in his talk at the workshop \emph{Flow equivalence of graphs, shifts and $C^\ast$-algebras}, held at the University of
Copenhagen in November 2013 indicated that the following lemma should hold.

\begin{Lemma}\label{l:idempotents-define-nice-periodic-points}
  If $u\in\Cl E(\Cl X)$ then the periodic
  point $u^\infty$ belongs to $\mathcal A_{2|u|}(\Cl X)$.
\end{Lemma}

\begin{proof}
  The fact that $[u]$ is a non-zero idempotent
  guarantees $u^\infty\in\Cl X$.
  Let $a\in A$ and $v\in A^+$ be such that
  $|v|=2|u|$ and $va$ is a factor of $u^\infty$.
  We want to show that $a\in\omega^+(v)$.
  Let $p\in A^+$ be such that $pv\in L(\Cl X)$.
  Since $|v|=2|u|$
  and $va$ is a factor of
  $u^\infty$,
  there is a word $z$ conjugate to $u$
  for which $v=z^2$ and the letter $a$ is a prefix of $z$, whence of $v$.
  Lemma~\ref{l:a-lemma-on-conjugate-words} implies that
  $[v]$ is idempotent, and so
  from $pv\in L(\Cl X)$
  we get $pv^2\in L(\Cl X)$,
  thus $pva\in L(\Cl X)$.
  This shows that $a\in\omega^+(v)$,
  establishing $u^\infty\in\mathcal A_{2|u|}^+(\Cl X)$.
  Dually, we have $u^\infty\in\mathcal A_{2|u|}^-(\Cl X)$.
\end{proof}

Two elements $x$ and $y$ of $\z A$
are \emph{left asymptotic} if there is $n\in\ZZ$
such that $x_i=y_i$ for all $i\leq n$.
 Of course, there is the dual
notion of \emph{right asymptotic}.
If $x$ is left  asymptotic
to $p^{\infty}$
and right asymptotic to $q^{\infty}$, then there is $u\in A^+$ such that $x$ is
in the orbit of
the bi-infinite sequence $\ldots ppp.uqqq\ldots$.
We represent by $p^{-\infty}uq^{+\infty}$ this sequence. In absence of
confusion, we may represent also by $p^{-\infty}uq^{+\infty}$ any
sequence in its orbit.

For a subset $Z$ of $\Cl X$, let $P(Z)$ be the set of periodic
points of $\Cl X$ belonging to $Z$.
We define a relation $\preceq $
in $P(\mathcal A(\Cl X))$
by letting $q\preceq r$ if there is $z\in \mathcal A(\Cl X)$
such that $z$ is left asymptotic to $q$
and $z$ is right asymptotic to~$r$.
Since $\mathcal A_n(\Cl X)$ is a finite type shift for every $n$, the relation $\preceq$ is transitive, thus a preorder.
The resulting equivalence
is denoted~$\sim$.

\begin{Prop}\label{p:characterization-of-robust-morphisms}
  Let $p,q\in\Cl E(\Cl X)$ and $u\in A^+$.
  Then $p^{-\infty}uq^{+\infty}$ belongs to $\mathcal A(\Cl X)$
  if and only if
  the morphism
  $([p],[puq],[q])$
  of $\KK(\Cl X)$ is a strong non-zero divisor.
\end{Prop}

\begin{proof}
  Suppose $n\geq 1$ is such that $p^{-\infty}uq^{+\infty}\in\mathcal A_n(\Cl X)$.
  Let $w\in A^+$ with $[wp]\neq 0$.
  As $[p]$ is idempotent,
  we have $wp^n\in L(\Cl X)$.
  Therefore, since $p^{-\infty}uq^{+\infty}\in\mathcal A_n(\Cl X)$,
  we have $wp^nu\in L(\Cl X)$ by Remark~\ref{r:glueing-ahead}.
  Again because $[p]$ is idempotent, we conclude
  that $[wpuq]\neq 0$. Dually,
  if $[qw]\neq 0$
  then $[puqw]\neq 0$. Hence,
  $([p],[puq],[q])$
  is a strong non-zero divisor.

  Conversely, suppose that $([p],[puq],[q])$
  is a strong non-zero divisor.
  By Lemma~\ref{l:idempotents-define-nice-periodic-points},
  there is $n\geq 1$ with $p^\infty,q^\infty\in \mathcal A_n(\Cl X)$.
  Let $m=\max\{n,|puq|\}$.
  Consider a factor of
  $p^{-\infty}uq^{+\infty}$ of the form $za$, with $a\in A$,
  and $|z|=m$.
  We claim that $a\in\omega^+(z)$.
  That holds if $za$ is a factor of $p^\infty$ or $q^\infty$,
  as $|z|\geq n$ and $p^\infty,q^\infty\in \mathcal A_n(\Cl X)$.
  So, suppose that $za$ is neither a factor of $p^\infty$, nor of $q^\infty$. Then, since $|z|\geq |puq|$, there are
  only two possibilities (see Figure~\ref{fig:two-case-for-za}):
  \begin{enumerate}
  \item there are $k\geq 1$ and $\alpha,\beta\in A^\ast$ such that
    $z=\alpha p\beta$ and $\alpha puq^k\in zaA^\ast$;\label{item:za-1}
  \item there are $k\geq 2$ and $\alpha,\beta\in A^\ast$
    such that $z=\alpha q\beta$ and $\alpha q^k\in zaA^\ast$.\label{item:za-2}
  \end{enumerate}
  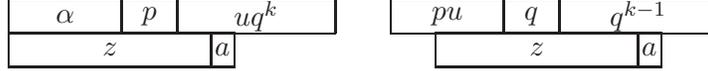
\begin{figure}[h]
    \centering
     \unitlength=0.85pt
     \begin{picture}(300,30)(0,0)
       \put(0,15){\framebox(50,15){$\alpha$}}
       \put(50,15){\framebox(25,15){$p$}}
       \put(75,15){\framebox(70,15){$uq^k$}}
       \put(0,0){\framebox(90,15){$z$}}
       \put(90,0){\framebox(10,15){$a$}}
       \put(170,15){\framebox(50,15){$pu$}}
       \put(220,15){\framebox(25,15){$q$}}
       \put(245,15){\framebox(70,15){$q^{k-1}$}}
       \put(190,0){\framebox(90,15){$z$}}
       \put(280,0){\framebox(10,15){$a$}}
     \end{picture}
         \caption{The two cases where $za$ is not a factor of $p^\infty$ or
    of $q^\infty$.}
  \label{fig:two-case-for-za}
  \end{figure}

  Assume we are in case \eqref{item:za-1},
  and let $x\in A^\ast$ be such that $xz\in L(\Cl X)$.
  Then, as $z=\alpha p\beta$,
  we have $[x\alpha p]\neq 0$. Therefore,
  $[x\alpha puq^k]\neq 0$ because
    $([p],[puq],[q])$ is a strong non-zero divisor morphism of $\KK(\Cl X)$.
    And since $\alpha puq^k\in zaA^\ast$, we
    conclude that $xza\in L(\Cl X)$, thus $a\in\omega^+(z)$.

  Suppose we are in case \eqref{item:za-2},
  and let $x\in A^\ast$ with $xz\in L(\Cl X)$.
  As $z=\alpha q\beta $ and $[q]$ is idempotent,
  we have $x\alpha q^k\beta\in L(\Cl X)$.
  Since $\alpha q^k\in zaA^\ast$,
  it follows that $xza\in L(\Cl X)$, thus $a\in\omega^+(z)$.

  All cases exhausted, we have shown that
  $p^{-\infty} uq^{+\infty}\in\mathcal A_m^+(\Cl X)$.
  Symmetrically, $p^{-\infty} uq^{+\infty}\in\mathcal A_m^-(\Cl X)$ holds.
  Hence, we have
  $p^{-\infty} uq^{+\infty}\in\mathcal A_m(\Cl X)$.
\end{proof}

\begin{Cor}\label{c:a-syntactic-condition-for-p-equivalent-to-q}
  Let $p,q\in\Cl E(\Cl X)$ be such that
  $[p]\mathrel{\D}[q]$.
  Then $p^\infty\sim q^\infty$.
  More precisely, if
  $u,v\in A^+$ are such that
  $[p]=[uv]$
  and
  $[q]=[vu]$
  then
  $p^{-\infty} uq^{+\infty}\in\mathcal A(\Cl X)$
  and
 $q^{-\infty} vp^{+\infty}\in\mathcal A(\Cl X)$.
\end{Cor}

\begin{proof}
  By~Lemma~\ref{l:isomorphism-criterion},
  as $[p]$
  and $[q]$ are idempotents,
 $[p]\mathrel{\D}[q]$
  means
  there are
  $u,v\in A^+$
  such that $[p]=[uv]$ and $[q]=[vu]$. Moreover, $([p],[puq],[q])$
  and $([q],[qvp],[p])$ are isomorphisms,
  whence strong non-zero divisors.
\end{proof}

\subsection{Property ($\mathscr A$)}

\begin{Def}[Property ($\mathscr A$)]\label{def:propertyA}
  Let $n\geq 1$ and $H\geq 0$ be
  integers.  The shift $\Cl X$
  is said to have \emph{property $(\mathscr A,n,H)$}
  when the following conditions hold:
  if $p,u,v,q$ are elements of $A^\ast$
  such that $puq,pvq\in L(\mathcal A_n(\Cl X))$
  and $|p|=|q|=H$, then
  $[puq]_{\Cl X}=[pvq]_{\Cl X}$.
  The shift $\Cl X$ has \emph{property ($\mathscr A$)} if, for every $n\geq 1$, there
  is $H\geq 0$ such that $\Cl X$ has property $(\mathscr A,n,H)$.
\end{Def}

For shifts with property ($\mathscr A$), we have a sort of converse
of Lemma~\ref{l:idempotents-define-nice-periodic-points}, also suggested by W.~Krieger.

\begin{Lemma}\label{l:with-property-A-periodic-points-define-idempotents}
  Suppose that $\Cl X$ has property ($\mathscr A$).
  If $u^\infty\in \mathcal A(\Cl X)$,
  then there is $h_0\geq 1$ such that $u^h\in\Cl E(\Cl X)$ for every $h\geq h_0$.
\end{Lemma}

\begin{proof}
  Let $n\geq 1$ be such that $u^\infty\in \mathcal A_n(\Cl X)$,
  and let $H$ be
  such that $\Cl X$ has property $(\mathscr A,n,H)$.
  Take $h_0=2H+1$.
  Let $h=h_0+d$, where $d$ is a nonnegative integer.
  Then, by the definition of property~($\mathscr A$),
  the words $u^h=u^{H}\cdot u^{d+1}\cdot u^{H}$ and
  $u^{2h}=u^{H}\cdot u^{2H+2d+2}\cdot u^{H}$
  of $L(\mathcal A_n(\Cl X))$
  satisfy $[u^h]=[u^{2h}]$, that is,
  $[u^h]$ is idempotent.
\end{proof}

From
Lemmas~\ref{l:idempotents-define-nice-periodic-points}
and~\ref{l:with-property-A-periodic-points-define-idempotents},
we immediately deduce the following characterization of periodic points of
$\mathcal A(\Cl X)$.

\begin{Cor}\label{c:with-property-A-periodic-points-define-idempotents}
  If $\Cl X$ has property ($\mathscr A$), then
  the periodic points of $\mathcal A(\Cl X)$ are the elements
  of $\Cl X$ of the form $p^{\infty}$ with $p\in\Cl E(\Cl X)$.
\end{Cor}

It follows from
Corollary~\ref{c:with-property-A-periodic-points-define-idempotents}
that if a sofic shift $\Cl X$ has property ($\mathscr A$), then
$P(\mathcal A(\Cl X))$ is dense in $\Cl X$ so long as $\Cl X$ is non-wandering (meaning that for $u\in L(\Cl X)$, there exists $v\in A^*$ such that $uvu\in L(\Cl X)$.)
When studying shifts with property ($\mathscr A$), it is
sometimes assumed in the first place that
$\mathcal A(\Cl X)$ is dense in $\Cl X$~\cite{Hamachi&Krieger:2013,Hamachi&Krieger:2013b}.

For a shift $\Cl X$ with property ($\mathscr A$), isomorphism
between objects of $\KK(\Cl X)$ has a transparent dynamical
description in terms of the equivalence relation $\sim$ between
periodic points of $\mathcal A(\Cl X)$.

\begin{Prop}\label{p:the-orbit-classes-and-the-D-classes}
  Suppose that $\Cl X$ has property ($\mathscr A$).  There is a well-defined bijection
  $\Lambda\colon P(\mathcal A(\Cl X))/{\sim}
  \to (E(S(\Cl X))\setminus\{0\})/{\D}$ given by
  $\Lambda(p^\infty/{\sim})=[p]/{\D}$
for $p\in\Cl E(\Cl X)$.
\end{Prop}

\begin{proof}
  To show that $\Lambda$ is well defined, recall first
  from Lemma~\ref{l:with-property-A-periodic-points-define-idempotents}
  that every element of $P(\mathcal A(\Cl X))$ is of
  the form $p^\infty$ for some $p\in\Cl E(\Cl X)$.

  Now, let $p,q\in\Cl E(\Cl X)$
  with $p^\infty/{\sim}=q^\infty/{\sim}$.
  Then there are $n\geq 1$ and $u,v\in A^+$ such that
  $p^{-\infty}uq^{+\infty}$
  and
  $q^{-\infty}vp^{+\infty}$ belong to $\mathcal A_n(\Cl X)$.
  Let $H$ be such that $\Cl X$ has property $(\mathscr A,n,H)$.
  Since $\mathcal A_n(\Cl X)$ is an $n$-step finite type shift, we have
  $p^{H}uq^nvp^{H}\in L(\mathcal A_n(\Cl X))$.
  Therefore,
  by the definition of property ($\mathscr A$), we have
      $[p^{2H+1}]=[p^{H}uq^nvp^{H}]$.
 Since $[p]$  and
  $[q]$ are idempotents,
  we get $[p]=[puqvp]$,
  and so $[p]\mathrel{\R} [puq]$.
  Similarly, we have
  $[q]=[qvpuq]$
  and $[q]\mathrel{\L} [puq]$.
  We deduce $[p]\mathrel{\D}[q]$,
  showing that $\Lambda$ is well defined.

  That $\Lambda$ is injective follows immediately
  from Corollary~\ref{c:a-syntactic-condition-for-p-equivalent-to-q}.
  Surjectivity is a direct consequence of
  Lemma~\ref{l:idempotents-define-nice-periodic-points}.
\end{proof}

We proceed in the task of relating property ($\mathscr A$) with properties of $\KK(\Cl X)$.

\begin{Prop}\label{p:characterization}
  If a shift $\Cl X$ has property ($\mathscr A$)
  then $\KK(\Cl X)_{\pv{snzd}}$ is a preorder.
  Conversely, if $\Cl X$ is sofic and
  $\KK(\Cl X)_{\pv{snzd}}$ is a preorder, then $\Cl X$ has
  property ($\mathscr A$).
\end{Prop}

\begin{proof}
  Suppose that $([p],[u],[q])$
  and $([p],[v],[q])$ are morphisms
  of $\KK(\Cl X)_{\pv{snzd}}$. By
  Proposition~\ref{p:characterization-of-robust-morphisms},
  there is $n\geq 1$ such that
  $p^{-\infty}uq^{+\infty}$
  and $p^{-\infty}vq^{+\infty}$ belong to $\mathcal A_n(\Cl X)$.
  Let $H$ be such that $\Cl X$ has property $(\mathscr A,n,H)$.
  Since $p^{H}uq^{H}$ and $p^{H}vq^{H}$
  belong to $L(\mathcal A_n(\Cl X))$,
  we have $[p^{H}uq^{H}]=[p^{H}vq^{H}]$
  by the definition of property $(\mathscr A,n,H)$,
  thus $[u]=[v]$, as $[u]=[puq]$ and $[v]=[pvq]$.
  Therefore, $\KK(\Cl X)_{\pv{snzd}}$ is a preorder.

  Conversely, suppose that $\Cl X$ is a sofic shift
  for which $\KK(\Cl X)_{\pv{snzd}}$ is a preorder.
  Fix $n\geq 1$.
  It is well known that if
  $\varphi$ is a homomorphism from $A^+$ onto a finite semigroup $S$,
  then there is an integer $N_{\varphi}$ such that
  every word $v$ with length at least $N_{\varphi}$
  has a factor $w$ with $|w|\geq n$ and $\varphi(w)$
  idempotent (cf.~\cite[Theorem 1.11]{PinBook}.)
  Take $H=N_{\delta_{\Cl X}}$.
  Let $puq\in L(\mathcal A_n(\Cl X))$
  be such that $|p|,|q|\geq H$.
  Then there are factorizations $p=p'\alpha p''$
  and $q=q'\beta q''$,
  with $|\alpha|,|\beta|\geq n$
  and $[\alpha],[\beta]$ idempotents.
  We claim that the morphism $([\alpha],[\alpha p''uq'\beta],[\beta])$
  of $\KK(\Cl X)$ is a strong non-zero divisor. Let $z\in A^+$
  be such that $z\alpha\in L(\Cl X)$.
  Note that $\alpha p''uq'\beta$ is an element of
  $\mathcal A_n(\Cl X)$, as it is a factor of
  $puq$. Hence, as $|\alpha|\geq n$,
  we have $z\alpha p''uq'\beta\in L(\Cl X)$
  by Remark~\ref{r:glueing-ahead}. Therefore, $[z\alpha]\neq 0$ implies
  $[z\alpha p''uq'\beta]\neq 0$, and, dually,
  $[\beta z]\neq 0$ implies
  $[\alpha p''uq'\beta z]\neq 0$, proving the claim that
  $([\alpha],[\alpha p''uq'\beta],[\beta])$ is a strong non-zero divisor.
  Moreover, if $pvq\in L(\mathcal A_n(\Cl X))$,
  then $([\alpha],[\alpha p''vq'\beta ],[\beta])$ is also a strong non-zero divisor.
  By the hypothesis that $\KK(\Cl X)_{\pv{snzd}}$ is a preorder, it follows
  that $[\alpha p''uq'\beta]=[\alpha p''vq'\beta]$,
  whence $[puq]=[pvq]$. This shows
  that $\Cl X$ has property $(\mathscr A,n,H)$.
\end{proof}

Proposition~\ref{p:characterization}
gives an effective decision process for property~($\mathscr A$) for sofic shifts,
solving a problem posed by W.~Krieger at the workshop \emph{Flow equivalence of
graphs, shifts and $C^\ast$-algebras}, held at the University of
Copenhagen in November 2013.
The related problem raised by W.~Krieger at that workshop
as to whether there exist sofic shifts without property ($\mathscr A$) is solved in
the positive, as follows from the next corollary.

\begin{Cor}\label{c:a-necessary-condition}
  If $\Cl X$ is a shift with property ($\mathscr A$),
  then $S(\Cl X)$ is aperiodic.
\end{Cor}

In particular, the even shift~\cite{MarcusandLind} consisting of all bi-infinite words in $\{0,1\}^{\mathbb Z}$ containing no factor of the form $10^{2n+1}1$ with $n\geq 0$ is a sofic shift without property ($\mathscr A$).

\begin{proof}[Proof of Corollary~\ref{c:a-necessary-condition}]
  Every automorphism of $\KK(\Cl X)$
  is a strong non-zero divisor.
  Since $\KK(\Cl X)_{\pv{snzd}}$ is a preorder,
  it follows from Proposition~\ref{p:characterization}
  that the automorphism groups of
  $\KK(\Cl X)$ are trivial.
  By Lemma~\ref{l:isomorphism-criterion},
  this means that the maximal subgroups of $S(\Cl X)$
  are trivial, i.e., $S(\Cl X)$ is aperiodic.
\end{proof}

\subsection{The Krieger semigroup}
\label{sec:krieg-assoc-semigr}

Let $\Cl X$ be a shift with property~($\mathscr A$).
Denote  by $Y(\Cl X)$ the set of elements of $\Cl X$
of the form $p^{-\infty}uq^{+\infty}$, with
$p,u,q\in A^+$ and $p^\infty,q^\infty\in\mathcal A(\Cl X)$.
By
Corollary~\ref{c:with-property-A-periodic-points-define-idempotents},
we can assume that $p,q\in\Cl E(\Cl X)$.
The following definition is from~\cite{Krieger:2000},
and appears also
in~\cite{Krieger:2012,Hamachi&Krieger:2013,Hamachi&Krieger:2013b}.

\begin{Def}\label{def:the-equivalence-on-YX}
For each integer $n\geq 1$,
suppose that $\Cl X$ has property $(\mathscr A,n,H)$.
Consider on $Y(\Cl X)$ the relation $\approx_{n,H}$ defined as follows.
Two elements $x_1,x_2\in Y(\Cl X)$ satisfy $x_1\approx_{n,H} x_2$
when there are
words $p_i,q_i,u_i$
such that $x_i=p_i^{-\infty}u_iq_i^{+\infty}$, with
$p_i^\infty,q_i^\infty\in\mathcal A_n(\Cl X)$, for $i=1,2$, and
the following conditions are satisfied:
\begin{enumerate}
\item $p_1^\infty\sim p_2^\infty$, $q_1^\infty\sim q_2^\infty$;
  \label{item:the-equivalence-on-YX-1}
\item $[z_1]=[z_2]$ whenever $z_1$ and $z_2$ are
  words for which there are factorizations\label{item:the-equivalence-on-YX-2}
  \begin{equation}\label{eq:z1z2}
  z_1=\pi \zeta_1v_1u_1w_1\xi_1\rho,\qquad
  z_2=\pi \zeta_2v_2u_2w_2\xi_2\rho,
  \end{equation}
    with:
    \begin{enumerate}
    \item $|\pi|=|\rho|=|v_i|=|w_i|=H$;\label{item:the-equivalence-on-YX-a}
    \item $p_i^{-\infty}$ ends with $v_i$\label{item:the-equivalence-on-YX-b}
    and $q_i^{+\infty}$ begins with $w_i$;
    \item the words $\pi \zeta_iv_i$\label{item:the-equivalence-on-YX-c}
      and $w_i\xi_i\rho$ belong to $L(\mathcal A_n(\Cl
      X))$.
    \end{enumerate}
  \end{enumerate}
We write $x_1\approx x_2$ if and only if
$x_1\approx_{n,H} x_2$ for some $n,H$ such that $\Cl X$
 has property $(\mathscr A,n,H)$.
\end{Def}

Krieger noted that the relation $\approx$ is a well-defined
equivalence relation.

\begin{Prop}\label{p:translation-of-elements-of-YX}
  Let $\Cl X$ be a shift with property ($\mathscr A$).
  Let $p_i,u_i,q_i$ be words
  such that $p_i^{-\infty}u_iq_i^{+\infty}\in Y(\Cl X)$
  and $p_i,q_i\in\Cl E(\Cl X)$, for $i=1,2$.
  Then
  $p_1^{-\infty}u_1q_1^{+\infty}\approx p_2^{-\infty}u_2q_2^{+\infty}$
  if and only if
  the morphisms
  $([p_1],[p_1u_1q_1],[q_1])$
  and
  $([p_2],[p_2u_2q_2],[q_2])$
  of $\KK(\Cl X)$ are isomorphic.
\end{Prop}

\begin{proof}
  Suppose the morphisms
    $([p_1],[p_1u_1q_1],[q_1])$
  and
  $([p_2],[p_2u_2q_2],[q_2])$ are isomorphic.
  Then, there are words $\alpha,\beta,\gamma,\delta$ such that
  \begin{equation}\label{eq:conjugates-piqi}
 [p_1]=[\alpha\beta],\quad [p_2]=[\beta\alpha],\qquad\text{and}\qquad
  [q_1]=[\gamma\delta],\quad [q_2]=[\delta\gamma],
  \end{equation}
  and for which the following diagram in $\KK(\Cl
  X)$  commutes:
      \begin{equation}\label{eq:isomorphism-of-morphisms}
        \begin{split}
              \xymatrix@C=100pt@R=30pt{
      [p_1]
      &
      [q_1]\ar[l]_{([p_1],[p_1u_1q_1],[q_1])}
      \ar[d]^{([q_2],[q_2\delta q_1],[q_1])}
      \\
      [p_2]\ar[u]^{([p_1],[p_1\alpha p_2],[p_2])}
      &[q_2].\ar[l]_{([p_2],[p_2u_2q_2],[q_2])}
    }
        \end{split}
  \end{equation}
  Hence, we have
  \begin{equation}
    \label{eq:eq-p1-p2}
    [p_1u_1q_1]=[p_1\alpha p_2u_2q_2 \delta q_1].
  \end{equation}
  By Lemma~\ref{l:idempotents-define-nice-periodic-points}
  and Corollary~\ref{c:a-syntactic-condition-for-p-equivalent-to-q},
  there is $n\geq 1$ such that
  the sequences
  $p_i^\infty$, $q_i^\infty$ (for $i=1,2$),
  $p_1^{-\infty}\alpha p_2^{+\infty}$
  and
  $q_2^{-\infty}\delta q_1^{+\infty}$
  belong to $\mathcal A_n(\Cl X)$.
  Let $h$ be such that $\Cl X$ has property $(\mathscr A,n,h)$.
  As we can replace $p_i$ by $p_i^{nh}$,
  and $q_i$ by $q_i^{nh}$,
  without changing the elements
  of~$\Cl X$ which have appeared so far,
  we may as well assume that
  the words $p_i,q_i$
  have length at least~$nh$.

  Take $H=|p_1||q_1||p_2||q_2|$.
  We claim that
  $p_1^{-\infty}u_1q_1^{+\infty}\approx_{n,H}
  p_2^{-\infty}u_2q_2^{+\infty}$.
  Let $z_1,z_2$ be words with factorizations as in~\eqref{eq:z1z2}, satisfying conditions
  \eqref{item:the-equivalence-on-YX-a}-\eqref{item:the-equivalence-on-YX-c}
  in Definition~\ref{def:the-equivalence-on-YX}.
  Since $p_1^\infty\sim p_2^\infty$ and $q_1^\infty\sim q_2^\infty$
  by Corollary~\ref{c:a-syntactic-condition-for-p-equivalent-to-q},
  to prove the claim it remains to show that $[z_1]=[z_2]$.
  As $H=|v_i|$ is a multiple of $|p_i|$ and
  $p_i^{-\infty}$ ends with $v_i$, we have
  $v_i=p_i^{\frac{H}{|p_i|}}$, for $i=1,2$.
  Similarly, $w_i=q_i^{\frac{H}{|q_i|}}$.
  Hence, $[v_iu_iw_i]=[p_iu_iq_i]$.
  Taking into account~\eqref{eq:z1z2} and~\eqref{eq:eq-p1-p2},
  we conclude that
  \begin{equation}\label{eq:trasforming-z}
    [z_1]=[\pi \zeta_1p_1\alpha p_2u_2q_2\delta q_1\xi_1 \rho].
  \end{equation}
  Recall that $\pi\zeta_i v_i$ belongs to
  $L(\mathcal A_n(\Cl X))$, whence so does its
  prefix $\pi\zeta_i p_i$.
  Then, since $p_1^{-\infty}\alpha p_2^{+\infty}\in\mathcal
  A_n(\Cl X)$, $|p_1|\geq n$ and
  $\mathcal A_n(\Cl X)$ is an $n$-step finite type shift,
  we deduce that $\pi\zeta_1 p_1\alpha p_2\in L(\mathcal A_n(\Cl X))$.
  As $|\pi|=|p_2|=H\geq h$,
  and $\Cl X$ has property $(\mathscr A,n,h)$, we then obtain
   $[\pi\zeta_1 p_1\alpha p_2]=[\pi\zeta_2p_2]$.
   Similarly, we have $[q_2\delta q_1\xi_1\rho]=[q_2\xi_2\rho]$.
   Hence, in~\eqref{eq:trasforming-z}
   we can replace $\pi\zeta_1 p_1\alpha p_2$
   and $q_2\delta q_1\xi_1\rho$
   respectively by $\pi\zeta_2p_2$
   and $q_2\xi_2\rho$, deducing $[z_1]=[z_2]$
   via~\eqref{eq:z1z2}.

   Conversely, suppose that
   $p_1^{-\infty}u_1q_1^{+\infty}\approx
  p_2^{-\infty}u_2q_2^{+\infty}$ holds.
  Then we have $p_1^\infty\sim p_2^\infty$ and $q_1^\infty\sim
  q_2^\infty$,
  thus $[p_1]\mathrel{\D}[p_2]$
  and $[q_1]\mathrel{\D}[q_2]$, by
  Proposition~\ref{p:the-orbit-classes-and-the-D-classes}.
  We may therefore consider words $\alpha,\beta,\gamma,\delta$
  as in \eqref{eq:conjugates-piqi}.
  By Corollary~\ref{c:a-syntactic-condition-for-p-equivalent-to-q},
  there is $n$ such that
  $p_1^{-\infty}\alpha p_2^{+\infty}$
  and
  $q_2^{-\infty}\delta q_1^{+\infty}$
  belong to~$\mathcal A_n(\Cl X)$.
  Suppose $\Cl X$ has property $(\mathscr A,n,H)$.
  Let $z_1=p_1^{2H}u_1q_1^{2H}$
  and $z_2=p_1^{H}\alpha p_2^H u_2 q_2^{H}\delta q_1^H$.
  Then clearly $z_1$ and $z_2$ admit factorizations of the
  form~\eqref{eq:z1z2}, satisfying conditions
  \eqref{item:the-equivalence-on-YX-a}-\eqref{item:the-equivalence-on-YX-c}.
  Therefore, applying the hypothesis
  $p_1^{-\infty}u_1q_1^{+\infty}\approx
  p_2^{-\infty}u_2q_2^{+\infty}$,
  we obtain $[z_1]=[z_2]$.
  This means that~\eqref{eq:eq-p1-p2} holds, as $[p_i],[q_i]$
  are idempotents.
  That is, \eqref{eq:isomorphism-of-morphisms} commutes, which
  shows that
  $([p_1],[p_1u_1q_1],[q_1])$
  and
  $([p_2],[p_2u_2q_2],[q_2])$ are isomorphic.
\end{proof}

We next describe an operation on $Y(\Cl X)/{\approx}$,
first introduced by
Krieger in~\cite{Krieger:2000}, and later reprised
in~\cite{Krieger:2012,Hamachi&Krieger:2013,Hamachi&Krieger:2013b},
its presentation showing small variations between these papers.

\begin{Def}\label{def:the-multiplication-in-YX}
  Consider a shift $\Cl X$ with property ($\mathscr A$).
  Let $x,y\in Y(\Cl X)$.
Suppose the following sequence of conditions holds:
\begin{enumerate}
\item $x=p^{-\infty}u\alpha^{+\infty}$\label{item:the-multiplication-in-YX-1}
  and $y=\beta^{-\infty}v q^{+\infty}$
  for some words $u,v\in A^+$ and
  $p,q,\alpha,\beta\in\Cl E(\Cl X)$ (cf.~Corollary~\ref{c:with-property-A-periodic-points-define-idempotents});
\item $\Cl X$ has property $(\mathscr A,n,H)$ for some $n,H$;\label{item:the-multiplication-in-YX-2}
\item $\alpha^{-\infty} v\beta^{+\infty}\in\mathcal A_n(\Cl X)$ for
  some word $v$;\label{item:the-multiplication-in-YX-3}
\item there is $z\in\Cl X$ such that
  $z=p^{-\infty}u\alpha' w \beta'vq^{+\infty}$, for some
  words $w$,
  $\alpha',\beta'$ such that
  $|\alpha'|,|\beta'|\geq H$,
  the sequence $\alpha^{+\infty}$ begins with
  $\alpha'$, the sequence
  $\beta^{-\infty}$ ends with $\beta'$,
  and $\alpha'w \beta'\in L(\mathcal A_n(\Cl X))$.\label{item:the-multiplication-in-YX-4}
\end{enumerate}
Then one defines $[x]_{\approx}\cdot [y]_{\approx}=[z]_{\approx}$.
If it is not possible to establish such a sequence of conditions, then
one defines $[x]_{\approx}\cdot [y]_{\approx}=0$, where $0$ is an extra element.
\end{Def}

Krieger observed that the binary operation appearing in
Definition~\ref{def:the-multiplication-in-YX}
defines the structure of a semigroup $(Y(\Cl X)/{\approx})^0$ with zero,
whose underlying set is $Y(\Cl X)/{\approx}$ together with the extra
element $0$.
We call this semigroup the
\emph{Krieger semigroup of $\Cl X$}.
If the shift $\Cl X$ has property ($\mathscr A$), then, thanks to
Proposition~\ref{p:characterization}, we may consider
the semigroup $\KK(\Cl X)_\circ$.
The following is the main result of this section.

\begin{Thm}\label{t:characterization-of-the-Krieger-associated-semigroup}
  Let $\Cl X$ be a shift with property ($\mathscr A$). The Krieger semigroup of $\Cl X$
  is isomorphic to $\KK(\Cl X)_\circ$.
\end{Thm}

\begin{proof}
For a shift $\Cl X$ with property ($\mathscr A$), let us denote by $\tau$ the
bijection from $(Y(\Cl X)/{\approx})^0$ to $\KK(\Cl X)_\circ$ defined
by
$\tau(0)=0$ and $\tau(p^{-\infty}uq^{+\infty})=\big\langle([p],[u],[q])\big\rangle$,
when $p^{-\infty}uq^{+\infty}\in Y(\Cl X)$
with $p,q\in\Cl E(\Cl X)$. That $\tau$
is well defined and injective, follows from
Proposition~\ref{p:translation-of-elements-of-YX},
and that it is surjective, follows from
Lemma~\ref{l:idempotents-define-nice-periodic-points}.

Let $x,y\in Y(\Cl X)$.
The proof of the theorem is concluded once we establish the
equality
  \begin{equation}
    \label{eq:translation-of-operation-1}
\tau([x]_{\approx}\cdot[y]_{\approx})=\tau([x]_{\approx})\circ\tau([y]_{\approx}).
\end{equation}
Suppose that $\tau([x]_{\approx})\circ\tau([y]_{\approx})\neq 0$.
We may take words $p,\alpha,\beta\in\Cl E(\Cl X)$
and $u,v\in A^+$
such that $x=p^{-\infty}u\alpha^{+\infty}$,
$y=\beta^{-\infty}v q^{+\infty}$
and $[u]=[pu\alpha ]$, $[v]=[\beta vq]$.
Then $\tau([x]_{\approx})=\big\langle([p],[u],[\alpha])\big\rangle$
and $\tau([y]_{\approx})=\big\langle([\beta],[v],[q])\big\rangle$.
The hypothesis $\tau([x]_{\approx})\circ\tau([y]_{\approx})\neq 0$
implies the existence of
a strong non-zero divisor of the form $([\alpha],[w],[\beta])$
 such that
  \begin{equation}\label{eq:computation-of-bullet}
\tau([x]_{\approx})\circ\tau([y]_{\approx})=\big\langle([p],[uwv],[q])\big\rangle
 \end{equation}
 and $[uwv]\neq 0$.
 In particular, for every $k\geq 1$, we have
 \begin{equation}
   \label{eq:glue-them}
 p^{-\infty} u\alpha^kw\beta^kvq^{+\infty}\in\Cl X
 \end{equation}
 and
 \begin{equation}
   \label{eq:tau-glue-them}
 \tau(p^{-\infty} u\alpha^kw\beta^kvq^{+\infty})
 =\big\langle([p],[uwv],[q])\big\rangle.
 \end{equation}
 Moreover,
 as $([\alpha],[w],[\beta])$ is a
 strong non-zero divisor morphism,
 we know by Proposition~\ref{p:characterization-of-robust-morphisms}
 that there is $n\geq 1$ such that
 \begin{equation}
   \label{eq:themiddle}
   \alpha^kw\beta^k\in L(\mathcal A_n(\Cl X)),
 \end{equation}
 for every $k\geq 1$, thus
 $\alpha^{-\infty}w\beta^{+\infty}\in \mathcal A_n(\Cl X)$.
 If we choose $k\geq 1$ such that~$\Cl X$ has property $(\mathscr A,n,k)$,
 we deduce from~\eqref{eq:glue-them}
 and~\eqref{eq:themiddle}
 that we are in the conditions of Definition~\ref{def:the-multiplication-in-YX}
 (by taking $\alpha'=\alpha^k$, $\beta'=\beta^k$) in such a way that the following equality holds:
 \begin{equation*}
[x]_{\approx}\cdot [y]_{\approx}=
[p^{-\infty} u\alpha^kw\beta^kvq^{+\infty}]_{\approx}.
 \end{equation*}
This, together with~\eqref{eq:tau-glue-them}
and~\eqref{eq:computation-of-bullet},
establishes~\eqref{eq:translation-of-operation-1}.

Conversely, suppose that
$[x]_{\approx}\cdot[y]_{\approx}\neq 0$.
Take the same notational setting as in
Definition~\ref{def:the-multiplication-in-YX}.
   Note that $ \tau([z]_{\approx})=
   \big\langle([p],[pu\alpha' w\beta' vq],[q])\big\rangle$.
   Since the words $\alpha^{H}v\beta^{H}$ and $\alpha'w\beta'$
   belong to $L(\mathcal A_n(\Cl X))$,
   have the same prefix of length $H$ and the same suffix of length
   $H$,
   and since $\Cl X$ has property $(\mathscr A,n,H)$,
   we have $[\alpha'w\beta']=[\alpha^Hv\beta^H]$.
   As $[\alpha]$ and $[\beta]$
   are idempotents, we actually
   have $[\alpha'w\beta']=[\alpha v\beta]$, and
   so the following holds:
\begin{equation}\label{eq:translation-of-operation-2}
\tau([z]_{\approx})=\big\langle([p],[pu\alpha],[\alpha])
([\alpha],[\alpha v\beta],[\beta])
([\beta],[\beta vq],[q])\big\rangle.
\end{equation}
By Proposition~\ref{p:characterization-of-robust-morphisms},
the morphism $([\alpha],[\alpha v\beta],[\beta])$ is
a strong non-zero divisor,
and so we get~\eqref{eq:translation-of-operation-1}.

Therefore, we proved that
$[x]_{\approx}\cdot[y]_{\approx}
\neq 0$ if and only~if
$\tau([x]_{\approx})\circ\tau([y]_{\approx})\neq 0$,
 and that~\eqref{eq:translation-of-operation-1}
always holds.
\end{proof}

Krieger announced in the aforementioned workshop in Copenhagen
that the Krieger semigroup is flow invariant.
We next give a proof of this fact via Theorem~\ref{t:characterization-of-the-Krieger-associated-semigroup},
under the usual assumption of density of $\mathcal A(\Cl X)$.

  \begin{Prop}\label{p:strong-non-zero-are-non-zero}
    Let $\Cl X$ be a shift with property ($\mathscr A$).
    Suppose that $\mathcal A(\Cl X)$ is dense in $\Cl X$ or $\Cl X$ is sofic.
    Then $\KK(\Cl X)_{\pv{snzd}}=\KK(\Cl X)_{\pv{nzd}}$, and so
    $\KK(\Cl X)_\circ$ is invariant under flow equivalence.
  \end{Prop}

  \begin{proof}
    By Propositions~\ref{p:invariance-of-the-krieger-semigroup}
    and~\ref{p:sufficient-conditions-for-catrobust-equal-robust},
    it suffices to show that $S(\Cl X)$
    satisfies the condition in the statement of
    Proposition~\ref{p:sufficient-conditions-for-catrobust-equal-robust}.
    This is the case if $\Cl X$ is sofic,
    as seen
    in Remark~\ref{rmk:sufficient-conditions-for-catrobust-equal-robust}.
    Assume that $\mathcal A(\Cl X)$ is dense.
    Take $u\in A^+$ with $[u]_{\Cl X}\neq 0$.
    As $\mathcal A(\Cl X)$ is dense in~$\Cl X$, there is
    $n\geq 1$ and $x\in\mathcal A_n(\Cl X)$ such
    that $u$ is a factor of $x$.
    Let $n,H$ be such that
    $\Cl X$ has property $(\mathscr A,n,H)$.
    Since
    $\mathcal A_n(\Cl X)$ is a finite type shift, there
    are words $w_1,w_2,v_1,v_2$
    such that $[w_1v_1uv_2w_2]_{\mathcal A_n(\Cl X)}\neq 0$
    with $[w_1]_{\mathcal A_n(\Cl X)}$
    and $[w_2]_{\mathcal A_n(\Cl X)}$ idempotents.
    As $[w_i]_{\mathcal A_n(\Cl X)}=[w_i^H]_{\mathcal A_n(\Cl
      X)}=[w_i^{2H}]_{\mathcal A_n(\Cl X)}\neq 0$
    and $\Cl X$ has property $(\mathscr A,n,H)$,
    we deduce that $e_i=[w_i^H]_{\Cl X}$
    is an idempotent of $S(\Cl X)$, for $i=1,2$.
    But $[w_1^Hv_1uv_2w_2^H]_{\mathcal A_n(\Cl X)}\neq 0$,
    which implies $e_1[v_1uv_2]_{\Cl X}e_2\neq 0$,
    showing that $S(\Cl X)$ satisfies the hypothesis in
    Proposition~\ref{p:sufficient-conditions-for-catrobust-equal-robust}.
  \end{proof}

 \section{Subsynchronizing subshifts of  a sofic shift}
\label{sec:subsynchr-subsh-sofi}

As a further application of our main results,
we apply them to the poset of subsynchronizing
subshifts of a sofic shift considered in~\cite{Jonoska:1998}. This poset,
whose definition is recalled in this section,
provides information about the structure of a reducible sofic shift.

Let $\Cl X$ be a sofic subshift of $\z A$.
We recall some definitions and remarks from~\cite{Jonoska:1998}.
If $m$ is a synchronizing word for $\Cl X$, then $m$ is \emph{magic}
for $\Cl X$ if $mum\in L(\Cl X)$ for some $u\in A^\ast$.
If $m$ is magic for $\Cl X$, then the set
\begin{equation*}
\{v\in A^+\mid \exists x\in A^\ast:mxv\in L(\Cl X)\}
\end{equation*}
is the set of finite blocks of a sofic subshift of $\Cl X$.
This shift is denoted~$S(m)$. If $M$ is a set of magic words for
$\Cl X$, then $S(M)$ denotes the sofic shift $\bigcup_{m\in M} S(m)$.
A subshift of $\Cl X$ of the form $S(M)$ is called a \emph{subsynchronizing
subshift of $\Cl X$}. The set $\mathsf{Subs}(\Cl X)$ of subsynchronizing
subshifts of $\Cl X$ is finite; see Lemma~\ref{l-magic-word-from-idempotents} below. It may be empty. If $\Cl X$ is
irreducible, then $\mathsf{Subs}(\Cl X)=\{\Cl X\}$.

Let $\Cl X$ be a sofic shift. Say that $s\in S(\Cl X)$ is \emph{synchronizing}
if $s=[u]$ with $u$ synchronizing. Note that $s$ is
synchronizing if and only if $rs,st\neq 0$ implies $rst\neq 0$, for
all $r,t\in S(\Cl X)$. It follows from Lemma~\ref{l:synchwords} that
the synchronizing elements of $S(\Cl X)$, together with $0$, form an
ideal and that a synchronizing element $s$ is idempotent if and only
if $s^2\neq 0$.
A synchronizing idempotent of $S(\Cl X)\setminus\{0\}$ will be called
a \emph{magic idempotent} for $\Cl X$.

\begin{Rmk}\label{rmk:magic-idempotent}
  If $\Cl X$ is a sofic shift, and $s\neq 0$,
  then there are $r,t\in S(\Cl X)$ and idempotents $g,h\in S(\Cl X)$
  such that $grsth\neq 0$. Therefore, an idempotent $e$ of $S(\Cl X)$ is magic if and only if
  the following happens: whenever $(e,x,f)$ and $(g,y,e)$ are non-zero
  morphisms of $\KK(\Cl X)$, the composition $(g,y,e)(e,x,f)$ is a
  non-zero morphism of $\KK(\Cl X)$.
\end{Rmk}

Let $e$ be a magic idempotent for $\Cl X$, and let $u$ be a word such that
$e=[u]$. Then $u$ is a magic word for $\Cl X$. Clearly,
if $[u]=[v]$, then $S(u)=S(v)$. We may
then define $S(e)$ as being $S(u)$.
If $M$ is a set of magic idempotents for~$\Cl X$, then $S(M)$ denotes
the sofic shift $\bigcup_{e\in M} S(e)$.

In the proof of the following lemma, $\mathrm{Fac}(X)$ denotes the set of
non-empty words which are factors of a language $X$.

\begin{Lemma}\label{l-magic-word-from-idempotents}
  If $m$ is a magic word for $\Cl X$, then $S(m)=S(e)$ for some
  magic idempotent $e$ for $\Cl X$.
\end{Lemma}

\begin{proof}
  As $m$ is a magic word for $\Cl X$, there is $u\in L(\Cl X)$
  with $mum\in L(\Cl X)$. We claim that $e=[mu]$ is
  a magic idempotent for $\Cl X$ such that $S(m)=S(e)$.

  We first show that $e$ is an idempotent. It follows from Lemma~\ref{l:synchwords} that we must show $(mu)^2\in L(\Cl X)$.  But $mum,mu\in L(\Cl X)$ implies $mumu\in L(\Cl X)$.  Also $mu$ is synchronizing by Lemma~\ref{l:synchwords}.  Thus $e$ is a magic idempotent.

  We have $L(S(m))=\mathrm{Fac}(R_{\Cl X}(m))$
  and $L(S(e))=\mathrm{Fac}(R_{\Cl  X}(mu))$ by definition.
  Clearly,
  $\mathrm{Fac}(R_{\Cl  X}(mum))\subseteq \mathrm{Fac}(R_{\Cl  X}(mu))\subseteq
  \mathrm{Fac}(R_{\Cl  X}(m))$.
  But $R_{\Cl  X}(mum)=R_{\Cl  X}(m)$ by Lemma~\ref{l:synchwords}.
  This shows that $L(S(m))=L(S(e))$, thus $S(m)=S(e)$.
\end{proof}

In the following lemma, we see how the inclusion relation
between elements of $\mathsf{Subs}(\Cl X)$
is codified in the Karoubi envelope.

\begin{Lemma}\label{l:inclusion-is-described-by-non-zero-morphisms}
  Let $e$ be a magic idempotent for $\Cl X$,
  and let $M$ be a set of magic idempotents for $\Cl X$.
  The following conditions are equivalent:
  \begin{enumerate}
  \item $S(e)\subseteq S(f)$ for some $f\in M$;\label{item:non-zero-morphisms-1}
  \item $S(e)\subseteq S(M)$;\label{item:non-zero-morphisms-2}
  \item for some $f\in M$,
    there is a non-zero morphism $e\to f$ in $\KK(\Cl X)$.\label{item:non-zero-morphisms-3}
  \end{enumerate}
\end{Lemma}

\begin{proof}
  The implication
  (\ref{item:non-zero-morphisms-1})$\Rightarrow$(\ref{item:non-zero-morphisms-2})
  is trivial.
  
  Let us show
  (\ref{item:non-zero-morphisms-2})
  $\Rightarrow $(\ref{item:non-zero-morphisms-3}).
  Assuming $S(e)\subseteq S(M)$, let $v\in L(\Cl X)$ be such that
  $e=[v]$. As $v\in S(e)$, by hypothesis there is
  $f\in M$ with $v\in S(f)$.
  Let $m\in L(\Cl X)$ be such that $f=[m]$.
  Then $mxv\in L(\Cl X)$ for some $x\in A^\ast$.
  Hence, $(f,[mxv],e)$ is a non-zero morphism $e\to f$ in $\KK(\Cl X)$.

  Finally, suppose that $(f,[u],e)$ is a non-zero morphism $e\to f$
  in $\KK(\Cl X)$ such that $f\in M$.
  Let $w\in S(e)$.
  If $m,v\in L(\Cl X)$ are synchronizing words
  such that $f=[m]$ and $e=[v]$, then
  $muv\in L(\Cl X)$ and $vxw\in L(\Cl X)$ for some $x\in A^\ast$.
  But then $muvxw\in L(\Cl X)$ by synchronization, and so
  $w\in S(f)$, thus showing that $S(e)\subseteq S(f)$
  and establishing
  (\ref{item:non-zero-morphisms-1})$\Rightarrow$
  (\ref{item:non-zero-morphisms-2}).
\end{proof}

\begin{Cor}\label{c:magic-D-equivalent}
  If $e,f$ are magic idempotents for $\Cl X$ such that
  $e\mathrel{\D}f$, then $S(e)=S(f)$.
\end{Cor}

\begin{proof}
  Since $e\mathrel{\D}f$, there are isomorphisms
  $e\to f$ and $f\to e$. As $e,f\neq 0$, these isomorphisms are
  non-zero, thus $S(e)=S(f)$
  by Lemma~\ref{l:inclusion-is-described-by-non-zero-morphisms}.
\end{proof}

Remark~\ref{rmk:magic-idempotent} ensures that
being a magic idempotent is a property preserved by equivalence of Karoubi
envelopes, and so in the next result the set $S(F(M))$ is well defined.

\begin{Prop}\label{p:first-cor-on-order-structure-SubsX}
    Let $F\colon\KK(\Cl X)\to \KK(\Cl Y)$ be an equivalence,
  where $\Cl X$ and $\Cl Y$ are sofic shifts.
  Then the
  mapping $\Psi_F\colon \mathsf{Subs}(\Cl X)\to \mathsf{Subs}(\Cl Y)$
  defined by $\Psi_F(S(M))=S(F(M))$, where $M$ runs over the sets of
  magic idempotents, is a well-defined isomorphism of posets.
  If $G$ is a quasi-inverse of $F$, then
  $\Psi_G$ is the inverse of $\Psi_F$.
\end{Prop}

\begin{proof}
  Let $M$ and $N$ be sets of magic idempotents with $S(M)\subseteq S(N)$. Since equivalences preserve non-zero morphisms,
  we have $S(F(M))\subseteq S(F(N))$
  by Lemma~\ref{l:inclusion-is-described-by-non-zero-morphisms}.
  This shows that $\Psi_F$ is a well-defined and order-preserving function.
  Moreover, since $GF(e)\mathrel{\D}e$ for every idempotent $e$,
  we conclude from Corollary~\ref{c:magic-D-equivalent}
  that $\Psi_F$ and $\Psi_G$ are mutually inverse.
\end{proof}

As a direct consequence of
  Theorem~\ref{splittingisinvflow}
  and Proposition~\ref{p:first-cor-on-order-structure-SubsX},
  we deduce the following.

\begin{Cor}\label{c:order-structure-SubsX}
  The order structure of the
  poset of subsynchronizing subshifts of a sofic shift is
  invariant under flow equivalence.\qed
\end{Cor}

The invariance of the order structure of $\mathsf{Subs}(\Cl X)$ under conjugacy
of sofic shifts was proved in~\cite{Jonoska:1998}. Concrete examples
were examined in that paper. Actually a more general result was obtained in~\cite{Jonoska:1998}:
viewing $\mathsf{Subs}(\Cl X)$ as a labeled poset, where the
label of each element is its conjugacy class, one obtains a conjugacy
invariant. We shall not give a new proof of this fact using our
methods, since we would not obtain a significative
simplification. However, we do generalize it to flow equivalence, in
the next theorem.
For a sofic  shift~$\Cl X$,
its \emph{labeled flow poset of subsynchronizing subshifts}
is the labeled poset obtained from
   $\mathsf{Subs}(\Cl X)$ in which the
   label of each element of $\mathsf{Subs}(\Cl X)$
   is its flow equivalence class.

\begin{Thm}\label{t:flow-inv-sub-sync}
  The labeled flow poset of subsynchronizing subshifts of a sofic shift is
  a flow equivalence invariant.
\end{Thm}

\begin{proof}
  In view of the proof of
  Corollary~\ref{c:order-structure-SubsX} and the results
  from~\cite{Jonoska:1998}, it suffices to show that if $\Cl X'$ is the
  symbol expansion of $\Cl X$ relatively to a letter $\alpha$, then there is
  an equivalence
  $F\colon\KK(\Cl X)\to \KK(\Cl X')$
  such that $S(F(e))$ is the symbol expansion of
  $S(e)$ relatively to $\alpha$, whenever $e$ is a
  magic idempotent for $\Cl X$.
  This is done in an appendix, in
  Proposition~\ref{p:symbol-expa-of-subsync}, using some tools
  introduced in Section~\ref{sec:proof-theor-reft:m}.
\end{proof} 

\section{Eventual conjugacy}\label{sec:eventual-conjugacy}

It remains a major open problem to determine whether conjugacy is decidable for shifts of finite type.
To attack this problem,  a relation called \emph{eventual conjugacy}, also known
as \emph{shift equivalence},
was introduced,
which may be defined as follows
(see~\cite[Chapter 7]{MarcusandLind} for historical background and
other details.) Let $n$ be a positive integer.
Recall that $A^n$ denotes the subset
of $A^+$ of words with length~$n$.
Considering the natural embedding of $(A^n)^+$
into~$A^+$,
one defines the \emph{$n^{th}$ higher power} of a
 subshift $\Cl X$ of $A^{\mathbb Z}$
as the subshift $\Cl X^n$ of $(A^n)^{\mathbb Z}$
such that $L(\Cl X^n)=L(\Cl X)\cap (A^n)^+$.
Two shifts \Cl X and \Cl Y are \emph{eventually conjugate} if and only if
 $\Cl X^n$ and $\Cl Y^n$ are conjugate
for all sufficiently large $n$ (cf.~\cite[Definition
1.4.4]{MarcusandLind}.)
Kim and Roush proved that
eventual conjugacy for sofic shifts is
decidable~\cite{Kim&Roush:1990},
but the algorithm which is available is quite intricate.
Another deep result by Kim and
Roush~\cite{Kim&Roush:1992,Kim&Roush:1999}
is that, for shifts of finite type, eventual conjugacy is not the same as
conjugacy.

It is easy to check that,
for every shift $\Cl X\subseteq A^{\mathbb Z}$ and $u,v\in (A^n)^+$,
we have $[u]_{\Cl X^n}=[v]_{\Cl X^n}$
if and only if
$[u]_{\Cl X}=[v]_{\Cl X}$,
and so $S(\Cl X^n)$
embeds naturally in $S(\Cl X)$.
For sofic shifts, we have the following sort of converse,
taken from \cite[Lemma 5.1]{Costa:2006} following arguments
from~\cite{Beal&Fiorenzi&Perrin:2005a}.

\begin{Lemma}
  \label{l:effect-on-LU-of-a-power}
  Let $\Cl X$ be a sofic shift.
  For each idempotent $e\in S(\Cl X)$,
  choose  $u_e\in A^+$ such that
  $[u_e]_{\Cl X}=e$. Let $A_{\Cl X}=\prod_{e\in E(S(\Cl X))}|u_e|$.
  Then, for every $n\geq 1$, we have ${LU}(\Cl X)={LU}(\Cl X^{nA_{\Cl
      X}+1})$.
\end{Lemma}

    Let $N$ be such that $\Cl X^n$
    and $\Cl Y^n$ are conjugate
    for all $n\geq N$.
    Consider the integer $k=NA_\Cl XA_\Cl Y+1$.
    Then $LU(\Cl X^k)=LU(\Cl X)$
    and $LU(\Cl Y^k)=LU(\Cl Y)$
    by Lemma~\ref{l:effect-on-LU-of-a-power}.
    Since $\Cl X^k$ and $\Cl Y^k$ are conjugate,
    this justifies the following result, implicitly used
    in~\cite{Costa:2006}.

\begin{Cor}\label{c:effect-on-LU-of-a-power}
    If $\Cl X$ and $\Cl Y$
  are eventually  conjugate sofic shifts,
  then, for every integer $N$, there is $k\geq N$, with
  $\gcd(k,N)=1$, and $\Cl X^k$ and $\Cl Y^k$ conjugate,
  such that $LU(\Cl X^k)=LU(\Cl X)$  and $LU(\Cl Y^k)=LU(\Cl Y)$.
  \qed
\end{Cor}

\begin{Rmk}\label{rmk:power-of-a-labeled-graph}
  For a labeled graph $\mathfrak G$, let
  $\mathfrak G^n$ be the labeled graph over $A^n$
  defined as follows:
  the vertices are those of $\mathfrak G$,
  and an edge from a vertex $p$ to a vertex $q$, with label $u\in A^n$,
  is a path in $\mathfrak G$ from $p$ to $q$ with label $u$.
  It is easy to see that, up to isomorphism, the Krieger cover
  of $\Cl X^n$ is the labeled graph
  $\mathfrak K(\Cl X)^n$, and, if $\Cl X$ is synchronizing,
  the Fischer cover
  of $\Cl X^n$ is $\mathfrak F(\Cl X)^n$.
  Moreover, the action of $S(\Cl X^n)$
  on $Q(\Cl X^n)$
  is the restriction to $S(\Cl X^n)$ of the action of
  $S(\Cl X)$ on $Q(\Cl X)$.
\end{Rmk}

\begin{Thm}\label{t:eventual-conjugacy-invariance-theorem}
      Suppose that $\Cl X$ and $\Cl Y$ are eventually conjugate
    sofic shifts.
    Then the actions
    $\mathbb A_{\Cl X}$ and $\mathbb A_{\Cl Y}$ are equivalent.
    The same happens with the actions
    $\mathbb A_{\Cl X}^{\mathfrak F}$ and $\mathbb A_{\Cl
      Y}^{\mathfrak F}$ if $\Cl X$  and $\Cl Y$
    are irreducible.
  In particular, the equivalence class of
  $\KK(\Cl X)$ is an eventual conjugacy invariant of sofic shifts.
\end{Thm}

  \begin{proof}
    By Corollary~\ref{c:effect-on-LU-of-a-power},
    there is $k$
    such that
   $\Cl X^k$ is conjugate to $\Cl Y^k$,
    $\KK(\Cl X^k)=\KK(\Cl X)$
    and $\KK(\Cl Y^k)=\KK(\Cl Y)$.
    The result follows then from
    Theorems~\ref{t:main-D-cate-AYF} and
    \ref{t:main-D-cate-AYF-Fischer-version}
    in view of Remark~\ref{rmk:power-of-a-labeled-graph}.
  \end{proof}

\section{Proofs of
  Theorems~\ref{splittingisinvflow},~\ref{t:main-D-cate-AYF}
  and~\ref{t:main-D-cate-AYF-Fischer-version}}\label{sec:proof-theor-reft:m}
This section is divided in three parts. First we establish
the versions of Theorem~\ref{splittingisinvflow} and Theorem~\ref{t:main-D-cate-AYF}
obtained by replacing ``flow equivalence'' by ``conjugacy'' and
then the corresponding version for symbol expansion.  In the 
third part we deduce
Theorem~\ref{t:main-D-cate-AYF-Fischer-version}.

\subsection{Invariance under conjugacy}
\label{sec:invar-under-conj}

The proofs we present here of the conjugacy invariance part of
Theorems~\ref{splittingisinvflow} and~\ref{t:main-D-cate-AYF} are via the results of Nasu~\cite{Nasu:1986}.  A direct, but technical, proof
appeared in the original version of this paper~\cite{ACosta&Steinberg:2013}.
 
The conjugacy invariance part of Theorems~\ref{splittingisinvflow}
and~\ref{t:main-D-cate-AYF} are immediate consequences of
the ``only if'' part of Theorem~\ref{t:nasu}
and of the following lemma
(which should be compared
to~\cite[Proposition 7.5]{Beal&Berstel&Eilers&Perrin:2010arxiv}.)

\begin{Lemma}\label{l:biglemma}
Suppose that $\mathcal A$ is a bipartite right-resolving labeled graph
with components $\mathcal A_1,\mathcal A_2$ (with corresponding state
sets $Q_1,Q_2$.)
Let $S_i$ be the transition semigroup of $\mathcal A_i$, for $i=1,2$, and let $S$ be the transition semigroup of $\mathcal A$.  Let $S,S_1,S_2$ act on $Q^0$, $Q_1^0$ and $Q_2^0$, respectively.  Then the Karoubi envelopes of $S$, $S_1$ and $S_2$ are all equivalent and the actions $(\mathbb A_{Q^0},\KK(S))$, $(\mathbb A_{Q_1^0},\KK(S_1))$  and $(\mathbb A_{Q_2^0},\KK(S_2))$ are all equivalent.
\end{Lemma}
\begin{proof}
Note that $Q^0=Q_1^0\cup Q_2^{0}$, where the base points $0$ of $Q_1$ and $Q_2$ are identified.  From $Q_1^0\cdot A_2A_1=\{0\}=Q_2^0\cdot A_1A_2$, it follows that $S_1,S_2$ are the subsemigroups of $S$ generated by $A_1A_2$ and $A_2A_1$, respectively.

Next observe that $E(S)=E(S_1)\cup E(S_2)$.  Indeed, since $A_1^2=0=A_2^2$ in $S$, each non-zero idempotent $e\in S$ is represented by an alternating word $u$ in $A_1$ and $A_2$.  From $0\neq e=e^2=u^2$, we conclude that the first and last letters of $u$ cannot both be in $A_i$, for either $i=1,2$. Thus $e\in S_1\cup S_2$.

Since $S_1S_2=0=S_2S_1$, it is now immediate that if $0\neq s\in S_1$
and $e,f\in E(S)$, then $esf=s$ implies $e,f\in S_1$.  Thus $\KK
(S_1)$ is a full subcategory of $S$ (that is, the inclusion functor is
full.)  To obtain that the inclusion is an equivalence of categories,
we need to show that each $e\in E(S)$ is isomorphic,
or equivalently $\mathscr D$-equivalent, to an idempotent of $S_1$.
By the previous paragraph, we may assume that $e\in E(S_2)\setminus
\{0\}$ and that $e=a_2va_1$ where $a_2\in A_2$, $a_1\in A_1$ and $v\in
S_1$. Lemma~\ref{l:a-lemma-on-conjugate-words} implies $f=(va_1a_2)^2\in S_1$ is an idempotent with $e\mathrel{\D} f$, that is $e\cong f$ in $\KK(S)$.
This proves that the inclusion functor $F\colon \KK (S_1)\to \KK (S)$ is an equivalence.

Next let $\eta_e\colon Q_1^0e\to Q^0e$ be the inclusion for each $e\in
E(S_1)$. Then it is immediate that $\eta\colon \mathbb
A_{Q_1^0}\Rightarrow \mathbb A_{Q^0}\circ F$ is a natural transformation.  Clearly, each $\eta_e$ is injective.  For surjectivity, we use that $Q_2^0e=\{0\}$ for $e\in E(S_1)$ and hence $Q^0e=Q_1^0e$. A symmetric argument for $S_2$ completes the proof.
\end{proof}

\subsection{Invariance under  symbol expansion}\label{sec:invar-under-flow}

The reader should review the definitions and notation from
Subsection~\ref{sec:flow-equivalence}.

\begin{Rmk}\label{r:image-of-E}
Using induction on the length of
words, one verifies that
\begin{equation*}
  \ci E(A^\ast)=B^\ast
  \setminus
  \Bigl(\dia B^\ast
  \cup
  B^\ast \alpha
  \cup
  \bigcup_{x\in A\setminus \{\dia\}}\!\!\!\!{B^\ast \alpha xB^\ast}
  \cup
  \bigcup_{x\in A\setminus \{\alpha\}}\!\!\!\!{B^\ast x\dia B^\ast}
  \Bigr).
\end{equation*}
\end{Rmk}

Remark~\ref{r:image-of-E} justifies several simple and useful facts,
like the following.

\begin{Lemma}\label{l:symbol-exp-3}
Let $v\in A^+$.
For $x,y,u\in B^\ast$, if
$x\ci E(v)y=\ci E(u)$ then
$x,y\in \ci E(A^*)$ and
$u=\ci E^{-1}(x)v\ci E^{-1}(y)$.
Consequently,
if $\ci E(v)\in L(\Cl X')$ then $v\in L(\Cl X)$.
\end{Lemma}

\begin{proof}
  The first part of the lemma follows from
  Remark~\ref{r:image-of-E}
  and the fact that $\ci E$ is injective.
  If $\ci E(v)\in L(\Cl X')$
  then there is $u\in L(\Cl X)$ and $x,y\in B^\ast$
  with $\ci E(u)=x\,\ci E(v)\,y$.
  From $u=\ci E^{-1}(x)\,v\,\ci E^{-1}(y)$,
  we deduce $v\in L(\Cl X)$.
\end{proof}

A direct consequence of Lemma~\ref{l:symbol-exp-3} is the following
analog.

\begin{Lemma}\label{l:symbol-exp-4}
  Let $x\in A^{\mathbb Z}$.
  If $\ci E(x)\in \Cl X'$ then $x\in \Cl X$.
\end{Lemma}

\begin{proof}
  If $u$ is a finite block of $x$,
  then $\ci E(u)\in L(\Cl X')$. By
  Lemma~\ref{l:symbol-exp-3}, we have $u\in L(\Cl X)$. Hence $x\in\Cl X$.
\end{proof}

In~\cite{Lawson:2011}, a semigroup homomorphism $\theta\colon S\to T$
is termed a \emph{local isomorphism} if the following conditions are
satisfied:\footnote{Actually, Lawson only defines the notion for semigroups with local
units.}
\begin{enumerate}
\item $\theta|_{eSf}$ is a bijection of $eSf$ with $\theta(e)T\theta(f)$;
\item if $e'\in E(\theta(S))$, then there is $e\in E(S)$ with $\theta(e)=e'$;
\item for each $e\in E(T)$, there is $f\in E(\theta(S))$ with $e\mathrel{\D} f$.
\end{enumerate}

\begin{Rmk}\label{rmk:equivalence-induced-by-local-iso}
It is immediate from the definition (cf.~\cite{Lawson:2011}) that if
$\theta\colon S\to T$ is a local isomorphism, then $\theta$ induces an equivalence $\Theta\colon \KK(S)\to \KK(T)$ given by $\Theta(e)=\theta(e)$ on objects and $\Theta(e,s,f) = (\theta(e),\theta(s),\theta(f))$ on morphisms.
\end{Rmk}

\begin{Prop}\label{p:symbol-expansion-embed}
  There is a well-defined homomorphism
  $\ci E'\colon S(\Cl X)\to S(\Cl X')$
  sending
  $[u]_{\Cl X}$ to $[\ci E(u)]_{\Cl X'}$ and $0$ to $0$. Moreover, $\ci E'$ is a local isomorphism.
\end{Prop}

\begin{proof}
We begin by showing that $[u]_{\Cl X}\subseteq [v]_{\Cl X}$
   if and only if
   $[\ci E(u)]_{\Cl X'}\subseteq [\ci E(v)]_{\Cl X'}$
   for $u,v\in A^+$.
  Suppose that
  $[u]_{\Cl X}\subseteq [v]_{\Cl X}$
  and
  let $x$ and $y$ be words such that
  $x\,\ci E(u)\,y$ belongs to $L(\Cl X')$.
  There are words $x'$ and $y'$ such that
  $x'x\,\ci E(u)\,yy'$ belongs to $L(\Cl X')\cap \ci E(A^+)$.
  By Lemma~\ref{l:symbol-exp-3},
  we have $\ci E^{-1}(x'x)\,u\,\ci E^{-1}(y'y)\in L(\Cl X)$.
  Since $[u]_{\Cl X}\subseteq [v]_{\Cl X}$,
  it follows that the word $z=\ci E^{-1}(x'x)\,v\,\ci E^{-1}(y'y)$
  also belongs to $L(\Cl X)$.
  Hence $x\,\ci E(v)\,y$ belongs to $L(\Cl X')$, since
  it is a factor of $\ci E(z)$.
  Therefore
  $[\ci E(u)]_{\Cl X'}\subseteq [\ci E(v)]_{\Cl X'}$.
  Let $z\in A^+\setminus L(\Cl X)$.
  Then $\ci E(z)\notin L(\Cl X')$, again by
  Lemma~\ref{l:symbol-exp-3}.
  Therefore, we have
  $\ci E'([z]_{\Cl X})
  =0=\ci E'(0)$.
  This proves $\ci E'$ is a well-defined homomorphism.

  On the other hand, if
  $[\ci E(u)]_{\Cl X'}\subseteq [\ci E(v)]_{\Cl X'}$ then,
  for every $x,y\in A^\ast$, we have the following chain of
  implications, where the last one uses Lemma~\ref{l:symbol-exp-3}:
  \begin{equation*}
    xuy\in L(\Cl X)\Rightarrow
    \ci E(xuy)\in L(\Cl X')\Rightarrow
    \ci E(xvy)\in L(\Cl X')\Rightarrow
    xvy\in L(\Cl X).
  \end{equation*}
  This shows that $[u]_{\Cl X}\subseteq [v]_{\Cl X}$.
  It follows that $\ci E'$ is injective.
  In particular, $\ci E'$
  satisfies the second condition in the definition of a local isomorphism.

  Suppose that $f=[w]_{\Cl X'}$ is an idempotent of
  $S(\Cl X')$
  with $f\notin \ci E'(S(\Cl X))$.
  Then $w\notin \ci E(A^+)$ on the one hand, and $w\in L(\Cl X')$, on
  the other hand (the latter because $f\neq 0$.)
  As $w\in L(\Cl X')$,  there is $v\in L(\Cl X)$
  such that
  $\ci E(v)=pwq$ for some $p,q$.
  Let $u$ be a (possibly empty) word of maximal length such that $\ci E(u)$ is a factor
  of $w$, and let $a,b$ be words with
  $w=a\,\ci E(u)\,b$.
  Since $\ci E(v)=pa\,\ci E(u)\,bq$,
  it follows that $pa,bq\in \ci E(A^+)$ by Lemma~\ref{l:symbol-exp-3}.
  By the maximality of $u$, we have $a,b\in\{1,\alpha,\dia\}$.
  Note also that
  $\{a,b\}\neq \{1\}$, because $w\notin \ci E(A^+)$.
  Since $[w]_{\Cl X'}$ is idempotent,
  the word $w^2=a\,\ci E(u)\,ba\,\ci E(u)\,b$
  belongs to $L(\Cl X')$,
  thus
  $ra\,\ci E(u)\,ba\,\ci E(u)\,bs\in \ci E(A^+)$
  for some $r,s$.
  Then $ba$ belongs to $\ci E(A^+)$
  by Remark~\ref{r:image-of-E}. The only possibility is
  $ba=\alpha\dia$, thus $w=\dia \,\ci E(u)\,\alpha$.
  Since $\ci E(\alpha u)=\alpha\dia \,\ci E(u)$,
  it follows from Lemma~\ref{l:a-lemma-on-conjugate-words}
  that $[\ci E(\alpha u)^2]_{\Cl X'}$
  is an idempotent in the image of~$\ci E'$
  which is $\D$-equivalent to $e$.

 It remains to show that $\ci E'(eS(\Cl X)f)=\ci E'(e)S(\Cl X')\ci E'(f)$,
 whenever $e,f\in E(S(\Cl X))$.
  Let $u,v\in A^+$ be
  such that $e=[u]_\Cl X$ and
  $f=[v]_\Cl X$.  We clearly have $0\in \ci E'(eS(\Cl X)f)$.
  Take $w\in L(\Cl X')$ such that $[w]_{\Cl X'}\in  \ci E'(e)S(\Cl X')\ci E'(f)$.
  Since $[w]_{\Cl X'}=[\ci E(u)\,w\,\ci E(v)]_{\Cl X'}$,
  there are words $p,q$ such that
  $p\,\ci E(u)\,w\,\ci E(v)\,q$ belongs to $L(\Cl X')\cap\mathrm{Im}\ {\ci E}$.
  From Remark~\ref{r:image-of-E}
  it follows that $w=\ci E(w')$ for some $w'\in A^+$.
  Moreover, $w'\in L(\Cl X)$ by Lemma~\ref{l:symbol-exp-3}.
  Clearly, $[uw'v]_{\Cl X}\in eS(\Cl X)f$.
  Furthermore,
  $\ci E'[uw'v]_{\Cl X}=
  \ci E'(e)\,[w]_{\Cl X'}\,\ci E'(f)=
  [w]_{\Cl X'}$, completing the proof.
\end{proof}

Consider the mapping  $F_{\ci E}\colon\KK(\Cl X)\to \KK(\Cl X')$
defined as follows:
\begin{enumerate}
\item $F_{\ci E}(e)=\ci E'(e)$ if $e$ is an object of $\KK(\Cl X)$;
\item $F_{\ci E}(e,s,f)=\bigl(\ci E'(e), \ci E'(s),\ci E'(f)\bigr)$
  if $(e,s,f)$ is morphism of $\KK(\Cl X)$.
\end{enumerate}
In view of Remark~\ref{rmk:equivalence-induced-by-local-iso} and
Proposition~\ref{p:symbol-expansion-embed}, we know that
$F_{\ci E}$ is an equivalence.
This completes the proof of Theorem~\ref{splittingisinvflow}.

We proceed with the proof of Theorem~\ref{t:main-D-cate-AYF},
by  steps, in order to produce an isomorphism
$\mathbb A_{\Cl X}\Rightarrow \mathbb A_{\Cl X'}\circ F_{\ci E}$.

\begin{Lemma}\label{l:symbol-expansion-preserves-order-contexts}
  The inclusion
    $C_{\Cl X}(x)\subseteq C_{\Cl X}(y)$
    holds if and only if
    the inclusion
    $C_{\Cl X'}(\ci E(x))\subseteq C_{\Cl X'}(\ci E(y))$
    holds, whenever $x,y\in A^{\mathbb Z^-}$.
\end{Lemma}

\begin{proof}
  Suppose that $C_{\Cl X'}(\ci E(x))\subseteq C_{\Cl X'}(\ci E(y))$.
  Let $z\in C_{\Cl X}(x)$, that is $x.z\in \Cl X$.
  Then $\ci E(x).\ci E(z)=\ci E(x.z)\in\Cl X'$.
  Hence $\ci E(y).\ci E(z)=\ci E(y.z)\in\Cl X'$,
  by hypothesis.
  By Lemma~\ref{l:symbol-exp-4}, we have $y.z\in \Cl X$,
  showing $C_{\Cl X}(x)\subseteq C_{\Cl X}(y)$.

  Conversely, assume $C_{\Cl X}(x)\subseteq C_{\Cl X}(y)$.
  Let $z\in C_{\Cl X'}(\ci E(x))$, meaning $\ci E(x).z\in \Cl X'$.
  By Remark~\ref{r:image-of-E},
  we have $z=\ci E(t)$ for some unique $t\in A^{\mathbb N}$.
  Moreover, $x.t\in\Cl X$ by Lemma~\ref{l:symbol-exp-4}.
  Since $C_{\Cl X}(x)\subseteq C_{\Cl X}(y)$,
  we obtain $y.t\in\Cl X$,
  whence $\ci E(y).z\in \Cl X'$.
  Therefore, $C_{\Cl X'}(\ci E(x))\subseteq C_{\Cl X'}(\ci E(y))$.
\end{proof}

Consider the map $h\colon Q(\Cl X)\to Q(\Cl X')$
       with $h(C_{\Cl X}(x))=C_{\Cl X'}(\ci E(x))$,
     for every $x\in A^{\mathbb Z^-}$.
     By Lemma~\ref{l:symbol-expansion-preserves-order-contexts},
     this is a well-defined
     injective function.

\begin{Lemma}\label{l:hqs}
  We have $h(q\cdot s)=h(q)\cdot \ci E'(s)$,
  for every $q\in Q(\Cl X)$ and $s\in S(\Cl X)$.
\end{Lemma}

\begin{proof}
  Take $u\in A^+$ with $s=[u]_{\Cl X}$,
  and $x\in A^{\mathbb Z^-}$ with $q=C_{\Cl X}(x)$.
  Then $q\cdot s=C_{\Cl X}(x)$ and
   $h(q\cdot s)=C_{\Cl X'}(\ci E(xu)) =C_{\Cl X'}(\ci E(x))\cdot
   [\ci E(u)]_{\Cl X'}$.
   Since $\ci E'(s)=[\ci E(u)]_{\Cl X'}$, this concludes the
   proof.
\end{proof}

  \begin{Lemma}\label{l:how-h-restricts}
    The equality $h(Q(\Cl X)\cdot s)=Q(\Cl X')\cdot \ci E'(s)$
    holds for every $s\in S(\Cl X)$.
  \end{Lemma}

  \begin{proof}
    By Lemma~\ref{l:hqs}, we have
    $h(Q(\Cl X)\cdot s)\subseteq Q(\Cl X')\cdot \ci E'(s)$.
    Conversely, let $q\in Q(\Cl X')\cdot \ci E'(s)$.
    We want to show that $q\in h(Q(\Cl X)\cdot s)$.
    Since $\emptyset=\emptyset\cdot 0$, by
    Lemma~\ref{l:hqs} we have
    $h(\emptyset)=h(\emptyset)\cdot 0=\emptyset$.
    Therefore, we may suppose $q\neq \emptyset$.
    Take $u\in A^+$ such that $s=[u]_{\Cl X}$.
    Then, there is $y\in B^{\mathbb Z^-}$ such that
    $q=C_{\Cl X'}(y\,\ci E(u))$.
    The assumption $q\neq 0$ means that $y\,\ci E(u).z\in \Cl X'$ for
    some $z\in B^\NN$,
    and so by Remark~\ref{r:image-of-E}
    there is
    $\tilde y\in B^{\mathbb Z^-}$
    such that
    $y=\ci E(\tilde y)$.
    Then $h(C_{\Cl X}(\tilde yu))=q$. Since
    $C_{\Cl X}(\tilde yu)\in Q(\Cl X)\cdot s$, this concludes the proof.
  \end{proof}

  We are now ready to exhibit an isomorphism
  $\mathbb A_{\Cl X}\Rightarrow\mathbb A_{\Cl X'}\circ F_{\ci E}$.

 \begin{Prop}\label{p:strong-lonk-symbol-expansion}
   For each idempotent $e\in S(\Cl X)$,
   let $\eta_e$ be the function $Q(\Cl X)e\to Q(\Cl X')\,\ci E'(e)$
   such that $\eta_e(r)=h(r)$ for every $r\in Q(\Cl X)e$.
   Let $\eta=(\eta_e)_{e\in E(S(\Cl X))}$. Then
   $\eta$ is an isomorphism
   $\mathbb A_{\Cl X}\Rightarrow\mathbb A_{\Cl X'}\circ F_{\ci E}$.
 \end{Prop}

 \begin{proof}
   By Lemma~\ref{l:how-h-restricts}, the range of
   $\eta_e$ is correctly defined,
   and $\eta_e$ is bijective (as $h$ is injective.)
   On the other hand, by Lemma~\ref{l:hqs},
   the family $\eta$ is a natural transformation from
   $\mathbb A_{\Cl X}$ to $\mathbb A_{\Cl X'}\circ F_{\ci E}$.
 \end{proof}

 \begin{proof}[Conclusion of the proof of Theorems~\ref{splittingisinvflow} and~\ref{t:main-D-cate-AYF}]
   As remarked in Subsection~\ref{sec:invar-under-conj},
   the conjugacy invariance part of
   Theorem~\ref{t:main-D-cate-AYF} follows immediately from
   Theorem~\ref{t:nasu}
   and Lemma~\ref{l:biglemma}.
   The symbol expansion part is contained in
   Proposition~\ref{p:strong-lonk-symbol-expansion}. Since flow
   equivalence is generated by conjugacy and symbol expansion, we are done.
 \end{proof}

\subsection{Proof of Theorem~\ref{t:main-D-cate-AYF-Fischer-version}}\label{sec:proof-theor-reft:m-1}
We are ready to show Theorem~\ref{t:main-D-cate-AYF-Fischer-version}

 \begin{proof}[Conclusion of the proof of
   Theorem~\ref{t:main-D-cate-AYF-Fischer-version}]

  Let $\Cl X$ and $\Cl Y$ be flow equivalent
  synchronizing shifts.
  By Theorem~\ref{t:main-D-cate-AYF}, there is an equivalence
    $F\colon \KK(\Cl X)\to \KK(\Cl Y)$
    for which there is an isomorphism
    $\eta\colon\mathbb A_{\Cl X}\to \mathbb A_{\Cl Y}\circ F$.

  Since $Q_{\mathfrak F}(\Cl X)$ is an $S(\Cl X)$-invariant subset of
  $Q(\Cl X)$, it follows that $\mathbb A^{\mathfrak F}_{\Cl X}$ is a
  subfunctor of $\mathbb A_{\Cl X}$.  Similarly,
  $\mathbb A^{\mathfrak F}_{\Cl Y}$ is a subfunctor of $\mathbb A_{\Cl
    Y}$.  Hence, the result will follow immediately as along as
  $\eta_e(Q_{\mathfrak F}(\Cl X)e)= Q_{\mathfrak F}(\Cl Y)$ for all
  $e\in E(S(\Cl X))$.  But this is the content of Remark~\ref{rmk:what-psi-does-to-fischer-cover}.
\end{proof}

\section{The proof of Theorem~\ref{t:master-syntactic-invariant}}\label{sec:master}

The proof of Theorem~\ref{t:master-syntactic-invariant} relies on the
next two propositions.

\begin{Prop}\label{p:conjugacy-induced-by-morita-equivalence}
  Let $S$ and $T$ be finite semigroups with zero having local units.
  If $S$ and $T$ are Morita equivalent, then
  there are shifts $\Cl X_S$ and $\Cl X_T$,
  respectively induced by $S$ and $T$,
  such that $\Cl X_S$ and $\Cl X_T$ are conjugate.
\end{Prop}

\begin{proof}
   Since $S$ and $T$ are Morita equivalent finite
   semigroups with local units, it follows by the results of
   Lawson~\cite{Lawson:2011} that there is a semigroup $R$ containing
     $S$ and $T$ (more precisely, disjoint isomorphic copies of $S$ and $T$) such that $SRS=S$, $RSR=R$, $TRT=T$
   and $RTR=R$.
  Moreover, Lawson's proof shows that $R$ can be taken to be finite.  
  Note that since $S,T$ have local units,
  it follows that $T=TT$ and $S=SS$, and so we have
  \begin{equation}\label{eq:product-expansion-of-S-and-T}
    S=SRT\cdot TRS,\quad T=TRS\cdot SRT.
  \end{equation}
  This suggests to take the finite sets $C=S\times R\times T$,
  $D=T\times R\times S$, $A=C\times D$ and $B=D\times C$.
  Thanks to~\eqref{eq:product-expansion-of-S-and-T},
  we may consider the homomorphisms $\varphi\colon A^+\to S$
  and $\psi\colon B^+\to T$
  defined by
  \begin{align*}
    \varphi\bigr((s_1,r_1,t_1),(t_2,r_2,s_2)\bigr)&=s_1r_1t_1t_2r_2s_2,\\
    \psi\bigr((t_2,r_2,s_2),(s_1,r_1,t_1)\bigr)&=t_2r_2s_2s_1r_1t_1,
  \end{align*}
  for every $s_i\in S$, $r_i\in R$, $t_i\in T$, $i\in \{1,2\}$.
  Also by~\eqref{eq:product-expansion-of-S-and-T},
  we have $S=\varphi(A)$
  and $T=\psi(B)$, thus $\varphi$ and $\psi$ are onto.
  Hence, we may consider the shifts
  $\Cl X_S=\Cl X_\varphi\subseteq \z A$
  and $\Cl X_T=\Cl X_\psi\subseteq \z B$,
   respectively induced by $\varphi$ and $\psi$.
   The proof is completed once we show that $\Cl X_S$
   and $\Cl X_T$ are conjugate.

Consider the map $f\colon\z A\to\z B$ given by
$f((c_i,d_i)_{i\in \ZZ})=(d_i,c_{i+1})_{i\in \ZZ}$.
One sees straightforwardly that $f$ is a conjugacy (it is actually
a \emph{bipartite code}, an important special class of conjugacies
introduced by Nasu~\cite{Nasu:1986}).
Its inverse is $f^{-1}((d_i,c_i)_{i\in \ZZ})=(c_{i-1},d_{i})_{i\in \ZZ}$.
We claim that $f(\Cl X_S)=\Cl X_T$.
Denote by $0_S$ the zero of $S$, and by $0_T$ the zero of~$T$.
Let 
\begin{equation*}
x=\ldots (c_{-2},d_{-2})(c_{-1},d_{-1}).(c_0,d_0)(c_1,d_1)(c_2,d_2)\ldots 
\end{equation*}
be an element of $\z A$.
Suppose we have $f(x)\notin \Cl X_T$.
Then, there are $i,j\in\ZZ$, with $i\leq j$,
such that $\psi\bigl((d_i,c_{i+1})(d_{i+1},c_{i+2})\cdots
(d_{j},c_{j+1})\bigr)= 0_T$.
Take $s=\varphi\bigl((c_i,d_i)(c_{i+1},d_{i+1})\cdots
(c_{j},d_{j})(c_{j+1},d_{j+1})\bigr)$.
To show that $x\notin\Cl X_S$
it suffices to show $s=0_S$.
Let $c_i=(s_1,r_1,t_1)$, $d_{j}=(t_2,r_2,s_2)$.
Then, we have in $R$ the factorization
$s=s_1r_1t_10_Tt_2r_2s_2$.
Also, we have
\begin{equation}\label{eq:0s-expanded}
  0_S=s0_Ss=s_1r_1t_10_Tt_2r_2s_2\cdot 0_S\cdot
  s_1r_1t_10_Tt_2r_2s_2=s_1r_1t_10_Tz0_Tt_2r_2s_2,
\end{equation}
where $z=t_2r_2s_20_Ss_1r_1t_1$.
By~\eqref{eq:product-expansion-of-S-and-T}, we have
$z\in T$, thus $0_Tz0_T=0_T$, and
so from~\eqref{eq:0s-expanded} we get $0_S=s_1r_1t_10_Tt_2r_2s_2=s$.
This shows $x\notin\Cl X_S$,
establishing $f(\Cl X_S)\subseteq \Cl X_T$.
Similarly, one shows
$f^{-1}(\Cl X_T)\subseteq \Cl X_S$.
Therefore, we have $f(\Cl X_S)=\Cl X_T$, whence
$\Cl X_S$ and $\Cl X_T$ are conjugate.
\end{proof}

\begin{Prop}\label{p:reduction-to-shift-with-local-units}
  Let $\Cl X$ be a sofic shift. Then there is a
  sofic shift $\Cl Z$ which is flow equivalent to
  $\Cl X$ and such that $S(\Cl Z)$ has local units.
\end{Prop}

\begin{proof}
     Let $\delta\colon A^+\to S(\Cl X)$ be
     the syntactic homomorphism.
   Say that $u\in A^+$ is \emph{$\delta$-idempotent}
   if $\delta(u)$ is idempotent.
   A $\delta$-idempotent word is \emph{minimal} if it has not proper factors
   which are  $\delta$-idempotent.
   Note that every $\delta$-idempotent word has some
   minimal $\delta$-idempotent factor.

    Say also that $u\in A^+$ is \emph{$\delta$-special} if it
    has a proper prefix and proper suffix which are minimal
    $\delta$-idempotents.
    A $\delta$-special word $u$ is \emph{minimal} if
    every proper factor of $u$ is not $\delta$-special.
    In a more informal and intuitive manner, one can say
    that a  minimal $\delta$-special
    word represents two consecutive occurrences
    of minimal $\delta$-idempotents words in an element of $\z A$.

    Denote by $W$ the set of minimal $\delta$-special words of $A^+$.
    Since $S(\Cl X)$ is finite, there is
    an integer $N\geq 1$  such that every element of $A^+$ with length $N$ has
    a $\delta$-idempotent factor (cf.~\cite[Theorem 1.11]{PinBook}.)
    Therefore, every word of $A^+$ with length at least $2N$
    has at least one factor which is minimal $\delta$-special,
    and so the elements of $W$ have length at most $2N$.
    In particular, $W$ is finite.

    For $u\in W$, let $\tau(u)=(p,u,q)$ be the triple
    such that $p$ and $q$ are respectively
    the unique minimal $\delta$-idempotent prefix
    and the unique minimal $\delta$-idempotent suffix
    of $u$. Let $V$ be the set $\{\tau(u)\mid u\in W\}$, which is in
    bijection with $W$.
    For future reference, we extend the bijection $\tau$
    as follows. Let $u$ be a $\delta$-special word, not necessarily minimal.
    We define $\bar\tau(u)\in V^+$
    recursively on the number of occurrences of
    minimal $\delta$-idempotent words, as follows:
    \begin{enumerate}
    \item if $v$ is minimal $\delta$-special, then
    $\bar\tau(u)=\tau(u)$ and the recursion stops;
    \item if $u$ is not minimal $\delta$-special, then
      $u$ admits a factorization $u=vqw$ in $A^+$,
      such that $vq$ is minimal $\delta$-special and $q$ is minimal
      $\delta$-idempotent; in particular,
      $qw$ is $\delta$-special with a number
      of occurrences of
    minimal $\delta$-idempotent words
    smaller than that of $u$; then
      we let $\bar\tau(u)=\tau(uq)\bar\tau(qw)$.
    \end{enumerate}

    For $u\in W$, let $\tau(u)=(p,u,q)$.
    We have a factorization $u=zq$. Define $m(u)=|z|-1$.
    Consider a set
    $\Lambda(u)=\{\dia_{u,1},\dia_{u,2},\ldots,\dia_{u,m(u)}\}$
    of $m(u)$ elements not in $V$. We use the convention that
    $\dia_{u,0}=\tau(u)$. Assume moreover
    that if $u,v$ are distinct elements of $W$, then
    $\Lambda(u)\cap\Lambda(v)=\emptyset$. Denote by $B$
    the alphabet $V\cup\bigl(\bigcup_{u\in W}\Lambda(u)\bigr)$.

    Take $x\in \z A$.
    Denote by $P(x)$ the set of occurrences in $x$ of elements of~$W$,
    which is precisely the set of occurrences in $x$ of minimal
    $\delta$-idempotents.
    If $i\in P(x)$, denote by $w_x(i)$ the element of $W$
    occurring in $x$ at position $i$, where we say that a word $w$ \emph{occurs} at $i$
      in $x$ if $x_{[i,j]}=w$, for some $j$. 
        For each $i\in\ZZ$, let $\kappa_x(i)=\max\{j\in P(x)\mid j\leq i\}$
    and consider the difference  $d_x(i)=i-\kappa_x(i)\geq 0$.
    We may then define $w_x(i)$ for all $i\in\ZZ$
    as being $w_x(i)=w_x(\kappa_x(i))$.
    That is, if $i,j$ are consecutive occurrences of elements
    of~$P(x)$,
    then $w_x(k)=w_x(i)$ for every $k$ such that $i\leq k<j$.
    For each $i\in\ZZ$, let $\tau(w_x(i))=(p_x(i),w_x(i),q_x(i))$.
        Note that if
        $i,j$ are consecutive elements of
        $P(x)$ then $q_x(i)=p_x(j)$.

    By the maximality of $\kappa_x(i)$, the
    word $q_x(i)$ cannot occur at a position belonging to $\{j\mid \kappa_x(i)<j\leq i\}$,
    and so $d_x(i)\leq m(w_x(i))$,
    with equality $d_x(i)= m(w_x(i))$
    if and only if $i+1\in P(x)$.
    Consider the mapping
    $f\colon\z A\to \z B$
    given by $f(x)=(\dia_{w_x(i),d_x(i)})_{i\in\ZZ}$,
    $x\in\z A$.
    More intuitively, $f$ is characterized by the following property:
    if $i,i+n$ are two consecutive elements of $P(x)$ (and so
    $m(w_x(i))=n-1$), then the sequence
    \begin{equation*}
      (p_x(i),w_x(i),q_x(i))
      \dia_{w_x(i),1}\cdots
      \dia_{w_x(i),n-1}
      (q_x(i),w_x(i+n),q_x(i+n))
    \end{equation*}
    is an word of $B^+$ at position $i$ in $f(x)$.

    For each $i\in\ZZ$, if  $d_x(i)<m(w_x(i))$
    then $\kappa_{\sigma(x)}(i)=\kappa_x(i)-1$
    and if $d_x(i)=m(w_x(i))$
    then $\kappa_{\sigma(x)}(i)=i$.
    In both cases, the word $w_{\sigma(x)}(i)$ is equal to
    $w_x(i+1)$. This amounts to say that $f$ commutes with the shift mapping.
    Observe also that $f(x)_i$ is determined by $x_{[i-N,i+N]}$,
     and so $f$ is continuous, whence a morphism of shifts.

     Let $x,y\in \z A$, and let $i\in\ZZ$.
          We have $f(x)_i=\dia_{u,d_x(i)}$
          and $f(y)_i=\dia_{v,d_x(i)}$, for some $u,v\in W$.
          Moreover,
          $x_i$ is the letter at position $d_x(i)+1$
          in $u$,
          and $y_i$ is the letter at position $d_y(i)+1$
          in $v$. Therefore, we have $x_i=y_i$ when
          $\dia_{u,d_x(i)}=\dia_{v,d_y(i)}$, as the latter
          implies $u=v$ and $d_x(i)=d_y(i)$.
          Hence, $x_i\neq y_i$ implies $f(x)_i\neq f(y)_i$, showing
          that $f$ is injective.
     We conclude that $f$ induces a conjugacy between $\Cl X$ and
     the shift $\Cl Y=f(\Cl X)$.

     Take $u\in W$ and $i\in \{1,\ldots,m(u)\}$.
     In an element of $\Cl Y$,
     an occurrence of $\diamond_{u,i}$ is always preceeded by
     an occurrence of $\diamond_{u,i-1}$.
     Hence, starting in $\Cl Y$, we may perform a
     sequence $C$ of symbol contractions
     $\diamond_{u,i-1}\diamond_{u,i}\mapsto \diamond_{u,i-1}$,
     with $i$ running $\{1,\ldots,m(u)\}$
     and $u$ running $W$, to obtain
     a subshift $\Cl Z$ of $\z V$
     which is flow equivalent to $\Cl X$.
     We denote by $f^C$ the mapping $\Cl X\to\Cl Z$
     assigning to each $x\in\Cl X$
     the element of
     $\Cl Z$ obtained from $f(x)$ by applying the sequence $C$ of symbol contractions.

     Let $\Cl X_0=\{x\in\Cl X\mid 0\in P(x)\}$.
     Then every element of $\Cl Z$ is in the orbit of an element of
     $f^C(\Cl X_0)$.
      Let $\nu=(p_0,p_0t_0,p_1)$ be an arbitrary element of~$V$.
      Suppose there is $x\in\Cl X_0$ such that $f^C(x)$ equals
     \begin{equation}\label{eq:of-the-form-contracted}
       \ldots
      (p_{-2},z_{-2}p_{-1},p_{-1})
      (p_{-1},z_{-1}p_0,p_{0}).(p_0,p_0t_0,p_1)
      (p_1,p_1t_1,p_2)
      \ldots.
     \end{equation}
     The sequence $x$ is completely determined
     by~\eqref{eq:of-the-form-contracted}. More precisely, $x$ equals
     \begin{equation}\label{eq:of-the-form-original}
       \ldots z_{-2}z_{-1}.p_0t_0t_1t_2\ldots.
     \end{equation}
     Since $\delta(p_0)$ is idempotent,
     it follows
     that the sequence
     \begin{equation}\label{eq:original-expanded-by-p0}
     \ldots z_{-2}z_{-1}.p_0p_0t_0t_1t_2\ldots
     \end{equation}
     also belongs to $\Cl X_0$.
     Its image under $f^C$ is
     \begin{equation}\label{eq:expanded-by-p0}
              \ldots
      (p_{-2},z_{-2}p_{-1},p_{-1})
      (p_{-1},z_{-1}p_0,p_{0}).\bar\tau(p_0p_0)(p_0,p_0t_0,p_1)
      (p_1,p_1t_1,p_2)
      \ldots.
     \end{equation}
     Comparing~\eqref{eq:of-the-form-contracted}
     and~\eqref{eq:expanded-by-p0}, we conclude that
     the context $[\nu]_{\Cl Z}$
     is contained in the context
     $[\bar\tau(p_0p_0)\nu]_{\Cl Z}$.
     Let us verify that the reverse inclusion also holds.
     If the sequence~\eqref{eq:expanded-by-p0}
     belongs to $\Cl Z$, then
     the sequence~\eqref{eq:original-expanded-by-p0}
     belongs to $\Cl X_0$, and therefore so
     does~\eqref{eq:of-the-form-original} by the hypothesis that
     $\delta(p_0)$
     is idempotent. We then deduce that the
     sequence~\eqref{eq:of-the-form-contracted}
     belongs to $\Cl Z$, showing the desired inclusion of contexts.
     Hence, we have
     $[\nu]_{\Cl Z}=[\bar\tau(p_0p_0)\nu]_{\Cl Z}$,
     thus $[\nu]_{\Cl Z}=s^n\cdot [\nu]_{\Cl Z}$,
     for every $n\geq 1$,
     where $s=[\bar\tau(p_0p_0)]_{\Cl Z}$.
     As $S(\Cl Z)$ is finite, there is $m\geq 1$ such that $s^m=e$
     is idempotent (cf.~\cite[Proposition 1.6]{PinBook}),
     thus $[\nu]_{\Cl  Z}=e\cdot [\nu]_{\Cl  Z}$.
     Similarly, one can show that $[\nu]_{\Cl  Z}=[\nu\bar\tau(p_1p_1)]_{\Cl Z}$
     and that therefore we have $[\nu]_{\Cl  Z}=[\nu]_{\Cl  Z}\cdot e'$ for some idempotent $e'$.
     Therefore, $[\nu]_{\Cl  Z}$ has local
     units.
     As the elements of the form $[\nu]_{\Cl  Z}$, $\nu\in V$, generate
     $S(\Cl Z)$, we conclude that $S(\Cl Z)$ has local units.
\end{proof}

We are now ready to prove Theorem~\ref{t:master-syntactic-invariant}.

\begin{proof}[Proof of Theorem~\ref{t:master-syntactic-invariant}]
  By Proposition~\ref{p:reduction-to-shift-with-local-units},
  there are sofic shifts $\widetilde{\Cl X}$ and $\widetilde{\Cl Y}$,
  respectively flow equivalent to $\Cl X$
  and $\Cl Y$,
  such that both $S(\widetilde{\Cl X})$
  and $S(\widetilde{\Cl Y})$
  have local units.
  Since $\vartheta$ is an invariant of flow equivalence,
  we have
  $\Cl X\mathrel{\vartheta}\widetilde{\Cl X}$
  and 
  $\Cl Y\mathrel{\vartheta}\widetilde{\Cl Y}$.
  Since the equivalence class of the Karoubi envelope is
  a flow invariant (Theorem~\ref{splittingisinvflow}), we
  know that $\KK(\Cl X)$ is equivalent to $\KK(\widetilde{\Cl X})$
  and $\KK(\Cl Y)$ is equivalent to $\KK(\widetilde{\Cl Y})$.
  As $\KK(\Cl X)$ is by hypothesis equivalent to $\KK(\Cl Y)$,
  we conclude that $\KK(\widetilde{\Cl X})$
  is equivalent to  $\KK(\widetilde{\Cl Y})$,
  that is,
  $S(\widetilde{\Cl X})$
  and $S(\widetilde{\Cl Y})$
  are Morita equivalent.
  By
  Proposition~\ref{p:conjugacy-induced-by-morita-equivalence},
  we know there are conjugate shifts $\Cl X^\sharp$
  and $\Cl Y^\sharp$
  respectively induced by
  $S(\widetilde{\Cl X})$
  and $S(\widetilde{\Cl Y})$.
  Again by the definition of flow invariant of sofic shifts,
  we have $\Cl X^\sharp\mathrel{\vartheta} \Cl Y^\sharp$.
  But by Lemma~\ref{l:0-disjunctive-are-syntactic-semigroups},
  we have $S(\widetilde{\Cl X})=S(\Cl X^\sharp)$
  and $S(\widetilde{\Cl Y})=S(\Cl Y^\sharp)$. So from the definition
  of syntactic invariant of sofic shifts we have that $\Cl
  X\mathrel{\vartheta} \Cl X^\sharp$ and  $\Cl Y\mathrel{\vartheta}
  \Cl Y^\sharp$.  Thus $\Cl X\mathrel{\vartheta} \Cl Y$, as required.
\end{proof}

\appendix

\section{Symbol expansion and subsynchronizing subshifts}\label{sec:symb-expans-subsynch}

Now that we have the tools developed in
Subsection~\ref{sec:invar-under-flow}, we are able to show the proposition needed
to conclude the proof of Theorem~\ref{t:flow-inv-sub-sync}.

  \begin{Prop}\label{p:symbol-expa-of-subsync}
     Let $\Cl X$ be a sofic shift, and let $e$ be a magic idempotent of
    $S(\Cl X)$.
    Consider a symbol expansion $\Cl X'$ of $\Cl X$, defined
    by a symbol expansion homomorphism $\mathcal E$.
    If $e$ is a magic idempotent for $\Cl X$, then
    the subsynchronizing subshift $S(F_{\mathcal E}(e))$ of $\Cl X'$
    is the symbol expansion of $S(e)$ defined by $\mathcal E$.
\end{Prop}

\begin{proof}
    Let $u\in L(\Cl X)$ be such that $e=[u]_\Cl X$.
    
  Consider an element of $L(S(e)')$ of the form $\mathcal E(v)$, with
  $v\in L(S(e))$.
  Then $uwv\in L(\Cl X)$ for $w$, thus $\mathcal E(uwv)\in L(\Cl X')$.
  Since $F_{\mathcal E}(e)=[\mathcal E(u)]_{\Cl X'}$, this shows
  that
  $\mathcal E(v)$ is a finite block of $S(F_{\mathcal E}(e))$.
  Every finite block of $S(e)'$ is a factor of a word such as
  $\mathcal E(v)$. Therefore, we proved that $S(e)'\subseteq S(F_{\mathcal E}(e))$.

  Conversely, let $v$ be a finite block of $S(F_{\mathcal E}(e))$.
  Then $\mathcal E(u)wvw'\mathcal E(u')\in L(\Cl X')$ for some $w,w',u'$, with $u'$ a non-empty word over the alphabet of $\Cl X$.
  By Remark~\ref{r:image-of-E},
  $\mathcal E(u)wvw'\mathcal E(u')$
  and $wvw'$ are in $\mathrm{Im}\mathcal E$, thus
  $u\,\mathcal E^{-1}(wvw')u'\in L(\Cl X)$. This establishes $\mathcal E^{-1}(wvw')\in S(e)$.
  Hence, $wvw'$ belongs to $L(S(e)')$, and hence so does its
  factor $v$. This shows that
  $S(F_{\mathcal E}(e))\subseteq S(e)'$.
\end{proof}

\bibliographystyle{abbrv}

\begin{thebibliography}{10}

\bibitem{AshHall}
C.~J. Ash and T.~E. Hall.
\newblock Inverse semigroups on graphs.
\newblock {\em Semigroup Forum}, 11(2):140--145, 1975/76.

\bibitem{Bates&Eilers&Pask:2011}
T.~Bates, S.~Eilers, and D.~Pask.
\newblock Reducibility of covers of {AFT} shifts.
\newblock {\em Israel J. Math.}, 185:207--234, 2011.

\bibitem{Beal:1993}
M.-P. B\'eal.
\newblock {\em Codage Symbolique}.
\newblock Masson, 1993.

\bibitem{Beal&Berstel&Eilers&Perrin:2010arxiv}
M.~P. B\'eal, J.~Berstel, S.~Eilers, and D.~Perrin.
\newblock Symbolic dynamics.
\newblock arXiv:1006.1265v3 [cs.FL].

\bibitem{Beal&Fiorenzi&Perrin:2005b}
M.-P. B\'eal, F.~Fiorenzi, and D.~Perrin.
\newblock A hierarchy of shift equivalent sofic shifts.
\newblock {\em Theoret. Comput. Sci.}, 345:190--205, 2005.

\bibitem{Beal&Fiorenzi&Perrin:2005a}
M.-P. B\'eal, F.~Fiorenzi, and D.~Perrin.
\newblock The syntactic graph of a sofic shift is invariant under shift
  equivalence.
\newblock {\em Int. J. Algebra Comput.}, 16(3):443--460, 2006.

\bibitem{Beauquier:1985}
D.~Beauquier.
\newblock Minimal automaton for a factorial transitive rational language.
\newblock {\em Theoret. Comput. Sci.}, 67:65--73, 1985.

\bibitem{Blanchard&Hansel:1986}
F.~Blanchard and G.~Hansel.
\newblock Syst\`emes cod\'es.
\newblock {\em Theoret. Comput. Sci.}, 44:17--49, 1986.

\bibitem{Boyle:2002}
M.~Boyle.
\newblock Flow equivalence of shifts of finite type via positive
  factorizations.
\newblock {\em Pacific J. Math.}, 204(2):273--317, 2002.

\bibitem{Boyle&Carlsen&Eilers}
M.~Boyle, T.~M. Carlsen, and S.~Eilers.
\newblock Flow equivalence of sofic shifts.
\newblock In preparation.

\bibitem{Clifford&Preston:1961}
A.~H. Clifford and G.~B. Preston.
\newblock {\em The Algebraic Theory of Semigroups}, volume~I.
\newblock Amer. Math. Soc., Providence, R.I., 1961.

\bibitem{Costa:2006}
A.~Costa.
\newblock Conjugacy invariants of subshifts: an approach from profinite
  semigroup theory.
\newblock {\em Int. J. Algebra Comput.}, 16(4):629--655, 2006.

\bibitem{Costa:2006b}
A.~Costa.
\newblock Pseudovarieties defining classes of sofic subshifts closed under
  taking shift equivalent subshifts.
\newblock {\em J. Pure Appl. Algebra}, 209:517--530, 2007.

\bibitem{Costa:2007}
A.~Costa.
\newblock {\em Semigrupos Profinitos e Din\^amica Simb\'olica}.
\newblock PhD thesis, Faculdade de Ci\^encias da Universidade do Porto, 2007.

\bibitem{newpaper}
A.~Costa and B.~Steinberg.
\newblock The {S}ch\"utzenberger category of a semigroup.
\newblock In preparation.

\bibitem{Costa&Steinberg:2011}
A.~Costa and B.~Steinberg.
\newblock Profinite groups associated to sofic shifts are free.
\newblock {\em Proc. London Math. Soc.}, 102:341--369, 2011.

\bibitem{ACosta&Steinberg:2013}
A.~Costa and B.~Steinberg.
\newblock A categorical invariant of flow equivalence of shifts.
\newblock Technical report, 2013.
\newblock arXiv:1304.3487 [math.DS].

\bibitem{Delgado&Linton&Morais:automata}
M.~Delgado, S.~Linton, and J.~Morais.
\newblock {\em {Automata: A GAP package on finite automata}}.
\newblock http://www.gap-system.org/Packages/automata.html.

\bibitem{Delgado&Morais:sgpviz}
M.~Delgado and J.~Morais.
\newblock {\em {SgpViz: A GAP package to visualize finite semigroups}}, 2008.
\newblock http://www.gap-system.org/Packages/sgpviz.html.

\bibitem{Eilenberg:1976}
S.~Eilenberg.
\newblock {\em Automata, Languages and Machines}, volume~B.
\newblock Academic Press, New York, 1976.

\bibitem{Fiebig&Fiebig:1991}
D.~Fiebig and U.-R. Fiebig.
\newblock Covers for coded systems.
\newblock In {\em Symbolic dynamics and its applications ({N}ew {H}aven, {CT},
  1991)}, volume 135 of {\em Contemp. Math.}, pages 139--179. Amer. Math. Soc.,
  Providence, RI, 1992.

\bibitem{Fischer:1975}
R.~Fischer.
\newblock Sofic systems and graphs.
\newblock {\em Monatsh. Math.}, 80:179--186, 1975.

\bibitem{Franks:1984}
J.~Franks.
\newblock Flow equivalence of subshifts of finite type.
\newblock {\em Ergodic Theory Dynam. Systems}, 4(1):53--66, 1984.

\bibitem{Fujiwara&Osikawa:1987}
M.~Fujiwara and M.~Osikawa.
\newblock Sofic systems and flow equivalence.
\newblock {\em Math. Rep. Kyushu Univ.}, 16(1):17--27, 1987.

\bibitem{FunkLawsonSteinberg}
J.~Funk, M.~V. Lawson, and B.~Steinberg.
\newblock Characterizations of {M}orita equivalent inverse semigroups.
\newblock {\em J. Pure Appl. Algebra}, 215(9):2262--2279, 2011.

\bibitem{GAP4:2006}
The GAP~Group.
\newblock {\em {GAP -- Groups, Algorithms, and Programming, Version 4.4}},
  2006.
\newblock \verb+(http://www.gap-system.org)+.

\bibitem{Hamachi&Inoue&Krieger:2009}
T.~Hamachi, K.~Inoue, and W.~Krieger.
\newblock Subsystems of finite type and semigroup invariants of subshifts.
\newblock {\em J. Reine Angew. Math.}, 632:37--61, 2009.

\bibitem{Hamachi&Krieger:2013}
T.~Hamachi and W.~Krieger.
\newblock A construction of subshifts and a class of semigroups, 2013.
\newblock arXiv:1303.4158v1 [math.DS].

\bibitem{Hamachi&Krieger:2013b}
T.~Hamachi and W.~Krieger.
\newblock On certain subshifts and their associated monoids, 2013.
\newblock arXiv:1202.5207v2 [math.DS].

\bibitem{Huang:2001}
D.~Huang.
\newblock Automorphisms of {B}owen-{F}ranks groups of shifts of finite type.
\newblock {\em Ergodic Theory Dynam. Systems}, 21, 2001.

\bibitem{Johansen:2011b}
R.~Johansen.
\newblock {\em On flow equivalence of sofic shifts}.
\newblock PhD thesis, University of Copenhagen, 2011.

\bibitem{Jones&Lawson:2014}
D.~G. Jones and M.~Lawson.
\newblock Graph inverse semigroups: their characterization and completion.
\newblock {\em Journal of Algebra}, 409:444--473, 2014.

\bibitem{Jonoska:1996a}
N.~Jonoska.
\newblock Sofic systems with synchronizing representations.
\newblock {\em Theoret. Comput. Sci.}, (158):81--115, 1996.

\bibitem{Jonoska:1998}
N.~Jonoska.
\newblock A conjugacy invariant for reducible sofic shifts and its semigroup
  characterizations.
\newblock {\em Israel J. Math.}, (106):221--249, 1998.

\bibitem{Kim&Roush:1990}
K.~H. Kim and F.~W. Roush.
\newblock An algorithm for sofic shift equivalence.
\newblock {\em Ergodic Theory Dynam. Systems}, 10:381--393, 1990.

\bibitem{Kim&Roush:1992}
K.~H. Kim and F.~W. Roush.
\newblock Williams conjecture is false for reducible subshifts.
\newblock {\em J. Amer. Math. Soc.}, 5:213--215, 1992.

\bibitem{Kim&Roush:1999}
K.~H. Kim and F.~W. Roush.
\newblock The {W}illiams conjecture is false for irreducible subshifts.
\newblock {\em Ann. of Math.}, 149(2):545--558, 1999.

\bibitem{Krieger:1974}
W.~Krieger.
\newblock On the uniqueness of the equilibrium state.
\newblock {\em Math. Systems Theory}, 8(2):97--104, 1974/1975.

\bibitem{Krieger:1984}
W.~Krieger.
\newblock On sofic systems {I}.
\newblock {\em Israel J. Math.}, 48:305--330, 1984.

\bibitem{Krieger:2000}
W.~Krieger.
\newblock On a syntactically defined invariant of symbolic dynamics.
\newblock {\em Ergodic Theory Dynam. Systems}, 20:501--516, 2000.

\bibitem{Kriger:2006}
W.~Krieger.
\newblock On subshifts and semigroups.
\newblock {\em Bull. London Math. Soc.}, 38(4):617--624, 2006.

\bibitem{Krieger:2012}
W.~Krieger.
\newblock On subshift presentations, 2012.
\newblock arXiv:1209.2578v1 [math.DS].

\bibitem{Krieger&Matsumoto:2011}
W.~Krieger and K.~Matsumoto.
\newblock Zeta functions and topological entropy of the {M}arkov-{D}yck shifts.
\newblock {\em M\"unster J. of Math.}, 4:171--184, 2011.

\bibitem{Lallement:1979}
G.~Lallement.
\newblock {\em Semigroups and Combinatorial Applications}.
\newblock Wiley, New York, 1979.

\bibitem{Lawson:2011}
M.~V. Lawson.
\newblock Morita equivalence of semigroups with local units.
\newblock {\em J. Pure Appl. Algebra}, 215(4):455--470, 2011.

\bibitem{MarcusandLind}
D.~Lind and B.~Marcus.
\newblock {\em An introduction to symbolic dynamics and coding}.
\newblock Cambridge University Press, Cambridge, 1995.

\bibitem{MacLane:1998}
S.~MacLane.
\newblock {\em Categories for the Working Mathematician}.
\newblock Number~5 in Grad. Texts in Math. Springer-Verlag, New York, 2nd
  edition, 1998.

\bibitem{Margolis&Steinberg:2012}
S.~Margolis and B.~Steinberg.
\newblock Quivers of monoids with basic algebras.
\newblock {\em Compositio Mathematica}, 148(5):1516--1560, 2012.

\bibitem{Matsumoto:2001}
K.~Matsumoto.
\newblock Bowen-{F}ranks groups as an invariant for flow equivalence of
  subshifts.
\newblock {\em Ergodic Theory Dynam. Systems}, 21(6):1831--1842, 2001.

\bibitem{Matsumoto:2011}
K.~Matsumoto.
\newblock A certain synchronizing property of subshifts and flow equivalence,
  2011.
\newblock arXiv:1105.3249v1 [math.DS].

\bibitem{Nasu:1986}
M.~Nasu.
\newblock Topological conjugacy for sofic systems.
\newblock {\em Ergodic Theory Dynamic. Systems}, 6:265--280, 1986.

\bibitem{Parry&Sullivan:1975}
B.~Parry and D.~Sullivan.
\newblock A topological invariant of flows on {$1$}-dimensional spaces.
\newblock {\em Topology}, 14(4):297--299, 1975.

\bibitem{Parry&Tuncel:1982}
W.~Parry and S.~Tuncel.
\newblock {\em Classification problems in ergodic theory}, volume~67 of {\em
  London Mathematical Society Lecture Note Series}.
\newblock Cambridge University Press, Cambridge, 1982.
\newblock Statistics: Textbooks and Monographs, 41.

\bibitem{Paterson}
A.~L.~T. Paterson.
\newblock {\em Groupoids, inverse semigroups, and their operator algebras},
  volume 170 of {\em Progress in Mathematics}.
\newblock Birkh{\"a}user Boston Inc., Boston, MA, 1999.

\bibitem{graphinverse}
A.~L.~T. Paterson.
\newblock Graph inverse semigroups, groupoids and their {$C\sp \ast$}-algebras.
\newblock {\em J. Operator Theory}, 48(3, suppl.):645--662, 2002.

\bibitem{Petrich:1984}
M.~Petrich.
\newblock {\em Inverse Semigroups}.
\newblock Wiley, New York, 1984.

\bibitem{PinBook}
J.-E. Pin.
\newblock {\em Varieties of formal languages}.
\newblock Foundations of Computer Science. Plenum Publishing Corp., New York,
  1986.
\newblock With a preface by M.-P. Sch{\"u}tzenberger, Translated from the
  French by A. Howie.

\bibitem{Rhodes&Steinberg:2009}
J.~Rhodes and B.~Steinberg.
\newblock {\em The {$q$}-theory of finite semigroups}.
\newblock Springer Monographs in Mathematics. Springer, New York, 2009.

\bibitem{strongmorita}
B.~Steinberg.
\newblock Strong {M}orita equivalence of inverse semigroups.
\newblock {\em Houston J. Math.}, 37(3):895--927, 2011.

\bibitem{Talwar3}
S.~Talwar.
\newblock Morita equivalence for semigroups.
\newblock {\em J. Austral. Math. Soc. Ser. A}, 59(1):81--111, 1995.

\bibitem{Tilson:1987}
B.~Tilson.
\newblock Categories as algebra: an essential ingredient in the theory of
  monoids.
\newblock {\em J. Pure Appl. Algebra}, 48:83--198, 1987.

\end{thebibliography}

\def\malce{\mathbin{\hbox{$\bigcirc$\rlap{\kern-7.75pt\raise0,50pt\hbox{${\tt
  m}$}}}}}\def\cprime{$'$} \def\cprime{$'$} \def\cprime{$'$} \def\cprime{$'$}
  \def\cprime{$'$} \def\cprime{$'$} \def\cprime{$'$} \def\cprime{$'$}
  \def\cprime{$'$}

\end{document}